\pdfoutput=1
\RequirePackage{ifpdf}
\ifpdf 
\documentclass[pdftex]{sigma}
\else
\documentclass{sigma}
\fi

\usepackage{aligned-overset}

\numberwithin{equation}{section}

\newtheorem{Theorem}{Theorem}[section]
\newtheorem*{Theorem*}{Theorem}
\newtheorem{Corollary}[Theorem]{Corollary}
\newtheorem{Lemma}[Theorem]{Lemma}
\newtheorem{Proposition}[Theorem]{Proposition}

\theoremstyle{definition}
\newtheorem{Definition}[Theorem]{Definition}

\newtheorem{Remark}[Theorem]{Remark}

\usepackage{mathtools}
\usepackage{enumitem}
\usepackage{multirow}
\usepackage{float}
\usepackage{tabularx}
\usepackage{xltabular}

\DeclareMathOperator{\diag}{diag}

\DeclareMathOperator{\Id}{Id}
\DeclareMathOperator{\sVect}{sVect}
\DeclareMathOperator{\Vect}{Vect}
\DeclareMathOperator{\Hom}{Hom}
\DeclareMathOperator{\Ad}{Ad}
\DeclareMathOperator{\Coad}{Coad}

\DeclareMathOperator{\st}{st}

\newcommand\AddLabel[1]{%
 \refstepcounter{equation}
 (\theequation)
 \label{#1}
}
\newcommand{\bbar}[1]{\bar{\mathbf{#1}}}
\newcommand{\iso}{ \vphantom{j^{X^2}}\smash{\overset{\sim}{\vphantom{\rule{0pt}{0.20em}}\smash{\longrightarrow}}} }
\newcommand{\sfrac}[2]{{\textstyle\frac{#1}{#2}}}

\newcommand{\BZ}{\mathbb{Z}}
\newcommand{\BC}{\mathbb{C}}
\newcommand{\BK}{\mathbb{K}}

\newcommand{\fg}{\mathfrak{g}}
\newcommand{\fh}{\mathfrak{h}}
\newcommand{\gl}{\mathfrak{gl}}
\newcommand{\ssl}{\mathfrak{sl}}
\newcommand{\fosp}{\mathfrak{osp}}
\newcommand{\fgosp}{\mathfrak{gosp}}
\newcommand{\fR}{\mathfrak{R}}

\newcommand{\CA}{\mathcal{A}}
\newcommand{\CC}{\mathcal{C}}
\renewcommand{\CD}{\mathcal{D}}
\newcommand{\CF}{\mathcal{F}}
\newcommand{\CI}{\mathcal{I}}
\newcommand{\CJ}{\mathcal{J}}
\newcommand{\CT}{\mathcal{T}}

\newcommand{\sL}{\mathsf{L}}
\newcommand{\sse}{\mathsf{e}}
\newcommand{\ssf}{\mathsf{f}}
\newcommand{\ssh}{\mathsf{h}}
\newcommand{\ssH}{\mathsf{H}}

\newcommand{\bfe}{\mathbf{e}}
\newcommand{\bff}{\mathbf{f}}
\newcommand{\bfH}{\mathbf{H}}

\newcommand{\uqgl}{U_{q}(\gl(V))}
\newcommand{\uqsl}{U_{q}(\ssl(V))}
\newcommand{\URtw}{\widetilde{U}(R)}
\newcommand{\uhgl}{U_{\hbar}(\gl(V))}
\newcommand{\uqV}{U_{q}(\fosp(V))}
\newcommand{\uqVe}{U_{q}(\fgosp(V))}

\newcommand{\DrJ}{\mathrm{DJ}}
\newcommand{\DF}{\mathrm{DF}}
\newcommand{\RLL}{\textrm{RLL}}
\newcommand{\End}{\mathrm{End}}
\newcommand{\opp}{\mathrm{op}}
\newcommand{\copp}{\mathrm{cop}}
\newcommand{\ID}{\mathrm{I}}
\newcommand{\rl}{\mathrm{l}}

\begin{document}

\allowdisplaybreaks

\newcommand{\arXivNumber}{2604.21785}

\renewcommand{\PaperNumber}{105}

\FirstPageHeading

\ShortArticleName{Orthosymplectic Quantum Groups Revisited}

\ArticleName{Orthosymplectic Quantum Groups Revisited}

\Author{Kyungtak HONG and Alexander TSYMBALIUK}

\AuthorNameForHeading{K.~Hong and A.~Tsymbaliuk}

\Address{Purdue University, Department of Mathematics, West Lafayette, IN 47907, USA}
\Email{\mail{hong420@purdue.edu}, \mail{sashikts@gmail.com}}

\ArticleDates{Received April 29, 2026, in final form October 05, 2026; Published online October 30, 2026}

\Abstract{We present the $\textrm{RLL}$-realization of extended orthosymplectic quantum supergroups for any parity sequence,
with $R$-matrices evaluated in the earlier work~[\textit{Lett. Math. Phys.} \textbf{116} (2026), 55, 97~pages, arXiv:2408.16720]. Our isomorphism is compatible with the internal structure of generalized doubles. We also relate different sign conventions through $2$-cocycle twists.}


\Classification{17B37; 81R10}

\section{Introduction}

\subsection{Summary}

For simple Lie algebras $\fg$ of classical type, the quantum groups $U^{\RLL}_q(\fg)$ first
implicitly appeared in the work of Faddeev's school on the \emph{quantum inverse scattering method}, see, e.g., \cite{frt}.
In this \emph{$\RLL$-realization}, the algebra generators are encoded by two square matrices $L^\pm$ subject to
the famous \emph{$\RLL$-relations}
\[
 RL^\pm_1L^\pm_2=L^\pm_2L^\pm_1R , \qquad RL^+_1L^-_2=L^-_2L^+_1R
\]
(along with central reduction), where $R$ is a~solution of the \emph{Yang--Baxter equation}
\[
 R_{12}R_{13}R_{23} = R_{23}R_{13}R_{12}.
\]
This is a~natural analogue of the matrix realization of classical
Lie algebras, and it manifestly exhibits the Hopf algebra structure, with the coproduct $\Delta$, antipode $S$,
and counit $\epsilon$ given by
\begin{gather*}
 \Delta\bigl(L^\pm\bigr)=L^\pm\otimes L^\pm, \qquad S\bigl(L^\pm\bigr)=\bigl(L^\pm\bigr)^{-1}, \qquad \epsilon\bigl(L^\pm\bigr)=\ID.
\end{gather*}
The uniform definition of quantum groups $U^{\DrJ}_q(\fg)$ for simple $\fg$ was provided
by Drinfeld~\cite{d0} and Jimbo~\cite{j0}, and is usually referred to as the \emph{Drinfeld--Jimbo realization}.
In this presentation, the generators $e_i$, $f_i$, $k^{\pm 1}_i=q^{\pm h_i}$ are labeled by simple roots $\alpha_i$ of $\fg$,
while the Hopf algebra structure is given formally by the assignment on the generators. In $A$-type, the corresponding
isomorphism
\begin{gather}\label{eq:DF-iso}
 U^{\RLL}_q(\gl_n)\simeq U^{\DrJ}_q(\gl_n),
\end{gather}
and subsequently its $\ssl_n$-version (as well as their affine counterparts), were constructed
in~\cite{df} by considering the Gauss decomposition of $L^\pm$.
The generalization of this result for other simple Lie algebras of classical type goes back
to~\cite{hm} and more recently to~\cite{jlm1, jlm2}.

The theory of quantum supergroups associated with Lie superalgebras is still far from a~full development.
One of the major deficiencies lies in the absence of the uniform (new) Drinfeld realization of such algebras in affine types
(except for types $A$ and $D(2,1,x)$), which is crucial for the study of their representation theory and further applications
to mathematical physics. A~novel feature of Lie superalgebras is that they admit several non-isomorphic Dynkin diagrams,
which can be labeled by so-called parity sequences.
The isomorphism of the Lie superalgebras corresponding to different Dynkin diagrams of the same finite/affine type was established
in~\cite[Appendix]{lss}. Upgrading the latter to the isomorphism of quantum finite/affine superalgebras associated
with different Dynkin diagrams is a~highly non-trivial technical task that constitutes one of the major results of~\cite{y}.
We also note that the super setup features so-called \emph{higher-order} Serre relations, which often cause major
technical complications in generalizing non-super results to the super setup.
Furthermore, different parity sequences result in fundamentally different sets of higher-order Serre
relations entering the definition of the corresponding quantum supergroup.\looseness=1

This note is devoted to the generalization of the aforementioned isomorphism~\eqref{eq:DF-iso} to the context of
quantum supergroups of finite $ABCD$-type. We want to emphasize right away that this is a~Hopf superalgebra
isomorphism. To this end, it is fundamental that both algebras admit \emph{generalized double} realizations in the context of a~pair
of Hopf superalgebras endowed with skew pairing (we present full details in Appendix~\ref{sec:double-construction} as we could not find a~uniform treatment of that
construction in the super setup, besides a~somewhat related paper~\cite{bgz}). As the pairing between the positive and negative
Borel subalgebras of Drinfeld--Jimbo quantum supergroups is known to be non-degenerate (which is actually the key result of~\cite{y0} and
is the source of all higher-order Serre relations), once a~surjective map $U^{\DrJ}_q(\fg) \to U^{\RLL}_q(\fg)$ is constructed,
compatible with coproducts and pairing, it must be an isomorphism. This approach is different from the one used in~\cite{df} to
prove~\eqref{eq:DF-iso} as well as the one utilized in~\cite{jlm1,jlm2} to establish $BCD$-type generalizations, where one rather
had to appeal to faithful representations (which is more subtle in the super case) or to the universal $R$-matrix $\fR$
(that was carried out in the $\hbar$-adic completion),\footnote{This reliance on the $\hbar$-adic topology
can be completely bypassed by following the arguments developed in~\cite[Section~3]{rw}.} respectively.
We note that evaluation of the $R$-matrices used in the $\RLL$-realization was carried out in our earlier paper~\cite{ht1} using the combinatorial approach of~\cite{chw} to the construction of PBW bases of~$U^{\DrJ}_q(\fg)$.

As another technical aspect of the super version of quantum groups, we note that there have been multiple sign-conventions in the literature. In particular, in pursuit to generalize the foundational work~\cite{df} to super-$A$ type (both finite and affine), the works~\cite{cwwz,fhs,yz} used slightly different definitions of the corresponding $\RLL$-realizations. We show that different sign-conventions can be related through the more general construction of $2$-cocycle twists. This gives rise to the twisted and untwisted versions, see details in Section~\ref{ssec:sign-twists}. Though the untwisted algebra is the correct object to work with, it is often more convenient to use the twisted one (in particular, when studying the Gauss decomposition) in order to simplify
numerous signs.

\subsection{Outline}

The structure of the present paper is the following:
\begin{itemize}\itemsep=0pt
\item In Section~\ref{sec:gl-DJ}, we recall the basic conventions on superalgebras and the notion of Drinfeld--Jimbo
general linear quantum supergroups, see Definitions~\ref{def:quantum-gl-lattice} and~\ref{def:quantum-gl-standard}.
We emphasize the key property of them being Drinfeld doubles in Proposition~\ref{prop:DJ-pairing_finite} and
introduce their twisted version $\uqgl^{\CF}$ in Section~\ref{ssec:drinfeld-twist}.

\item
In Section~\ref{sec:RTT_finite}, we recall the $\RLL$-realization of the general linear quantum supergroups, see
Definition~\ref{def:rll-algebra}, with the $R$-matrix given in~\eqref{eq:gl-evaluated-R}. The key property of them being
generalized doubles is established in Proposition~\ref{prop:RTT-pairing_finite}. In Section~\ref{ssec:sign-twists}, we
discuss a~closely related algebra $\URtw$ which is often more convenient to work with in order to track various Koszul signs,
and establish its relation to $U(R)$ via the 2-cocycle twist in Proposition~\ref{prop:gl-twisted-rtt} (this addresses a~common discrepancy in various definitions of $\RLL$-superalgebras in the existing literature). One of the key results
is the identification of Drinfeld--Jimbo and $\RLL$ realizations of the general linear quantum supergroups, established in
Theorem~\ref{thm:finite DrJ to RTT}. While the inverse morphism is constructed in
Section~\ref{ssec:inverse morphism}, we wish to emphasize that the map $\xi$ of Theorem~\ref{thm:finite DrJ to RTT} is
compatible with Drinfeld double structures and is thus easily argued to be an isomorphism.

\item
In Section~\ref{sec:osp-DJ}, we recall the basic conventions about orthosymplectic Lie superalgebra $\fosp(V)$ and its extended
version $\fgosp(V)$. We then recall the notion of Drinfeld--Jimbo-type (extended) orthosymplectic quantum supergroups, see
Definitions~\ref{def:quantum-osp-lattice} and~\ref{def:quantum-osp-standard}. The key property of $\uqVe$ being a~Drinfeld double
and the corresponding Drinfeld twist~$\uqVe^{\CF}$ are discussed in Propositions~\ref{prop:osp-DJ-pairing_finite}
and~\ref{prop:osp-drinfeld-twist-iso}.

\item
In Section~\ref{sec:RTT_osp}, we study the $\RLL$-realization of extended orthosymplectic quantum supergroups, with the corresponding
$R$-matrix~\eqref{eq:osp-evaluated-R} evaluated in~\cite{ht1}. Similar to Section~\ref{sec:RTT_finite}, we establish their generalized
double realization and relate them to the twisted algebra $\URtw$ in Propositions~\ref{prop:osp-RTT-pairing_finite}
and~\ref{prop:osp-twisted-rtt}. The key (twisted) relations among the entries in the Gauss decomposition of the matrices $L^\pm$ are
summarized
in tables of Section~\ref{ssec:osp-Gauss}. One of the key results is the identification of Drinfeld--Jimbo and $\RLL$ realizations
of the extended orthosymplectic quantum supergroups, established in Theorem~\ref{thm:osp DrJ to RTT}. Again, while the inverse morphism
is constructed in Section~\ref{ssec:inverse morphism osp}, we emphasize that the map $\xi$ of Theorem~\ref{thm:osp DrJ to RTT}
is compatible with Drinfeld double structures and is thus easily argued to be an isomorphism.

\item
In Appendix~\ref{sec:double-construction}, we discuss in detail the super versions of the usual Drinfeld and generalized doubles.

\end{itemize}

\section{General linear Lie superalgebras and quantum supergroups}\label{sec:gl-DJ}

\subsection{Superalgebras and Lie superalgebras}
 A~vector space $V$ over a~base field $\BK$ is called a~\emph{superspace} if it is $\BZ_{2}$-graded, i.e., $V = V_{\bbar{0}} \oplus V_{\bbar{1}}$, where $V_{\bbar{i}}$ denotes the component of degree $\bbar{i} \in \BZ_{2}$.
For a~homogeneous element $v \in V$, its $\BZ_{2}$-degree is denoted by $|v| \in \BZ_{2}$. Elements of
degree $\bbar{0}$ (resp.\ $\bbar{1}$) are called \emph{even} (resp.\ \emph{odd}). If the underlying vector space
of $V$ is finite-dimensional, the pair $(\dim V_{\bbar{0}}, \dim V_{\bbar{1}})$ is called the \emph{dimension} of
the superspace $V$.

The category $\sVect_{\BK}$ of all superspaces over $\BK$, together with all degree-preserving linear maps,
forms a~symmetric tensor category. The tensor product is the standard tensor product of vector spaces equipped
with the $\BZ_{2}$-grading
\begin{gather*}
 (V \otimes W)_{\bbar{k}} = \bigoplus_{\bbar{i}+\bbar{j} = \bbar{k}} V_{\bbar{i}} \otimes W_{\bbar{j}},
\end{gather*}
and the tensor unit is the base field $\BK$, viewed as a~superspace concentrated in degree $\bbar{0}$.
The braiding is given by the graded permutation
\begin{gather*}\label{eq:graded perm}
 \tau_{VW} \colon\ V \otimes W \to W \otimes V, \qquad v \otimes w \mapsto (-1)^{|v||w|} w \otimes v.
\end{gather*}
Here, we adopt the convention $(-1)^{\bbar{0}} = 1$ and $(-1)^{\bbar{1}} = -1$. Unless stated otherwise, $\BK$ as an
object of $\sVect_{\BK}$ will always denote that tensor unit. We frequently omit the subscripts on the tensor product
$\otimes$ and the braiding $\tau$ whenever the underlying base field or vector spaces are clear from the context.

\begin{Remark}\label{rem:tensor braid action}
Since the braiding $\tau$ is symmetric (that is, $\tau_{W V} \circ \tau_{V W} = \Id_{V \otimes W}$), composite morphisms
built from consecutive braidings depend only on the underlying permutation. Thus, for each permutation $\sigma$ in the
symmetric group $S_{r}$, there is a~uniquely defined natural isomorphism with components
\begin{gather}\label{eq:tensor braid action}
 \tau^{\sigma}_{V_{1} \ldots V_{r}} \colon\
 V_{1} \otimes \cdots \otimes V_{r} \iso V_{\sigma^{-1}(1)} \otimes \cdots \otimes V_{\sigma^{-1}(r)}.
\end{gather}
Again, the subscripts will frequently be omitted when the underlying superspaces are clear from the context.
\end{Remark}

\begin{Remark}
We note that the forgetful functor from $\sVect_{\BK}$ to $\Vect_{\BK}$ is strictly monoidal (but not braided monoidal)
and faithful. However, it is not full, as an arbitrary linear map between superspaces need not be degree-preserving.
In practice, we often regard superspaces as objects in $\Vect_{\BK}$ to consider such arbitrary linear maps. To prevent
ambiguity, we reserve the term \emph{morphism} for degree-preserving linear maps (i.e., the morphisms in
$\sVect_{\BK}$), and use the term \emph{linear map} or \emph{map} for arbitrary linear transformations in $\Vect_{\BK}$.
\end{Remark}

Since $\sVect_{\BK}$ is monoidal, one can define \emph{superalgebras}, which are monoid objects of $\sVect_{\BK}$.
Moreover, since $\sVect_{\BK}$ is braided, one can define the tensor product of superalgebras $A$, $B$ with the
multiplication given by
\begin{gather}\label{eq:superalgebras tensoring}
 \mu_{A \otimes B} \colon\ (A \otimes B) \otimes (A \otimes B)
 \xrightarrow{\Id \otimes \tau_{BA} \otimes \Id} (A \otimes A) \otimes (B \otimes B)
 \xrightarrow{\mu_{A} \otimes \mu_{B}} A \otimes B,
\end{gather}
in which $\mu_{A}$, $\mu_{B}$ denote the multiplication of $A$, $B$, respectively. Concretely, for any homogeneous~${a,a' \in A}$
and $b,b' \in B$, the multiplication can be written as follows:
\begin{gather*}
 (a \otimes b)(a' \otimes b') = (-1)^{|b||a'|} (aa') \otimes (bb').
\end{gather*}
We note that if both $A$ and $B$ are superalgebras, then $\tau \colon A \otimes B \to B \otimes A$ is a~superalgebra homomorphism.

Finally, the symmetric braiding allows one to define Lie algebra objects in $\sVect_{\BK}$, called \emph{Lie superalgebras}.
Explicitly, a~Lie superalgebra is a~superspace $\fg$ together with a~morphism (hence degree-preserving)
\begin{gather*}
 [-,-] \colon\ \fg \otimes \fg \to \fg,
\end{gather*}
called a~\emph{Lie superbracket}, which will be frequently identified with the corresponding bilinear map
$\fg \times \fg \to \fg$, satisfying the graded skew-symmetry
\begin{gather*}
 [x,y] = -[\tau(x \otimes y)] = -(-1)^{|x||y|} [y,x] \qquad \text{for any homogeneous} \ x,y \in \fg,
\end{gather*}
and the graded Jacobi identity (equivalent to saying that the superbracket gives an adjoint action of $\fg$)
\begin{gather*}
 [x,[y,z]] = [[x,y],z] + (-1)^{|x||y|} [y,[x,z]] \qquad \text{for any homogeneous} \ x,y,z \in \fg.
\end{gather*}
We note that any superalgebra $A$ admits a~canonical Lie superalgebra structure given by the \emph{supercommutator}
\begin{gather}\label{eq:super commutator}
 [a,a'] \coloneqq aa' - (-1)^{|a||a'|} a'a \qquad \text{for any homogeneous} \ a,a' \in A.
\end{gather}
This construction provides a~forgetful functor from the category of superalgebras to the category of Lie superalgebras,
which admits a~left adjoint known as the \emph{universal enveloping superalgebra} functor. To be specific, for any Lie
superalgebra $\fg$, its universal enveloping superalgebra $U(\fg)$ is defined as the quotient of the tensor superalgebra
$T(\fg) \coloneqq \bigoplus_{k \geq 0} \fg^{\otimes k}$ by the two-sided ideal generated by
$\{x \otimes y - \tau(x \otimes y) - [x,y] \mid x,y \in \fg\}$.

For a~superspace $V$ together with a~choice of an ordered homogeneous basis $\{v_{1}, v_{2}, \ldots, v_{N}\}$,
we define the \emph{parity} of an index $i \in \mathbb{I} \coloneqq \{1, 2, \ldots, N\}$ to be the $\BZ_{2}$-degree
of the corresponding basis vector $v_{i}$
\begin{gather*}
 \overline{i} \coloneqq |v_{i}| =
 \begin{cases}
 \bbar{0} & \text{if } v_{i} \text{ is even}, \\
 \bbar{1} & \text{if } v_{i} \text{ is odd}.
 \end{cases}
\end{gather*}
The grading of this ordered basis is recorded in the \emph{parity sequence} $\gamma_{V}$, defined as the $N$-tuple
\begin{gather*}
 \gamma_{V} \coloneqq (|v_{1}|, \ldots, |v_{N}|) = \bigl(\overline{1}, \ldots, \overline{N}\bigr) \in \{\bbar{0}, \bbar{1}\}^{N}.
\end{gather*}

\subsection{General linear Lie superalgebras}\label{ssec:super A}

For any pair of objects $V,W \in \sVect_{\BK}$, there exists an \emph{internal hom} object, whose underlying
vector space consists of all $\BK$-linear maps
$\Hom_{\BK}(V,W) = \bigoplus_{\bbar{i},\bbar{j} \in \BZ_{2}} \Hom_{\BK}(V_{\bbar{i}},W_{\bbar{j}})$,
equipped with the $\BZ_{2}$-grading
\begin{gather*}
 \Hom_{\BK}(V,W)_{\bbar{k}} = \bigoplus_{\bbar{j} - \bbar{i} = \bbar{k}} \Hom_{\BK}(V_{\bbar{i}},W_{\bbar{j}}).
\end{gather*}
In other words, $\sVect_{\BK}$ is a~closed monoidal category.

Thus, for a~fixed superspace $V$, the set of all $\BK$-linear maps $\End_{\BK}(V) \coloneqq \Hom_{\BK}(V,V)$ forms a~superalgebra, where the multiplication is given by the composition of linear maps. Moreover, by~\eqref{eq:super commutator},
this superalgebra admits a~canonical Lie superalgebra structure, called the \emph{general linear Lie superalgebra}. We denote
this Lie superalgebra by $\gl(V)$ to distinguish the Lie structure from the associative structure of $\End_{\BK}(V)$.

We also introduce a~closely related object, the \emph{special linear Lie superalgebra}. To this end, for a~finite-dimensional
superspace $V$, recall the definition of the \emph{supertrace} morphism
\begin{gather*}
 \mathrm{sTr} \colon \ \gl(V) \simeq V \otimes V^{*} \xrightarrow{\tau} V^{*} \otimes V \xrightarrow{\textrm{ev}} \BC,
\end{gather*}
where we use the canonical identification $\gl(V) \simeq V \otimes V^*$ and $\textrm{ev}$ is the (left) evaluation morphism.
By construction, one can verify that the supertrace satisfies
\[
 \mathrm{sTr}(XY) = (-1)^{|X||Y|} \mathrm{sTr}(YX) \qquad \text{for any homogeneous $ X,Y \in \End_{\BK}(V)$},
\]
which immediately implies $\mathrm{sTr}([X,Y]) = 0$ for all $X,Y \in \gl(V)$. Therefore, the kernel of the supertrace
morphism forms an ideal of $\gl(V)$, called the \emph{special linear Lie superalgebra} and denoted by $\ssl(V)$.

\subsection{Chevalley--Serre-type presentation}\label{ssec:gl-Chevalley-Serre}

Henceforth, we fix a~superspace $V$ of dimension $(m,n)$ over the base field $\BK = \BC$. We also make a~choice of an ordered homogeneous basis $\{v_{1}, v_{2}, \ldots, v_{N}\}$ of $V$, where $N \coloneqq m+n$, which allows
the identification of $\gl(V)$ with the space of all $N \times N$ matrices. Let $E_{ij}$ be the elementary matrix with
$1$ at the $(i,j)$-entry and $0$ elsewhere. The set \smash{$\{E_{ij}\}_{i,j=1}^{N}$} forms a~homogeneous basis of $\gl(V)$ with
$\BZ_{2}$-grading $|E_{ij}| = \overline{i} + \overline{j}$ and the Lie superbracket given by
\begin{gather}\label{eq:gl-bracket}
 [E_{ij}, E_{kl}] = \delta_{jk} E_{il} - (-1)^{(\overline{i} + \overline{j})(\overline{k} + \overline{l})} \delta_{li} E_{kj}.
\end{gather}
We also note that $\mathrm{sTr}(E_{ij}) = \delta_{ij} (-1)^{\overline{i}}$, and consequently, the following forms a~homogeneous basis of $\ssl(V)$
\begin{gather*}
\{ E_{ij} \mid 1 \leq i \neq j \leq N \} \cup
 \bigl\{ (-1)^{\overline{i}} E_{ii} - (-1)^{\overline{i+1}} E_{i+1,i+1} \mid 1 \leq i < N \bigr\}.
\end{gather*}

We further choose the Borel subalgebra of $\gl(V)$ consisting of all upper triangular matrices, giving rise to the
Borel subalgebra of $\ssl(V)$. The corresponding Cartan subalgebra $\fh_{\gl}$ of~$\gl(V)$ consists of all diagonal
matrices, with a~basis $\{E_{ii}\}_{i=1}^N$, while the Cartan subalgebra~$\fh_{\ssl}$ of~$\ssl(V)$ is a~codimension $1$
subspace in~$\fh_{\gl}$:
\[
 \fh_{\ssl} = \{ X \in \fh_{\gl} \mid \mathrm{sTr}(X) = 0 \}.
\]
We define the linear functionals~$\{\varepsilon_{i}\}_{i=1}^N$ in $\fh_{\gl}^*$ to be the dual basis to $\{E_{ii}\}_{i=1}^N$, i.e., $\varepsilon_i(E_{jj}) = \delta_{ij}$ for all $i$, $j$.

A~direct computation shows that $[E_{ii}, E_{ab}] = (\varepsilon_a - \varepsilon_b)(E_{ii}) \cdot E_{ab}$, so that
$E_{ab}$ is a~root vector corresponding to the root $\varepsilon_a - \varepsilon_b$ (for $a \neq b$). Hence, we obtain
the \emph{root space decomposition}~${\gl(V) = \fh_{\gl} \oplus \bigoplus_{\alpha \in \Phi} \gl(V)_{\alpha}}$ with the root
system given by
\begin{gather}\label{eq:gl-root-sys}
 \Phi = \{ \varepsilon_a - \varepsilon_b \mid 1 \leq a, b \leq N,\, a \neq b \},
\end{gather}
and the corresponding polarization
\begin{gather}\label{eq:gl-polarization}
 \Phi^{+} = \{ \varepsilon_a - \varepsilon_b \mid a < b \}, \qquad
 \Phi^{-} = \{ \varepsilon_a - \varepsilon_b \mid a > b \}.
\end{gather}
The root system has a~decomposition $\Phi = \Phi_{\bbar{0}} \cup \Phi_{\bbar{1}}$ into \emph{even} and \emph{odd} parts,
where
\begin{gather*}
 \Phi_{\bbar{0}} = \{ \varepsilon_a - \varepsilon_b \in \Phi \mid \overline{a} = \overline{b} \},\qquad
 \Phi_{\bbar{1}} = \{ \varepsilon_a - \varepsilon_b \in \Phi \mid \overline{a} \neq \overline{b} \}.
\end{gather*}
We note that $\ssl(V)$ shares the same root system $\Phi$ and admits the corresponding root space decomposition
\[
 \ssl(V) = \fh_{\ssl} \oplus \bigoplus_{\alpha \in \Phi} \ssl(V)_{\alpha} \qquad \text{with $
 \ssl(V)_{\alpha} = \gl(V)_{\alpha}$, $\forall \alpha \in \Phi$}.
\]

Following the polarization~\eqref{eq:gl-polarization} of the root system~\eqref{eq:gl-root-sys}, the simple roots and
the corresponding root vectors can be written as follows (with $1 \leq i < N$)
\begin{gather}\label{eq:gl-Lie-action}
 \alpha_{i} = \varepsilon_{i} - \varepsilon_{i+1}, \qquad
 \sse_{i} = E_{i,i+1}, \qquad \ssf_{i} = (-1)^{\overline{i}} E_{i+1,i}.
\end{gather}
We define the \emph{Weyl vector} $\rho$ of $\Phi$ by
\begin{gather}\label{eq:weyl-vector}
 \rho \coloneqq
 \frac{1}{2} \sum_{\alpha \in \Phi_{\bbar{0}}^{+}} \alpha - \frac{1}{2} \sum_{\alpha \in \Phi_{\bbar{1}}^{+}} \alpha.
\end{gather}
The following property of the Weyl vector will be used extensively in the sequel (cf.~\cite[Proposition 1.33]{cw}):
\begin{gather}\label{eq:weyl-vector-prop}
 (\rho, \alpha_i) = \tfrac{1}{2}(\alpha_i, \alpha_i) \qquad \forall i.
\end{gather}

We also define the following elements of the Cartan subalgebra $\fh_{\gl}$:
\begin{gather}\label{eq:gl-Lie-action-Cartan}
 \ssH_{k} = (-1)^{\overline{k}} E_{kk} \qquad \text{for } 1 \leq k \leq N,\qquad
 \ssh_{i} = \ssH_{i} - \ssH_{i+1} \qquad \text{for } 1 \leq i < N.
\end{gather}
Note that the sets \smash{$\{\ssH_{k}\}_{k=1}^{N}$} and \smash{$\{\ssh_{i}\}_{i=1}^{N-1}$} form bases for $\fh_{\gl}$ and $\fh_{\ssl}$,
respectively. Furthermore, the lattices \smash{$P^{\vee} \coloneqq \bigoplus_{k=1}^{N} \BZ \ssH_{k}$} and
\smash{$Q^{\vee} \coloneqq \bigoplus_{i=1}^{N-1} \BZ \ssh_{i}$} constitute the coweight and coroot lattices of $\gl(V)$,
respectively.

Dually, we define the weight lattice $P \coloneqq \bigoplus_{k=1}^{N} \BZ\varepsilon_{k}$,
the root lattice $Q \coloneqq \bigoplus_{i=1}^{N-1} \BZ\alpha_{i} \subseteq P$,
and \smash{$Q^{+} \coloneqq \bigoplus_{i=1}^{N-1} \BZ_{\geq 0}\alpha_{i}$}, together with the following
partial order on $P$: $\mu \leq \nu \Leftrightarrow \nu - \mu \in Q^{+}$.
Then~\eqref{eq:gl-bracket} implies that $\gl(V)$ is graded by $Q$ (and thus also by $P$) via
\[
 \deg(E_{ij}) = \varepsilon_{i} - \varepsilon_{j},
\]
which restricts to a~$Q$-grading on $\ssl(V)$, and is compatible with the $\BZ_{2}$-grading via the group homomorphism
\begin{gather}\label{eq:gl-grading-compatible}
 P \to \BZ_{2}, \qquad \varepsilon_{i} \mapsto \overline{i}.
\end{gather}
In particular, we have
\begin{gather}\label{eq:gl-gen-Q-grading}
 \deg(\sse_i) = \alpha_i,\qquad \deg(\ssf_i) = -\alpha_i,\qquad \deg(\ssh_i) = 0,\qquad \deg(\ssH_k) = 0.
\end{gather}

We consider the non-degenerate invariant \emph{supertrace} bilinear form $(\cdot,\cdot) \colon \gl(V) \times \gl(V) \to \BC$
defined by
\[
 (X,Y) = \mathrm{sTr}(XY), \qquad \forall X,Y\in \gl(V).
\]
Its restriction to the Cartan subalgebra $\fh_{\gl}$ is non-degenerate, thus giving rise to an identification
$\fh_{\gl}^{*} \simeq \fh_{\gl}$ via \smash{$\varepsilon_{k} \leftrightarrow (-1)^{\overline{k}} E_{kk}$} and inducing a~bilinear
form $(\cdot, \cdot) \colon \fh_{\gl}^{*} \times \fh_{\gl}^{*} \to \BC$ such that
\begin{gather}\label{eq:gl-epsilon-pairing}
 (\varepsilon_{k}, \varepsilon_{l}) = \delta_{kl} (-1)^{\overline{k}} \qquad \textrm{for any} \ 1 \leq k,l \leq N.
\end{gather}
We also recall that an odd root $\alpha \in \Phi_{\bbar{1}}$ is called \emph{isotropic} if $(\alpha,\alpha)=0$.

Define the \emph{symmetrized Cartan matrix} \smash{$(a_{ij})_{i,j=1}^{N-1}$} of $\gl(V)$ via $a_{ij} = (\alpha_{i},\alpha_{j})$.
It is straightforward to verify that the elements $\{\sse_{i}, \ssf_{i}\}_{i=1}^{N-1} \cup \{\ssH_{k}\}_{k=1}^{N}$
satisfy the following relations:
\begin{gather}\label{eq:gl-chevalley-rel}
 [\ssH_{k}, \ssH_{l}] = 0,\qquad
 [\ssH_{k}, \sse_{j}] = (\varepsilon_{k},\alpha_{j}) \sse_{j},\qquad
 [\ssH_{k}, \ssf_{j}] = -(\varepsilon_{k},\alpha_{j}) \ssf_{j},\qquad
 [\sse_{i}, \ssf_{j}] = \delta_{ij} \ssh_{i},
\end{gather}
which imply the following Chevalley-type relations:
\begin{gather}\label{eq:sl-chevalley-rel}
 [\ssh_{i}, \ssh_{j}] = 0,\qquad
 [\ssh_{i}, \sse_{j}] = a_{ij} \sse_{j},\qquad
 [\ssh_{i}, \ssf_{j}] = -a_{ij} \ssf_{j}.
\end{gather}
Furthermore, it is shown in~\cite{z} (see also~\cite{y}) that $\ssl(V)$ is isomorphic to the Lie superalgebra generated by
$\{\sse_i,\ssf_i,\ssh_i\}_{i=1}^{N-1}$, subject to the $\BZ_{2}$-grading induced from the $Q$-grading
\eqref{eq:gl-gen-Q-grading}, the Chevalley-type relations~\eqref{eq:sl-chevalley-rel}, together with the
\emph{standard} and \emph{higher-order Serre relations}. As we shall not need the explicit form of the
Serre relations, we refer the interested reader to~\cite[Section~3.2.1]{z} for the exact formulas.
Similarly, $\gl(V)$ admits a~generator-relation presentation: it is isomorphic to the Lie superalgebra generated
by $\{\sse_{i}, \ssf_{i}\}_{i=1}^{N-1} \cup \{\ssH_{k}\}_{k=1}^{N}$, subject to the same grading and Serre relations
(which depend only on $\sse_i$, $\ssf_i$), but satisfying the Chevalley-type relations~\eqref{eq:gl-chevalley-rel}.

\begin{Remark}\label{rem:lattice generator}
Instead of using Cartan generators $\ssh_{i}$ or $\ssH_{k}$, one can use an additive subgroup~${\Gamma \subseteq \fh_{\gl}}$
containing $\{\ssh_{i}\}_{i=1}^{N-1}$ to provide a~uniform Chevalley--Serre type presentation for both~$\gl(V)$ and $\ssl(V)$.
Consider the Lie superalgebra $\fg(\Gamma)$ generated by $\{\sse_{i}, \ssf_{i}\}_{i=1}^{N-1}$ and $\Gamma$, subject to the
Serre relations together with the following Chevalley-type relations:
\begin{gather*}
 [\ssH, \ssH'] = 0,\qquad
 [\ssH, \sse_{j}] = \alpha_{j}(\ssH) \sse_{j},\qquad
 [\ssH, \ssf_{j}] = -\alpha_{j}(\ssH) \ssf_{j},\qquad
 [\sse_{i}, \ssf_{j}] = \delta_{ij} \ssh_{i}\\
 \forall \ssH,\ssH' \in \Gamma,\qquad 1 \leq i,j < N.
\end{gather*}
It is clear that $\fg\bigl(P^{\vee}\bigr) = \gl(V)$ and $\fg\bigl(Q^{\vee}\bigr) = \ssl(V)$. While various choices of the lattice $\Gamma$ yield only
these two distinct Lie superalgebras, the $\Gamma$-dependence will be more significant for quantum supergroups below.
\end{Remark}

\subsection{Drinfeld--Jimbo-type general linear quantum supergroups}\label{ssec:q-gl}

The \emph{Drinfeld--Jimbo-type general linear quantum supergroup} $\uqgl$ is a~$\BC(q)$-superalgebra which is a~natural
quantization of the universal enveloping superalgebra $U(\gl(V))$. To construct it, we first define a~more general class of quantum supergroups $U_q(\Gamma)$, parameterized by an additive subgroup $\Gamma \subseteq P^{\vee}$ containing $Q^{\vee}$
(cf.\ Remark~\ref{rem:lattice generator}).

\begin{Remark}\label{rem:Cartan-notation-gl}
To avoid confusion, in the quantum supergroup setting, the elements $\ssH_k$ and $\ssh_i \in \fh_{\gl}$ will be denoted
by $H_k$ and $h_i$ respectively to distinguish them from the classical Lie superalgebra setting. Moreover, following
the physics literature, we use formal exponentials~$q^{\pm H_k}$ and $q^{\pm h_i}$ instead of the more common $K_k^{\pm 1}$
and $k_i^{\pm 1}$.
\end{Remark}

\begin{Definition}\label{def:quantum-gl-lattice}
For a~fixed additive subgroup $\Gamma \subseteq P^{\vee}$ containing $Q^{\vee}$, the quantum supergroup~$U_q(\Gamma)$ is the $\BC(q)$-superalgebra generated by $\{e_i, f_i\}_{i=1}^{N-1}$ and the formal symbols $\bigl\{q^{\pm H} \mid H \in \Gamma\bigr\}$. It is equipped with the $Q$-grading (cf.~\eqref{eq:gl-gen-Q-grading})
\begin{gather}\label{eq:uqGamma-Q-grading}
 \deg(e_i) = \alpha_i,\qquad \deg(f_i) = -\alpha_i,\qquad \deg\bigl(q^{\pm H}\bigr) = 0,
\end{gather}
and the compatible $\BZ_{2}$-grading, see~\eqref{eq:gl-grading-compatible}. These generators are subject to the Serre
relations~\eqref{eq:q-gl-serre-rel-standard-1}--\eqref{eq:q-gl-serre-rel-standard-3} introduced later in this subsection,
together with the following Chevalley-type relations:\looseness=-1
\begin{align}
 & q^{H}q^{H'} = q^{H + H'} , \qquad
 q^0=1 ,
 \label{eq:q-gamma-chevalley-rel-HH}\\
 & q^{H} e_j q^{-H} = q^{\alpha_j(H)} e_j , \qquad
 q^{H} f_j q^{-H} = q^{-\alpha_j(H)} f_j ,
 \label{eq:q-gamma-chevalley-rel-He} \\
 & [e_i, f_j] = \delta_{ij} \frac{q^{h_i} - q^{-h_i}}{q - q^{-1}} ,
 \label{eq:q-gamma-chevalley-rel-ef}
\end{align}
for all $H,H' \in \Gamma$ and $1 \leq i$, $j < N$, where $[-,-]$ denotes the supercommutator~\eqref{eq:super commutator},
and~${\alpha_j(H)\in \BZ}$ represents the evaluation of the root $\alpha_j$ on the Cartan element $H \in P^{\vee}$.
We note that~\eqref{eq:q-gamma-chevalley-rel-HH} implies the commutativity $\bigl[q^H, q^{H'}\bigr] = 0$
for all $H,H' \in \Gamma$.
\end{Definition}

To express the Serre relations, we introduce the \emph{$q$-supercommutator} $[\![ -,- ]\!]$ via
\begin{gather}\label{eq:q-gl-superbracket}
 [\![ a,b ]\!] \coloneqq ab - (-1)^{|a||b|} q^{(\deg(a), \deg(b))} \cdot ba
\end{gather}
for any $Q$-homogeneous elements $a,b \in U_q(\Gamma)$, cf.~\eqref{eq:gl-epsilon-pairing}.
The Serre relations, which complement the Chevalley-type relations above, are given as follows (cf.~\cite[Proposition 10.4.1]{y0}):
\begin{align}
 & [\![e_i,e_j]\!]=0, \qquad [\![f_i,f_j]\!]=0 \qquad \text{if } a_{ij}=0,
 \label{eq:q-gl-serre-rel-standard-1}\\
 & [\![e_i,[\![e_i,e_j]\!]]\!]=0, \qquad [\![f_i,[\![f_i,f_j]\!]]\!]=0
 \qquad \text{if } |i-j|=1\ \text{and}\ \alpha_i\in \Phi_{\bbar{0}},
 \label{eq:q-gl-serre-rel-standard-2}\\
 & [\![[\![[\![e_{i-1},e_{i}]\!],e_{i+1}]\!],e_i]\!] = 0,\quad
 [\![[\![[\![f_{i-1},f_{i}]\!],f_{i+1}]\!],f_i]\!] = 0 \nonumber \\
 & \qquad \text{if } 2 \leq i \leq N-2 \ \text{and}\ \alpha_{i}\in \Phi_{\bbar{1}}.\label{eq:q-gl-serre-rel-standard-3}
\end{align}

\begin{Remark}
In particular, if $\alpha_i \in \Phi_{\bbar{1}}$, then~\eqref{eq:q-gl-serre-rel-standard-1} immediately implies
$e_i^2 = f_i^2 = 0$. We also note that the relation~\eqref{eq:q-gl-serre-rel-standard-2} can be alternatively written
as $[\![[\![e_{i \pm 1},e_i]\!],e_i]\!] = 0$, $[\![[\![f_{i \pm 1},f_i]\!],f_i]\!] =0$.
\end{Remark}

The quantum supergroup $U_q(\Gamma)$ is naturally endowed with a~Hopf superalgebra structure, defined on the generators
by the following comultiplication:
\begin{gather}
 \Delta(e_i) = q^{h_i} \otimes e_i + e_i \otimes 1,\qquad
 \Delta(f_i) = 1 \otimes f_i + f_i \otimes q^{-h_i},\nonumber\\
 \Delta\bigl(q^{H}\bigr) = q^{H} \otimes q^{H},\label{eq:gl-comult}
\end{gather}
and with the counit and antipode given explicitly by
\begin{gather}
 \epsilon(e_i) = 0, \qquad \epsilon(f_i) = 0, \qquad \epsilon\bigl(q^{H}\bigr) = 1,\nonumber\\
 S(e_i) = -q^{-h_i} e_i, \qquad S(f_i) = -f_i q^{h_i}, \qquad S\bigl(q^{H}\bigr) = q^{-H}.\label{eq:gl-counit}
\end{gather}
The assignment~\eqref{eq:uqGamma-Q-grading} indeed gives rise to a~$Q$-grading on $U_q(\Gamma)$, since all
the defining relations are homogeneous. Furthermore, the explicit formulas~\eqref{eq:gl-comult}, \eqref{eq:gl-counit}
show that $U_q(\Gamma)$ is actually a~$Q$-graded Hopf superalgebra.

\begin{Definition}\label{def:quantum-gl-standard}\quad
\begin{itemize}\itemsep=0pt
\item[(a)] The quantum supergroup $U_q\bigl(P^{\vee}\bigr)$ is called the
\emph{Drinfeld--Jimbo-type general linear quantum supergroup} and is denoted by $\uqgl$. Equivalently, $\uqgl$ admits a~generator-relation presentation with only finitely many Cartan generators: it is the $\BC(q)$-superalgebra generated by
$\{e_i, f_i\}_{i=1}^{N-1}$ and $\{q^{\pm H_k}\}_{k=1}^{N}$, equipped with the same $Q$-grading and $\BZ_2$-grading as
in~\eqref{eq:uqGamma-Q-grading}. These generators are subject to the Serre
relations~\eqref{eq:q-gl-serre-rel-standard-1}--\eqref{eq:q-gl-serre-rel-standard-3}, alongside the following explicit
Chevalley-type relations:
\begin{align}
 & \bigl[q^{H_k}, q^{H_l}\bigr] = 0, \qquad
 q^{\pm H_k} q^{\mp H_k} = 1,
 \label{eq:q-gl-chevalley-rel-HH}\\
 & q^{H_k} e_j q^{-H_k} = q^{(\varepsilon_k, \alpha_j)} e_j, \quad
 q^{H_k} f_j q^{-H_k} = q^{-(\varepsilon_k, \alpha_j)} f_j,
 \label{eq:q-gl-chevalley-rel-He} \\
 & [e_i, f_j] = \delta_{ij} \frac{q^{h_i} - q^{-h_i}}{q - q^{-1}},
 \label{eq:q-gl-chevalley-rel-ef}
\end{align}
for all $1 \leq k, l \leq N$ and $1 \leq i, j < N$, where $q^{h_i} \coloneqq q^{H_i} q^{-H_{i+1}}$, and the pairing
in~\eqref{eq:q-gl-chevalley-rel-He} is given by~\eqref{eq:gl-epsilon-pairing}.

\item[(b)] Analogously, the \emph{Drinfeld--Jimbo-type special linear quantum supergroup} is defined as
follows: $\uqsl \coloneqq U_q\bigl(Q^{\vee}\bigr)$.
 Equivalently, it admits a~finite generator-relation presentation as the
$\BC(q)$-superalgebra generated by \smash{$\{e_i, f_i\}_{i=1}^{N-1}$} and the Cartan generators \smash{$\{q^{\pm h_i}\}_{i=1}^{N-1}$},
subject to the same $Q$-grading and $\BZ_2$-grading as in~\eqref{eq:uqGamma-Q-grading}. These generators are subject to
the Serre relations~\eqref{eq:q-gl-serre-rel-standard-1}--\eqref{eq:q-gl-serre-rel-standard-3} and the following
explicit Chevalley-type relations:
\begin{gather*}
 \bigl[q^{h_i}, q^{h_j}\bigr] = 0,\qquad
 q^{\pm h_i} q^{\mp h_i} = 1, \\
 q^{h_i} e_j q^{-h_i} = q^{(\alpha_i, \alpha_j)} e_j,\qquad
 q^{h_i} f_j q^{-h_i} = q^{-(\alpha_i, \alpha_j)} f_j,\qquad
 [e_i, f_j] = \delta_{ij} \frac{q^{h_i} - q^{-h_i}}{q - q^{-1}},
\end{gather*}
for all $1 \leq i, j < N$. By definition, there is a~canonical Hopf superalgebra morphism from $\uqsl$ to $\uqgl$.
This morphism can be shown to be an embedding via the triangular decomposition~\eqref{eq:DJ-tri-decomp} discussed below.

\end{itemize}
\end{Definition}

\begin{Remark}
We note that a~slightly different definition of the $q$-supercommutator is used in~\cite{y0}:
\begin{align*}
 [\![ a,b ]\!]_{\mathrm{Y}} &\coloneqq ab - (-1)^{|a||b|} q^{-(\deg(a), \deg(b))} \cdot ba =
 - (-1)^{|a||b|} q^{-(\deg(a), \deg(b))} \cdot [\![ b,a ]\!].
\end{align*}
Nevertheless, the superalgebra $U_{\:q}(\gl\:(V))_{\:\mathrm{Y}}$ defined via $[\![\:-,-\:]\!]_{\:\mathrm{Y}}$ in \cite{y0}
is isomorphic to $U_{q^{-1}}(\gl(V))$ as $\BC(q)$-superalgebras via the assignment
$e_i \mapsto e_i$, $f_i \mapsto f_i$, $q^{\pm H_k} \mapsto q^{\mp H_k}$.
\end{Remark}

\begin{Remark}\label{rem:DrJ involution}
We note that the assignment (see~\eqref{eq:weyl-vector})
\begin{gather*}
 \omega_{\DrJ}\bigl(q^{\pm H_k}\bigr) = q^{\mp H_k},\qquad
 \omega_{\DrJ}(e_i) = -(-1)^{\overline{i} \overline{i+1}} q^{3(\rho, \alpha_{i})} f_i,\\
 \omega_{\DrJ}(f_i) = -(-1)^{\overline{i} \overline{i+1}}(-1)^{\overline{i}+\overline{i+1}} q^{-3(\rho, \alpha_{i})} e_i
\end{gather*}
for $1 \leq k \leq N$, $1 \leq i < N$ gives rise to a~Hopf superalgebra isomorphism
$\omega_{\DrJ} \colon \uqgl \iso \allowbreak \uqgl^{\copp}$.
\end{Remark}

\subsection[Drinfeld double construction of U\_q(gl(V))]{Drinfeld double construction of $\boldsymbol{\uqgl}$}

We first introduce the \emph{triangular decomposition} of $\uqgl$. Let $U^>_{q}(\gl(V))$, $U^<_{q}(\gl(V))$, and
$U^0_{q}(\gl(V))$ denote the subalgebras of $\uqgl$ generated by \smash{$\{e_i\}_{i=1}^{N-1}$}, \smash{$\{f_i\}_{i=1}^{N-1}$}, and
\smash{$\bigl\{q^{\pm H_{k}}\bigr\}_{k=1}^{N}$}, respectively. The triangular decomposition states that the multiplication morphism
\begin{gather}\label{eq:DJ-tri-decomp}
 U^<_{q}(\gl(V)) \otimes U^0_{q}(\gl(V)) \otimes U^>_{q}(\gl(V)) \iso \uqgl
\end{gather}
is an isomorphism of underlying superspaces. We further define the \emph{positive} and \emph{negative Borel subalgebras},
denoted $U^{\geq}_{q}(\gl(V))$ and $U^{\leq}_{q}(\gl(V))$, as the subalgebras generated by
$\{e_i\}_{i=1}^{N-1} \cup \smash{\{q^{\pm H_{k}}\}_{k=1}^{N}}$ and \smash{$\{f_i\}_{i=1}^{N-1} \cup \{q^{\pm H_{k}}\}_{k=1}^{N}$},
respectively. Note that these are not just subalgebras but are also Hopf subalgebras of $\uqgl$.

\begin{Remark}\label{rem:triangular-decomp-morphisms}
The triangular decomposition~\eqref{eq:DJ-tri-decomp} implies that the multiplication morphisms
\begin{gather*}
 U^0_{q}(\gl(V)) \otimes U^>_{q}(\gl(V)) \iso U^{\geq}_{q}(\gl(V)),\qquad
 U^<_{q}(\gl(V)) \otimes U^0_{q}(\gl(V)) \iso U^{\leq}_{q}(\gl(V))
\end{gather*}
are isomorphisms of underlying superspaces. Invoking the automorphism $\omega_{\DrJ}$ of Remark~\ref{rem:DrJ involution},
the following multiplication morphisms are also isomorphisms:
\begin{gather*}
 U^>_{q}(\gl(V)) \otimes U^0_{q}(\gl(V)) \iso U^{\geq}_{q}(\gl(V)),\qquad
 U^0_{q}(\gl(V)) \otimes U^<_{q}(\gl(V)) \iso U^{\leq}_{q}(\gl(V)).
\end{gather*}
\end{Remark}
 A~fundamental structural property of quantum (super)groups is their realization via the Drinfeld double construction
(cf.\ Appendix~\ref{ssec:generalized-doubles}). Recall that the \emph{opposite comultiplication} is defined by
$\Delta^{\opp}(x) \coloneqq \tau(\Delta(x))$.

\begin{Proposition}\label{prop:DJ-pairing_finite}\quad
\begin{itemize}\itemsep=0pt
\item[$(a)$] There exists a~unique skew-pairing $($cf.~\eqref{eq:skew pairing structural property}$)$
\begin{gather}\label{eq:gl-skew-pairing}
 (\cdot,\cdot)_{\DrJ}\colon \ U^\leq_{q}(\gl(V)) \times U^\geq_{q}(\gl(V)) \to \BC(q)
\end{gather}
satisfying the following properties on the generators for $1 \leq i,j < N$ and $1 \leq k,l \leq N$:
\begin{gather}
 \bigl(f_i, q^{\pm H_k}\bigr)_{\DrJ} = 0, \qquad \bigl(q^{\pm H_k}, e_i\bigr)_{\DrJ} = 0,\nonumber \\
 (f_i, e_j)_{\DrJ} = \frac{\delta_{ij} (-1)^{|f_i||e_j|}}{q^{-1} - q}, \qquad
 \bigl(q^{H_k}, q^{H_l}\bigr)_{\DrJ} = q^{-(\varepsilon_k, \varepsilon_l)}.\label{eq:generators-pairing}
\end{gather}
Moreover, this pairing is non-degenerate.

\item[$(b)$] $\uqgl$ is isomorphic to the quotient of the Drinfeld double $\CD_{\DrJ} \coloneqq \CD\bigl(U^\leq_q(\gl(V)),\allowbreak U^\geq_q(\gl(V))\bigr)$
obtained by identifying the Cartan subalgebras $U^0_q(\gl(V))$ from both factors.
\end{itemize}
\end{Proposition}

\begin{proof}
For part (a), the existence of such a~pairing is established in~\cite{y0} (cf.~\cite[Propositions~6.12 and 6.18]{jan} for
the non-super setup). This pairing is of $Q$-degree zero, i.e., for any $Q$-homogeneous~${x \in U^\geq_{q}(\gl(V))}$ and
$y \in U^\leq_{q}(\gl(V))$:
\begin{gather}\label{eq:skew-pairing-degree-zero}
 \deg(x) + \deg(y) \neq 0 \implies (y,x)_{\DrJ} = 0.
\end{gather}

We shall now prove the non-degeneracy of this pairing.
Define the following subspaces, spanned by elements of strictly positive and negative degrees, respectively:
\begin{gather*}
 U^>_{q}(\gl(V))_{\deg>0} \coloneqq \bigl\{x \in U^>_{q}(\gl(V)) \mid \epsilon(x) = 0\bigr\},\\
 U^<_{q}(\gl(V))_{\deg<0} \coloneqq \bigl\{y \in U^<_{q}(\gl(V)) \mid \epsilon(y) = 0\bigr\}.
\end{gather*}
For any $Q$-homogeneous elements $e \in U^>_{q}(\gl(V))$ and $f \in U^<_{q}(\gl(V))$,
the structure of~\eqref{eq:gl-comult} implies that
\begin{gather}
\Delta(e) \in e \otimes 1 + U^\geq_{q}(\gl(V)) \otimes U^>_{q}(\gl(V))_{\deg>0},\nonumber\\
 \Delta(f) \in 1 \otimes f + U^<_{q}(\gl(V))_{\deg<0} \otimes U^\leq_{q}(\gl(V)).\label{eq:gl-borel-comult}
\end{gather}
Consequently, for any $Q$-homogeneous $e \in U^>_{q}(\gl(V))$, $f \in U^<_{q}(\gl(V))$ and group-like
$k,k' \in U^0_{q}(\gl(V))$:
\begin{align}
 (fk',ek)_{\DrJ}
 &= (f \otimes k',\Delta(e)\Delta(k))_{\DrJ}
 \overset{\eqref{eq:skew-pairing-degree-zero}, \, \eqref{eq:gl-borel-comult}}{=} (f,ek)_{\DrJ}(k',k)_{\DrJ}\nonumber\\
 & = \bigl(\Delta^{\opp}(f), e \otimes k\bigr)_{\DrJ}(k',k)_{\DrJ}
 \overset{\eqref{eq:skew-pairing-degree-zero},\, \eqref{eq:gl-borel-comult}}{=} (f,e)_{\DrJ}(1,k)_{\DrJ}(k',k)_{\DrJ}\nonumber\\
 &
 = (f,e)_{\DrJ}(k',k)_{\DrJ}.\label{eq:DrJ-skew-pairing-identity}
\end{align}
As $U^{0}_{q}(\gl(V))$ is spanned by its group-like elements, bilinearity ensures that~\eqref{eq:DrJ-skew-pairing-identity}
holds for all~${k,k' \in U^{0}_{q}(\gl(V))}$. Thus, the non-degeneracy of~\eqref{eq:gl-skew-pairing} follows from that of its
restrictions to~${U^<_{q}(\gl(V)) \times U^>_{q}(\gl(V))}$ and $U^0_{q}(\gl(V)) \times U^0_{q}(\gl(V))$. The former is a~non-trivial result established in~\cite{y0}, while the latter is straightforward.

\begin{Remark}
We note that the Serre relations were introduced in~\cite{y0} precisely to ensure the non-degeneracy of the skew-pairing~\eqref{eq:gl-skew-pairing} upon its restriction to the subalgebras $U^<_{q}(\gl(V))$ and~$U^>_{q}(\gl(V))$.
\end{Remark}

Part (b) follows from the fact that the Drinfeld double $\CD_{\DrJ}$ is the quotient of the free product (super)algebra
$U^\leq_{q}(\gl(V)) * U^\geq_{q}(\gl(V))$, subject to the following cross-relations (cf.\ Proposition~\ref{prop:gen double}):
\begin{gather}\label{eq:gl-DJ-cross-rel}
 \sum_{(x)(y)} (-1)^{|x_{2}||y|} (-1)^{|y_{1}||x_{1}|} \cdot (y_{1},x_{1})_{\DrJ} \cdot y_{2}x_{2}
 = \sum_{(x)(y)} (-1)^{|x_{2}||y|} \cdot x_{1}y_{1} \cdot (y_{2},x_{2})_{\DrJ},
\end{gather}
for all $\BZ_{2}$-homogeneous $x \in U^\geq_{q}(\gl(V))$ and $y \in U^\leq_{q}(\gl(V))$, using Sweedler notation. To avoid
confusion, we denote the Cartan generators of $U^\geq_{q}(\gl(V))$ and~$U^\leq_{q}(\gl(V))$ by \smash{$\bigl\{q^{\pm H_{k}^{e}}\bigr\}_{k=1}^{N}$}
and~\smash{$\bigl\{q^{\pm H_{k}^{f}}\bigr\}_{k=1}^{N}$}, respectively. Analogously, we use \smash{$\bigl\{q^{\pm h_{i}^{e}}\bigr\}_{i=1}^{N-1}$} and
\smash{$\bigl\{q^{\pm h_{i}^{f}}\bigr\}_{i=1}^{N-1}$} to refer to the corresponding composite elements.

Following the notation in Proposition~\ref{prop:gen double}, consider the elements $u_{y,x}$
(defined as the difference between the left- and right-hand sides of~\eqref{eq:gl-DJ-cross-rel}) for $\BZ_{2}$-homogeneous
$y \in U^\leq_{q}(\gl(V))$ and $x \in U^\geq_{q}(\gl(V))$, and let $\CI$ be the two-sided ideal generated by them.
Since $U^\geq_{q}(\gl(V))$ and~${U^\leq_{q}(\gl(V))}$ are graded by $Q^{+}$ and $-Q^{+}$ respectively, the explicit
comultiplication formula~\eqref{eq:gl-comult} together with an induction on the degree via Proposition~\ref{prop:double gen rel}
implies that $\CI$ is generated by the elements $u_{y,x}$ where $y$ and $x$ range over the generators of $U^\leq_{q}(\gl(V))$
and~$U^\geq_{q}(\gl(V))$, respectively. Applying~\eqref{eq:gl-DJ-cross-rel} to pairs of generators yields
\begin{align*}
 & (-1)^{|e_{i}||f_{j}|} \cdot \bigl(1,q^{h_{i}^{e}}\bigr)_{\DrJ} \cdot f_{j}e_{i}
 + (-1)^{|e_{i}||f_{j}|} \cdot (f_{j},e_{i})_{\DrJ} \cdot q^{-h_{j}^{f}} \\
 &\qquad = e_{i}f_{j} \cdot \bigl(q^{-h_{j}^{f}},1\bigr)_{\DrJ}
 + (-1)^{|e_{i}||f_{j}|} \cdot q^{h_{i}^{e}} \cdot (f_{j},e_{i})_{\DrJ},\\
 & \bigl(1,q^{H_{k}^{e}}\bigr)_{\DrJ} \cdot f_{j}q^{H_{k}^{e}}
 = q^{H_{k}^{e}}f_{j} \cdot \bigl(q^{-h_{j}^{f}},q^{H_{k}^{e}}\bigr)_{\DrJ},\\
 & \bigl(q^{H_{l}^{f}},q^{h_{i}^{e}}\bigr)_{\DrJ} \cdot q^{H_{l}^{f}}e_{i}
 = e_{i}q^{H_{l}^{f}} \cdot \bigl(q^{H_{l}^{f}},1\bigr)_{\DrJ},\\
 & \bigl(q^{H_{l}^{f}},q^{H_{k}^{e}}\bigr)_{\DrJ} \cdot q^{H_{l}^{f}}q^{H_{k}^{e}}
 = q^{H_{k}^{e}}q^{H_{l}^{f}} \cdot \bigl(q^{H_{l}^{f}},q^{H_{k}^{e}}\bigr)_{\DrJ},
\end{align*}
which simplifies as follows:
\begin{gather*}
 e_{i}f_{j} - (-1)^{|e_{i}||f_{j}|} f_{j}e_{i}
 = \delta_{ij} \frac{q^{h_{i}^{e}} - q^{-h_{i}^{f}}}{q - q^{-1}}, \\
 q^{H_{k}^{e}}f_{j}q^{-H_{k}^{e}}
 = q^{-(\alpha_{j},\varepsilon_{k})}f_{j},\qquad
 q^{H_{l}^{f}}e_{i}q^{-H_{l}^{f}}
 = q^{(\varepsilon_{l},\alpha_{i})}e_{i}, \qquad
 q^{H_{l}^{f}}q^{H_{k}^{e}} = q^{H_{k}^{e}}q^{H_{l}^{f}}.
\end{gather*}
To identify the Cartan elements from both factors, we shall quotient the double $\CD_{\DrJ}$ by the two-sided ideal
\begin{gather*}
 \CJ \coloneqq \big\langle q^{\pm H_{k}^{e}} - q^{\pm H_{k}^{f}} \mid 1 \leq k \leq N \big\rangle.
\end{gather*}
Passing to the quotient $\CD_{\DrJ}/\CJ$ precisely recovers the commutation relations
\eqref{eq:q-gl-chevalley-rel-HH}--\eqref{eq:q-gl-chevalley-rel-ef} between the Borel subalgebras $U^\leq_{q}(\gl(V))$
and $U^\geq_{q}(\gl(V))$, which concludes the proof that $\CD_{\DrJ}/\CJ \simeq \uqgl$.
\end{proof}

\begin{Remark}\label{rem:DrJ pairing order}
By Proposition~\ref{prop:opposite double}, one may alternatively use the transposed inverse
\begin{gather*}
 (-,-)_{\widetilde{\DrJ}} \colon\ U^\geq_{q}(\gl(V)) \times U^\leq_{q}(\gl(V)) \to \BC(q)
\end{gather*}
of $(-,-)_{\DrJ}$ in the double construction, which amounts to taking the Borel subalgebras in the opposite order.
The latter order is necessary to align $(-,-)_{\widetilde{\DrJ}}$ with the corresponding
skew-pairing~\eqref{eq:gl-RTT-skew-pairing} of the $\RLL$-realization $U(R)$ defined below,
cf.\ Remark~\ref{rem:RTT double full Cartan}. Specifically, the isomorphism $\xi$ from Theorem~\ref{thm:finite DrJ to RTT}
maps the subbialgebras $U_{q}^{\geq}(\gl(V))$ and $U_{q}^{\leq}(\gl(V))$ onto $U^{\geq}(R)$ and $U^{\leq}(R)$, respectively.
Because both Borel subalgebras are Hopf superalgebras with invertible antipodes, \eqref{eq:inverse-pairing} ensures that
the transposed inverse pairing is also non-degenerate, and its evaluation on the generators is given as follows:
\begin{gather*}
 \bigl(e_i, q^{\pm H_k}\bigr)_{\widetilde{\DrJ}} = 0, \qquad \bigl(q^{\pm H_k}, f_i\bigr)_{\widetilde{\DrJ}} = 0,\\
 (e_i, f_j)_{\widetilde{\DrJ}} = \frac{\delta_{ij}}{q - q^{-1}},\qquad
 \bigl(q^{H_k}, q^{H_l}\bigr)_{\widetilde{\DrJ}} = q^{(\varepsilon_k, \varepsilon_l)},
\end{gather*}
for all $1 \leq i,j < N$ and $1 \leq k,l \leq N$.
\end{Remark}

\begin{Remark}\label{rem:DrJ double half Cartan}
Instead of the Hopf subalgebras $U^{\leq}_{q}(\gl(V))$ and $U^{\geq}_{q}(\gl(V))$, the double construction can
also be carried out using the subbialgebras $U^{-}_{q}(\gl(V))$ and $U^{+}_{q}(\gl(V))$ generated by
\smash{$\{f_i\}_{i=1}^{N-1} \cup \{q^{-H_{k}}\}_{k=1}^{N}$} and \smash{$\{e_i\}_{i=1}^{N-1} \cup \{q^{H_{k}}\}_{k=1}^{N}$} respectively,
equipped with the restriction of the exact same skew-pairing.
\end{Remark}

\begin{Remark}\label{rem:DrJ involution subalgebras}
Invoking the Hopf superalgebra isomorphism $\omega_{\DrJ}$ of Remark~\ref{rem:DrJ involution}, we note that it
restricts to superalgebra isomorphisms between $U^{\leq}_{q}(\gl(V))$ and $U^{\geq}_{q}(\gl(V))$ and also between~$U^{-}_{q}(\gl(V))$ and $U^{+}_{q}(\gl(V))$.
\end{Remark}

\subsection[Drinfeld twist of U\_q(gl(V))]{Drinfeld twist of $\boldsymbol{\uqgl}$}\label{ssec:drinfeld-twist}

For the later use in Section~\ref{ssec:homomorphism thm}, we shall now introduce a~specific Drinfeld twist of $\uqgl$.
Recall that a~\emph{Drinfeld twist} (or gauge transformation) of a~Hopf superalgebra $H$ is an invertible element
$\CF \in H \otimes H$ satisfying the 2-cocycle condition (cf.~\cite[equation~(4.3)]{xz}; for the non-super setup, see~\cite[Section~4.2]{cp})
\begin{gather*}
 (\CF \otimes 1)(\Delta \otimes \Id)(\CF) = (1 \otimes \CF)(\Id \otimes \Delta)(\CF),
\end{gather*}
and the counitality condition $(\epsilon \otimes \Id)(\CF) = 1 = (\Id \otimes \epsilon)(\CF)$. Given such an element,
one can construct a~new Hopf superalgebra $H^{\CF}$ whose underlying superalgebra structure and the counit morphism are
identical to those of $H$, but with the comultiplication conjugated by $\CF$
\begin{gather*}
 \Delta^{\CF}(x) = \CF \Delta(x) \CF^{-1} \qquad \text{for all} \ x \in H,
\end{gather*}
and its antipode $S^{\CF}$ modified correspondingly.

We now apply the above construction to $H = \uqgl$ using
$\CF = \fR_{s} \coloneqq q^{-\sum_{i=1}^{N} (-1)^{\overline{i}} H_i \otimes H_i}$,
see~\eqref{eq:gl-universal-Rs Ru}. Since $\fR_{s}$ is the canonical element of the Drinfeld double
$\CD\bigl(U_q^{0}(\gl(V)),U_q^{0}(\gl(V))\bigr)$ (cf.\ Remark~\ref{rem:U0-canonical}), it naturally satisfies the conditions
of a~Drinfeld twist by Proposition~\ref{prop:R properties}. A~direct $\hbar$-adic computation
(cf.\ Remark~\ref{rem:h-adic-completion-avoidance}) yields the following
twisted comultiplication $\Delta^{\CF}$ on the generators:
\begin{gather*}
 \Delta^{\CF}(e_i) = 1 \otimes e_i + e_i \otimes q^{-h_i},\qquad
 \Delta^{\CF}(f_i) = q^{h_i} \otimes f_i + f_i \otimes 1,\\
 \Delta^{\CF}\bigl(q^{\pm H_k}\bigr) = q^{\pm H_k} \otimes q^{\pm H_k}
\end{gather*}
with the antipode determined by
$S^{\CF}(e_i) = -e_i q^{h_i}$,
$S^{\CF}(f_i) = -q^{-h_i} f_i$,
$S^{\CF}\bigl(q^{\pm H_k}\bigr) = q^{\mp H_k}$.

In general, a~Hopf superalgebra $H$ and its Drinfeld twist $H^{\CF}$ need not be isomorphic as Hopf superalgebras.
However, there is a~simple Hopf superalgebra isomorphism between $\uqgl$ and its twist $\uqgl^{\CF}$.
Let $U^{\geq}_q(\gl(V))^{\CF}$ and $U^{\leq}_q(\gl(V))^{\CF}$ denote the positive and negative Borel subalgebras
of \smash{$\uqgl^{\CF}$} generated by \smash{$\{e_i\}_{i=1}^{N-1} \cup \{q^{\pm H_k}\}_{k=1}^{N}$} and
\smash{$\{f_i\}_{i=1}^{N-1} \cup \{q^{\pm H_k}\}_{k=1}^{N}$}, respectively. Being closed under $\Delta^\CF$ and $S^{\CF}$,
they are actually Hopf subalgebras. We emphasize that they share the exact same underlying superalgebra structure as
their untwisted counterparts~$U^{\geq}_q(\gl(V))$ and $U^{\leq}_q(\gl(V))$, differing only in their supercoalgebra structure.

\begin{Proposition}\label{prop:drinfeld-twist-iso}\quad
\begin{itemize}\itemsep=0pt
\item[$(a)$] The assignment
\begin{gather}
 \phi_{\CF}(e_i) = e_i q^{-h_i}, \qquad \phi_{\CF}(f_i) = q^{h_i} f_i, \qquad \phi_{\CF}\bigl(q^{\pm H_k}\bigr) = q^{\pm H_k} \label{eq:twist-identification}
\end{gather}
for $1 \le i < N$ and $1 \le k \le N$, uniquely extends to a~$Q$-graded Hopf superalgebra isomorphism
\[\phi_{\CF} \colon\ \uqgl^{\CF} \iso \uqgl.\]

\item[$(b)$] The isomorphism $\phi_{\CF}$ preserves the Borel subalgebras $U^{\geq}_q(\gl(V))$ and $U^{\leq}_q(\gl(V))$.

\item[$(c)$] Define the skew-pairing $(-,-)^{\CF}_{\DrJ}$ on the Drinfeld twist via
$(x,y)^{\CF}_{\DrJ} \coloneqq (\phi_{\CF}(x), \phi_{\CF}(y))_{\DrJ}$. The skew-pairing $(-,-)^{\CF}_{\DrJ}$ is also
non-degenerate and evaluates identically to the skew-pairing $(-,-)_{\DrJ}$ on the generators~\eqref{eq:generators-pairing}.
\end{itemize}
\end{Proposition}

\begin{Remark}
As an immediate consequence, $\uqgl^{\CF}$ can also be realized as the Drinfeld double $\CD^{\CF}_{\DrJ}$ of
$U^{\leq}_q(\gl(V))^{\CF}$ and $U^{\geq}_q(\gl(V))^{\CF}$ with respect to the skew-pairing $(-,-)^{\CF}_{\DrJ}$,
modulo the identification of the Cartan subalgebras.
\end{Remark}

\begin{Remark}\label{rem:wt-pairing-Dr-twist}
The same argument shows that Proposition~\ref{prop:drinfeld-twist-iso}\,(c) remains valid when using~the alternate skew-pairing
\smash{$(-,-)_{\widetilde{\DrJ}}$} from Remark~\ref{rem:DrJ pairing order}. Hence, the same Drinfeld double $\CD^{\CF}_{\DrJ}$ can be
realized with respect to the pullback \smash{$(-,-)^{\CF}_{\widetilde{\DrJ}}$} of \smash{$(-,-)_{\widetilde{\DrJ}}$} via $\phi_{\CF}$, using the opposite order of the Borel
subalgebras $U^{\geq}_q(\gl(V))^{\CF}$ and~$U^{\leq}_q(\gl(V))^{\CF}$.
\end{Remark}

\begin{proof}[Proof of Proposition~\ref{prop:drinfeld-twist-iso}]
(a) We first check that the assignment~\eqref{eq:twist-identification} gives rise to a~superalgebra morphism.
As the underlying superalgebra structure of $\uqgl^{\CF}$ is identical to that of $\uqgl$, we need to show
that~\eqref{eq:twist-identification} defines a~superalgebra endomorphism on $\uqgl$. The Chevalley-type relations
are straightforward, e.g., \eqref{eq:q-gamma-chevalley-rel-ef} is verified as follows:
\begin{align*}
 [\phi_{\CF}(e_i), \phi_{\CF}(f_j)]
 &= e_i q^{-h_i+h_j} f_j - (-1)^{|e_i||f_j|} q^{h_j} f_j e_i q^{-h_i} \\
 &= q^{(\alpha_{i}-\alpha_{j},\alpha_{j})} [e_i, f_j] q^{-h_i+h_j}
 = \delta_{ij} \frac{\phi_{\CF}\bigl(q^{h_i}\bigr)-\phi_{\CF}\bigl(q^{-h_i}\bigr)}{q-q^{-1}}.
\end{align*}
For the Serre relations, we note that $\phi_{\CF}(e_i) \phi_{\CF}(e_j) = q^{-(\alpha_i,\alpha_j)} \cdot e_i e_j q^{-h_i-h_j}$,
which implies that their $q$-supercommutator evaluates to
 $[\![\phi_{\CF}(e_i), \phi_{\CF}(e_j)]\!] = q^{-(\alpha_i,\alpha_j)} \cdot [\![e_i, e_j]\!] q^{-h_i-h_j}$.
By induction, any nested $q$-supercommutator evaluated on the images $\phi_{\CF}(e_i)$ is simply proportional to the
same commutator evaluated on the generators $e_i$, right-multiplied by a~Cartan element uniquely determined by its total
$Q$-degree. As the Serre relations vanish on the generators $e_i$ in $\uqgl$, they identically vanish on the images
$\phi_{\CF}(e_i)$.

The inverse morphism is given by the assignment\vspace{-1mm}
\[
\phi_{\CF}^{-1}(e_i) = e_i q^{h_i}, \qquad \phi_{\CF}^{-1}(f_i) = q^{-h_i} f_i, \qquad \phi_{\CF}^{-1}\bigl(q^{\pm H_k}\bigr) = q^{\pm H_k},\vspace{-1mm}
\]
which can be shown to define a~valid superalgebra homomorphism by a~completely analogous argument. Furthermore,
the compatibility of the coalgebra structures $(\phi_{\CF} \otimes \phi_{\CF}) \circ \Delta^{\CF} = \Delta \circ \phi_{\CF}$
is directly checked on the generators. Since $\phi_{\CF}$ clearly preserves the $Q$-grading, this proves part~(a).

(b) This is clear from the definitions, as
$\phi_{\CF}(e_i) = e_i q^{-h_i} \in U^{\ge}_q(\gl(V))$ and
$\phi_{\CF}(f_i) = q^{h_i} f_i \in U^{\le}_q(\gl(V))$.

(c) We first emphasize that $(-,-)^{\CF}_{\DrJ}$ is indeed a~skew-pairing by part (a).
The evaluation of~$(f_{i},e_{i})^{\CF}_{\DrJ}$ is the only nontrivial case among the generators
in~\eqref{eq:generators-pairing}, the computation of which is done as follows:\vspace{-1mm}
\begin{align*}
 (f_{i},e_{i})^{\CF}_{\DrJ}
 &= \bigl(q^{h_{i}}f_{i}, e_{i}q^{-h_{i}}\bigr)_{\DrJ}
 = q^{-(\alpha_{i},\alpha_{i})} \bigl(f_{i}q^{h_{i}}, e_{i}q^{-h_{i}}\bigr)_{\DrJ} \\
\overset{\eqref{eq:DrJ-skew-pairing-identity}}&{=} q^{-(\alpha_{i},\alpha_{i})} (f_{i},e_{i})_{\DrJ} \bigl(q^{h_{i}},q^{-h_{i}}\bigr)_{\DrJ}
 = (f_{i},e_{i})_{\DrJ}. \tag*{\qed}\vspace{-1mm}
\end{align*}
\renewcommand{\qed}{}
\end{proof}

\section[\RLL-realization of general linear quantum supergroups]{$\boldsymbol{\RLL}$-realization of general linear quantum supergroups}\label{sec:RTT_finite}

\subsection[The universal R-matrix]{The universal $\boldsymbol{R}$-matrix}\label{ssec:universal R}
 A~fundamental property of the Drinfeld double of Hopf superalgebras is that it naturally carries the
structure of a~(topological) quasi-triangular Hopf superalgebra via its canonical element~\eqref{eq:canonical element},
see Remark~\ref{rem:h-adic-completion-avoidance} for more details. By Remark~\ref{rem:triangular-decomp-morphisms} and
the identity~\eqref{eq:DrJ-skew-pairing-identity}, the skew-pairing~\eqref{eq:gl-skew-pairing} factors multiplicatively into its
restrictions to $U^{<}_{q}(\gl(V)) \times U^{>}_{q}(\gl(V))$ and~\smash{$U^{0}_{q}(\gl(V)) \times U^{0}_{q}(\gl(V))$}. Let \smash{$\bigl\{u^{f}_{i}\bigr\}_{i}$}
and $\{u^{e}_{i}\}_{i}$ be dual bases for \smash{$U^{<}_{q}(\gl(V))$} and \smash{$U^{>}_{q}(\gl(V))$} with respect to this restricted pairing. Analogously,
if \smash{$\{s^{f}_{j}\}_{j}$} and $\{s^{e}_{j}\}_{j}$ were dual bases for the respective copies of \smash{$U^{0}_{q}(\gl(V))$}, then
the products \smash{$\bigl\{u^{f}_{i} s^{f}_{j}\bigr\}_{i,j}$} and $\{u^{e}_{i} s^{e}_{j}\}_{i,j}$ would form dual bases for \smash{$U^{\leq}_{q}(\gl(V))$} and
\smash{$U^{\geq}_{q}(\gl(V))$}. Therefore, the canonical element $\fR$ factors as $\fR = \fR_{u}\fR_{s}$, where
(note that every element in $U^{0}_{q}(\gl(V))$ has $\BZ_2$-degree $\bbar{0}$)\vspace{-1mm}
\begin{gather*}
 \fR_{u} = \sum_{i} u^{e}_{i} \otimes u^{f}_{i},\qquad
 \fR_{s} = \sum_{j} s^{e}_{j} \otimes s^{f}_{j}.\vspace{-2mm}
\end{gather*}

\begin{Remark}
Motivated by their actions on finite-dimensional weight modules, the above factors~$\fR_{s}$ and $\fR_{u}$ are often referred to
as the \emph{semisimple} and \emph{unipotent} parts of the canonical element~$\fR$. Furthermore, the unipotent part~$\fR_{u}$ is
also often called the \emph{reduced} $R$-matrix.
\end{Remark}

\begin{Remark}\label{rem:U0-canonical}

Since $U^{0}_{q}(\gl(V))$ is in fact a~Hopf subalgebra, the restriction of the skew-pair\-ing~\eqref{eq:gl-skew-pairing} yields a~Drinfeld double $\CD\bigl(U^{0}_{q}(\gl(V)), U^{0}_{q}(\gl(V))\bigr)$ with $\fR_{s}$ being the corresponding canonical element.\vspace{-1mm}
\end{Remark}

\begin{Remark}\label{rem:h-adic-completion-avoidance}\quad
\begin{itemize}\itemsep=0pt
\item[(a)] To treat the canonical element $\fR = \fR_{u}\fR_{s}$ rigorously, one must address the fact that its expansion involves infinite
sums of tensor products, which are not well-defined elements in the algebraic tensor product $\uqgl \otimes \uqgl$. While the infinite
sum in $\fR_{u}$ can be resolved quite simply, $\fR_{s}$ presents a~more fundamental algebraic obstruction: the dual bases for the
Cartan subalgebras cannot be constructed in $\uqgl$, as expressing them requires taking formal logarithms of the generators $q^{H}$.
To resolve this and make sense of these elements, one must pass to the \emph{$\hbar$-adic topological quantum supergroups} $\uhgl$,
defined over $\BC[[\hbar]]$, and consider the completed tensor product $\uhgl \hat{\otimes} \uhgl$. In this framework, the
indeterminate $q \in \BC(q)$ and the Cartan elements $q^{H}$ (for $H \in \Gamma$) are treated as formal exponentials
\smash{$q = e^{\hbar} \coloneqq \sum_{k \ge 0} \hbar^{k}/k! \in \BC[[\hbar]]$} and
$q^{H} = e^{\hbar H} \coloneqq \sum_{k \ge 0} \hbar^{k} H^{k}/k!$. We refer the reader to~\cite[Section~8.3.2]{kas}
or~\cite[Section~17]{ks} for a~systematic development of this topological setting in the non-super case.

\item[(b)] However, $\fR_{s}$ is rarely necessary in practice: it generally suffices to work with its evaluations on finite-dimensional
weight modules. By adapting the approach of~\cite[Section~7]{jan} to the super setting (as in~\cite[Section~4.1]{ht1}), one can evaluate these
matrices directly within the rational $\BC(q)$-setup, entirely bypassing topological completions.

\item[(c)] In the present paper, the only instance explicitly requiring the universal element $\fR$ is the construction of the isomorphism $\phi^{\DF}$ in Sections~\ref{ssec:inverse morphism} and~\ref{ssec:inverse morphism osp}, which serves as the inverse to the isomorphism $\xi$ constructed in Sections~\ref{ssec:homomorphism thm} and~\ref{ssec:osp-homomorphism thm}, up to the canonical automorphisms (see Propositions~\ref{prop:almost-inverse-gl} and~\ref{prop:almost-inverse-osp}).
We emphasize that the isomorphism $\xi$ between the Drinfeld--Jimbo and $\RLL$ realizations
is unequivocally established via the generalized double construction. Therefore, if one only requires the existence of the
isomorphism without an explicit formula for the inverse morphism $\phi^{\DF}$, the issue of $\hbar$-adic completions can be completely circumvented.\footnote{Furthermore, even when constructing the explicit inverse morphism, the use of the semisimple part $\fR_s$ (and the associated $\hbar$-adic topology) can be completely bypassed by following the arguments in~\cite[Section~3]{rw}.}
\end{itemize}
\end{Remark}

To obtain an explicit formula for $\fR$ (see~\cite[Theorem 5.16 and Appendix A]{ht1} for the derivation through the
combinatorics of shuffle superalgebras, while the $\hbar$-adic version of this result goes back to~\cite[Theorem 10.5.1]{y0}),
we first recall the PBW-type bases of $U^>_q(\gl(V))$ and $U^<_q(\gl(V))$.
For convenience, let $\gamma_{ij} \coloneqq \varepsilon_{i} - \varepsilon_{j}$ for all $1 \leq i<j \leq N$, so that
\[
 \Phi^{+} = \{\gamma_{ij} \mid 1 \leq i < j \leq N\}.
\]
For each $\gamma_{ij} \in \Phi^{+}$, we define the corresponding \emph{quantum root vectors} $e_{\gamma_{ij}}$ and
$f_{\gamma_{ij}}$ inductively on $j-i$. The base cases are $e_{\gamma_{i,i+1}} \coloneqq e_i$ and
$f_{\gamma_{i,i+1}} \coloneqq f_i$. For $j > i$, we set
\begin{gather}
 e_{\gamma_{i,j+1}} \coloneqq
 e_{\gamma_{ij}}e_{j} - (-1)^{|e_{\gamma_{ij}}||e_{j}|} q^{(\gamma_{ij},\alpha_{j})} e_{j}e_{\gamma_{ij}}
 = [\![ e_{\gamma_{ij}}, e_{j} ]\!],\nonumber\\
 f_{\gamma_{i,j+1}} \coloneqq
 f_{j}f_{\gamma_{ij}} - (-1)^{|f_{\gamma_{ij}}||f_{j}|} q^{-(\gamma_{ij},\alpha_{j})} f_{\gamma_{ij}}f_{j}
 = - (-1)^{|f_{\gamma_{ij}}||f_{j}|} q^{-(\gamma_{ij},\alpha_{j})} [\![ f_{\gamma_{ij}}, f_{j} ]\!].\label{eq:quantum root vectors}
\end{gather}
We also consider the \emph{lexicographic order} on $\Phi^{+}$, corresponding to the lexicographic order on pairs
$\{(i,j)\}_{1 \leq i<j \leq N}$
\begin{gather}\label{eq:gl-convex-order}
 \gamma_{12} \prec \gamma_{13} \prec \cdots \prec \gamma_{1N} \prec \gamma_{23} \prec \cdots \prec \gamma_{2N} \prec \cdots
 \prec \gamma_{N-2,N-1} \prec \gamma_{N-2,N} \prec \gamma_{N-1,N}.
\end{gather}

\begin{Theorem}\label{thm:gl-PBW-general}\quad\samepage
\begin{itemize}\itemsep=0pt
\item[$(a)$] The sets of ordered monomials
\begin{gather*}
 \left\{ \prod_{\gamma \in \Phi^{+}}^{\leftarrow} e_{\gamma}^{m_{\gamma}} \;\Bigg|\;
 \substack{m_{\gamma} \in \BZ_{\ge 0} \\ m_\gamma \leq 1 \text{ if } \gamma \in \Phi_{\bbar{1}} } \right\}
 \qquad \text{and} \qquad \left\{ \prod_{\gamma \in \Phi^{+}}^{\leftarrow} f_{\gamma}^{m_{\gamma}} \;\Bigg|\;
 \substack{m_{\gamma} \in \BZ_{\ge 0} \\ m_\gamma \leq 1 \text{ if } \gamma \in \Phi_{\bbar{1}} } \right\}
\end{gather*}
form bases for $U^>_q(\gl(V))$ and $U^<_q(\gl(V))$, respectively.
Here, the arrow $\leftarrow$ over the product signs refers to the lexicographic order~\eqref{eq:gl-convex-order} on~$\Phi^+$, thus ordering the elements of $\Phi^+$ in decreasing order.

\item[$(b)$] The skew-pairing~\eqref{eq:gl-skew-pairing} is orthogonal with respect to these bases.
\end{itemize}
\end{Theorem}

\pagebreak

Using the PBW bases above, one can express the universal $R$-matrix $\fR$ in the completion
of the tensor square of $\uqgl$:
\[
 \fR = \fR_u \fR_s \in \uqgl \hat{\otimes} \uqgl,
\]
where (cf.~\cite[Theorem~5.18]{ht1})
\begin{gather}\label{eq:gl-universal-Rs Ru}
 \fR_s = q^{-\sum_{i=1}^{N} (-1)^{\overline{i}} H_i \otimes H_i}, \qquad
 \fR_u = \prod_{\gamma \in \Phi^{+}}^{\leftarrow}
 \left( \sum^{k\geq 0}_{k \leq 1 \text{ if } \gamma \in \Phi_{\bbar{1}} }
 \frac{e_\gamma^k \otimes f_\gamma^k}{\bigl(f_\gamma^k, e_\gamma^k\bigr)_{\DrJ}} \right).
\end{gather}
Let us now evaluate $\fR$ on the \emph{first fundamental representation}
$\varrho \colon \uqgl \to \End(V_{q})$. Here, $V_{q} \coloneqq \BC(q) \otimes_{\BC} V$
is the $\BC(q)$-superspace obtained by extension of scalars, and $\End(V_{q})$ denotes the $\BC(q)$-superalgebra of all
$\BC(q)$-linear maps. The representation $\varrho$ is defined via (cf.~\cite[Proposition~A.1]{ht1}):
\[
 \varrho(e_i) = \sse_i, \qquad
 \varrho(f_i) = \ssf_i, \qquad
 \varrho\bigl(q^{\pm H_k}\bigr) = q^{\pm \ssH_k},
 \]
for $1 \leq i < N$ and $1 \leq k \leq N$, where $\sse_i$, $\ssf_i$, $\ssH_k$ are the Chevalley-type generators of $\gl(V)$
from~\eqref{eq:gl-Lie-action}, \eqref{eq:gl-Lie-action-Cartan}. Henceforth, for a~diagonal matrix
$D = \diag(d_{1}, \ldots, d_{N})$, the notation $q^D$ signifies
\begin{gather}\label{eq:q-D}
 q^D \coloneqq \sum_{1\leq i\leq N} q^{d_i} E_{ii} \in \End(V_{q}).
\end{gather}

The evaluation of the semisimple part $\varrho^{\otimes 2}(\fR_s)$ follows from a~direct computation
(interpreting~${q = e^{\hbar}}$ in the $\hbar$-adic setting), while the evaluation of the unipotent part
$\varrho^{\otimes 2}(\fR_u)$ is established in~\cite[Proposition~A.15]{ht1}:
\begin{gather*}
 R_{s} \coloneqq \varrho^{\otimes 2}(\fR_s) =
 \sum_{1\leq i,j\leq N} q^{-(\varepsilon_{i},\varepsilon_{j})} E_{ii} \otimes E_{jj},\\
 R_{u} \coloneqq \varrho^{\otimes 2}(\fR_u) = \ID - \bigl(q-q^{-1}\bigr) \sum_{i<j} (-1)^{\overline{j}} E_{ij} \otimes E_{ji}.
\end{gather*}
Combining these, we obtain the explicit formula for $R \coloneqq \varrho^{\otimes 2}(\fR)$ and its inverse $R^{-1}$
in $\End(V_{q})^{\otimes 2}$ as presented in~\cite[Remark A.5]{ht1},
\begin{gather}
 R = \sum_{1\leq i,j\leq N} q^{-(\varepsilon_{i},\varepsilon_{j})} E_{ii} \otimes E_{jj}
 - \bigl(q-q^{-1}\bigr) \sum_{i<j} (-1)^{\overline{j}} E_{ij} \otimes E_{ji},\nonumber\\
 R^{-1} = \sum_{1\leq i,j\leq N} q^{(\varepsilon_{i},\varepsilon_{j})} E_{ii} \otimes E_{jj}
 + \bigl(q-q^{-1}\bigr) \sum_{i<j} (-1)^{\overline{j}} E_{ij} \otimes E_{ji}.\label{eq:gl-evaluated-R}
\end{gather}

\begin{Remark}\label{rem:YBE}
We note that $\fR$ satisfies the \emph{Yang--Baxter equation} (cf. Proposition~\ref{prop:R properties})
\begin{equation}\label{eq:YBE}
 \fR_{(12)}\fR_{(13)}\fR_{(23)} = \fR_{(23)}\fR_{(13)}\fR_{(12)} \in \uqgl^{\hat{\otimes} 3} ,
\end{equation}
and the following comultiplication properties:
\begin{equation}\label{eq:R comult}
 (\Delta \otimes \Id)\fR = \fR_{(13)}\fR_{(23)} ,\qquad
 (\Id \otimes \Delta)\fR = \fR_{(13)}\fR_{(12)} ,
\end{equation}
where the indices indicate the tensor factors on which $\fR$ acts (e.g.,
$\fR_{(12)} = \fR \otimes 1 \in \uqgl^{\hat{\otimes} 3}$, with analogous definitions for $\fR_{(13)}$ and $\fR_{(23)}$).
This leg-numbering notation will be formalized in the next subsection.
It follows that the evaluation $R$ also satisfies the Yang--Baxter equation $R_{12}R_{13}R_{23} = R_{23}R_{13}R_{12}$
in $\End(V_{q})^{\otimes 3}$. The multiplication in $\uqgl^{\hat{\otimes} 3}$ follows the standard sign convention for
the tensor product of superalgebras, cf.\ \eqref{eq:superalgebras tensoring}. Moreover, the operator
\begin{equation*}
 \hat{R} \coloneqq \tau_{VV} \circ R \in \End(V_{q})^{\otimes 2}
\end{equation*}
is a $\uqgl$-module homomorphism from $V_{q}^{\otimes 2}$ to itself, and satisfies the \emph{braid relation}
\begin{equation*}
 \hat{R}_{12}\hat{R}_{23}\hat{R}_{12} = \hat{R}_{23}\hat{R}_{12}\hat{R}_{23}\in
 \End(V_{q})^{\otimes 3},
\end{equation*}
where the indices follow the same leg-numbering conventions.
\end{Remark}

\subsection{Leg-numbering notation}\label{ssec:leg-numbering}

Let $\CA$ be a~$\BC(q)$-superalgebra (in many cases, $\CA$ will be the free superalgebra $\CT$ defined
in~\eqref{eq:gl-rtt-free superalgebra} below, or one of its subquotients). We will frequently encounter superalgebras
of the form $\CA^{\otimes r} \otimes \End(V_{q})^{\otimes s}$ $(r,s \geq 0)$ and morphisms between them.
To keep equations involving these tensor products concise, we employ the standard leg-numbering notation (cf.~\cite[Section~1.5]{mol}).

For strictly increasing sequences of indices $1 \leq a_1 < \dots < a_r \leq r'$ and $1 \leq b_1 < \dots < b_s \leq s'$,
there is a~canonical embedding
\[
 \iota_{(a_1,\ldots,a_r),b_1,\ldots,b_s} \colon \
 \CA^{\otimes r} \otimes \End(V_{q})^{\otimes s} \hookrightarrow \CA^{\otimes r'} \otimes \End(V_{q})^{\otimes s'}
\]
that allocates the respective tensor arguments into the target spaces specified by the $a_i$ and $b_j$ positions.
For arbitrary sequences of distinct indices, we uniquely extend this ordered embedding by precomposing it with the braiding isomorphism $\tau^{\sigma} \otimes \tau^{\sigma'}$ corresponding to the required $\sigma \in S_{r}$ and $\sigma' \in S_{s}$,
see Remark~\ref{rem:tensor braid action}. We note that due to the use of the braiding $\tau^{\sigma} \otimes \tau^{\sigma'}$,
evaluating these embeddings introduces Koszul signs.

For any sequences of distinct indices $(a_1, \ldots, a_r)$, $(b_1, \ldots, b_s)$, the image of an element
$M \in \CA^{\otimes r} \otimes \End(V_{q})^{\otimes s}$ under the above embedding is denoted by
\begin{gather*}
 M_{(a_1,\ldots,a_r),b_1,\ldots,b_s} \coloneqq
 \iota_{(a_1,\ldots,a_r),b_1,\ldots,b_s}(M) \in \CA^{\otimes r'} \otimes \End(V_{q})^{\otimes s'}.
\end{gather*}
In practice, the integers $r'$ and $s'$ will be clear from the context. When
no confusion can arise, the corresponding subscript is omitted. We note that this convention is compatible with our earlier
notation in~\eqref{eq:YBE}; for instance, the elements $R_{12}$, $R_{13}$, $R_{23}$ are precisely the images of~${R \in \End(V_{q})^{\otimes 2}}$ under the respective embeddings $\iota_{12}$, $\iota_{13}$, $\iota_{23}$ into
$\End(V_{q})^{\otimes 3}$.

This convention extends analogously to the following morphisms:
\[ {\phi \colon \ \CA^{\otimes k} \to \CA^{\otimes k'}} \qquad \text{and}\qquad
\psi \colon \ \End(V_{q})^{\otimes l} \to \End(V_{q})^{\otimes l'}.
\] For any integers $r$ and $s$ large enough to contain the consecutive index sequences $(a, \ldots, a+k-1)$ and $(b, \ldots, b+l-1)$ respectively, we define
\begin{gather*}
 \phi_{(a,\ldots,a+k-1)} \coloneqq
 \bigl(\Id^{\otimes (a-1)} \otimes \phi \otimes \Id^{\otimes (r-a-k+1)}\bigr) \otimes \Id^{\otimes s}
 \colon
\\
\phantom{ \phi_{(a,\ldots,a+k-1)} \coloneqq}{}\qquad\CA^{\otimes r} \otimes \End(V_{q})^{\otimes s} \to \CA^{\otimes (r-k+k')} \otimes \End(V_{q})^{\otimes s},\\
 \psi_{b,\ldots,b+l-1} \coloneqq
 \Id^{\otimes r} \otimes \bigl(\Id^{\otimes (b-1)} \otimes \psi \otimes \Id^{\otimes (s-b-l+1)}\bigr)\colon\\
 \phantom{ \psi_{b,\ldots,b+l-1} \coloneqq}{}\qquad
 \CA^{\otimes r} \otimes \End(V_{q})^{\otimes s} \to \CA^{\otimes r} \otimes \End(V_{q})^{\otimes (s-l+l')}.
\end{gather*}
Again, these subscripts will frequently be omitted when the context is clear and no confusion can arise.

\subsection[The superalgebra U(R)]{The superalgebra $\boldsymbol{U(R)}$}

Following the construction in~\cite{frt}, we now define the $\RLL$-realization $U(R)$ of the general linear quantum supergroup,
corresponding to the matrix $R$:

\begin{Definition}\label{def:rll-algebra}
The $\BC(q)$-superalgebra $U(R)$ is generated by the elements $\{l^{+}_{ij}, l^{-}_{ji}\}_{1 \leq i \leq j \leq N}$
subject to the relations~\eqref{eq:gl-rtt-diag-rel}--\eqref{eq:gl-rtt-rel-2} below. The $\BZ_{2}$-grading is
compatible via~\eqref{eq:gl-grading-compatible} with the $Q$-grading determined~by
\begin{gather}\label{eq:uR-Q-grading}
 \deg\bigl(l^{\pm}_{ij}\bigr) = \varepsilon_{i} - \varepsilon_{j}.
\end{gather}
\end{Definition}

To rigorously formulate the defining relations, we first consider the free $\BC(q)$-superalgebra
\begin{gather}\label{eq:gl-rtt-free superalgebra}
 \CT \coloneqq \BC(q) \big\langle l^{+}_{ij}, l^{-}_{ji} \mid 1 \leq i \leq j \leq N \big\rangle,
\end{gather}
equipped with the same $Q$-grading and $\BZ_{2}$-grading as in~\eqref{eq:uR-Q-grading}. By adopting the convention that~${l^{+}_{ij} = 0 = l^{-}_{ji}}$ for $i>j$, we now organize these generators into formal matrices
\begin{gather*}
 L^{\pm} \coloneqq \sum_{1\leq i,j\leq N} l^{\pm}_{ij} \otimes E_{ij} \in \CT \otimes \End(V_{q}).
\end{gather*}

\noindent
Following notations of Section~\ref{ssec:leg-numbering}, the algebra $U(R)$ is the quotient of $\CT$ by
the following diagonal relations
\begin{gather}\label{eq:gl-rtt-diag-rel}
 l^{\pm}_{ii}l^{\mp}_{ii} = 1 \qquad \text{for } 1 \leq i \leq N,
\end{gather}
together with $RLL$-relations, which is a family of all matrix coefficients in~$\CT$ of the following matrix equations, evaluated in the superalgebra
$\CT \otimes \End(V_{q})^{\otimes 2}$:
\begin{gather}
 R_{12}L^{\pm}_{1}L^{\pm}_{2} = L^{\pm}_{2}L^{\pm}_{1}R_{12}, \label{eq:gl-rtt-rel-1}\\
 R_{12}L^{+}_{1}L^{-}_{2} = L^{-}_{2}L^{+}_{1}R_{12}. \label{eq:gl-rtt-rel-2}
\end{gather}

\begin{Remark}\quad
\begin{itemize}\itemsep=0pt
\item[(a)] As the relations~\eqref{eq:gl-rtt-diag-rel}--\eqref{eq:gl-rtt-rel-2} are $Q$-homogeneous, \eqref{eq:uR-Q-grading}
indeed defines a~$Q$-grading on the superalgebra $U(R)$.

\item[(b)] As $L^{+}$ and $L^{-}$ are upper and lower triangular matrices, respectively, and their diagonal entries are
invertible by~\eqref{eq:gl-rtt-diag-rel}, they are invertible $U(R)$-valued matrices.
\end{itemize}
\end{Remark}

Furthermore, $U(R)$ admits the structure of a~Hopf superalgebra, with the comultiplication and counit given by
\begin{gather}\label{eq:gl-rtt-coalg}
 \Delta\bigl(L^{\pm}\bigr) = L^{\pm}_{(1)} L^{\pm}_{(2)} \in U(R)^{\otimes 2} \otimes \End(V_{q}),\qquad
 \epsilon\bigl(L^{\pm}\bigr) = \ID \in \End(V_{q}),
\end{gather}
and the antipode uniquely determined by
\begin{gather}\label{eq:gl-rtt-antipode}
 S\bigl(L^{\pm}\bigr) = \bigl(L^{\pm}\bigr)^{-1} \in U(R) \otimes \End(V_{q}).
\end{gather}
Component-wise~\eqref{eq:gl-rtt-coalg} reads as
\begin{gather}\label{eq:gl-rtt-coalg-component}
 \Delta\bigl(l^{\pm}_{ij}\bigr) = \sum_{1\leq k\leq N} (-1)^{(\overline{i}+\overline{k})(\overline{k}+\overline{j})} l^{\pm}_{ik} \otimes l^{\pm}_{kj},
 \qquad \epsilon\bigl(l^{\pm}_{ij}\bigr) = \delta_{ij} \qquad \text{for all} \ 1 \leq i,j \leq N.
\end{gather}
These formulas imply that $U(R)$ is actually a~$Q$-graded Hopf superalgebra.

\begin{Remark}\label{rem:useful-lemma}
For the later use, we record a~few identities using the leg-numbering notation of Section~\ref{ssec:leg-numbering}. Let $\mu$ denote the multiplication morphism of $\CA$, which is either $\CT$ or
one of its subquotients. We then have
\[
 \mu \bigl(L^{\nu}_{(1),1} L^{\eta}_{(2),2}\bigr) = L^{\nu}_{1}L^{\eta}_{2} \in \CA \otimes \End(V_{q})^{\otimes 2},
\]
where $\nu,\eta \in \{\pm\}$. Since the braiding $\tau \colon A \otimes B \to B \otimes A$ itself is a~superalgebra morphism
for any superalgebras $A$ and~$B$, we also obtain
\begin{gather}
 \tau_{(12)} \bigl(L^{\nu}_{(1),i}L^{\eta}_{(2),j}\bigr) =
 \tau_{(12)} \bigl(L^{\nu}_{(1),i}\bigr) \tau_{(12)} \bigl(L^{\eta}_{(2),j}\bigr) =
 L^{\nu}_{(2),i}L^{\eta}_{(1),j},\nonumber\\
 \tau_{12} \bigl(L^{\nu}_{(i),1}L^{\eta}_{(j),2}\bigr) =
 \tau_{12} \bigl(L^{\nu}_{(i),1}\bigr) \tau_{12} \bigl(L^{\eta}_{(j),2}\bigr) =
 L^{\nu}_{(i),2}L^{\eta}_{(j),1},\label{eq:matrix notation leg swap}
\end{gather}
where $i$ and $j$ are not necessarily distinct. Furthermore, since $L^{\pm} \in \CA \otimes \End(V_{q})$ have
$\BZ_{2}$-degree~$\bbar{0}$, we obtain the following commutativity relation between matrices acting on disjoint
tensor arguments
\begin{gather}\label{eq:matrix notation mult order swap}
 L^{\nu}_{(i),a}L^{\eta}_{(j),b} = L^{\eta}_{(j),b}L^{\nu}_{(i),a} \qquad \text{for} \ a \neq b \ \text{and}\ i \neq j.
\end{gather}
\end{Remark}

\begin{Remark}\label{rem:U(R) involution}
Recall the notion of the \emph{supertranspose}. For any superalgebra $\CA$ and matrix
$X = \sum_{i,j=1}^{N} x_{ij} \otimes E_{ij} \in \CA \otimes \End(V_{q})$, its supertranspose is defined as
\begin{gather}\label{eq:supertranspose}
 X^{\st} \coloneqq \sum_{1\leq i,j\leq N} (-1)^{\overline{j}(\overline{i}+\overline{j})} x_{ij} \otimes E_{ji}.
\end{gather}
Suppose \smash{$X = \sum_{i,j=1}^{N} x_{ij} \otimes E_{ij}$} and \smash{$Y = \sum_{i,j=1}^{N} y_{ij} \otimes E_{ij}$} are elements in
$\CA \otimes \End(V_{q})$ that are homogeneous with respect to the total $\BZ_{2}$-grading, and whose entries satisfy
the following supercommutativity condition:
\begin{gather}\label{eq:matrix coeff supercommute}
 [x_{ik},y_{kj}] = 0 \qquad \text{for all} \ 1 \leq i,j,k \leq N,
\end{gather}
where $[-,-]$ denotes the supercommutator~\eqref{eq:super commutator}. Under this condition, the following identity holds:
\begin{gather}\label{eq:supertranspose product}
 (XY)^{\st} = (-1)^{|X||Y|} Y^{\st} X^{\st}.
\end{gather}
Following the notation of Section~\ref{ssec:leg-numbering}, for any element $X \in \CA \otimes \End(V_{q})^{\otimes s}$
and ${1 \leq i \leq s}$, we let~$X^{\st_{i}}$ denote the \emph{partial supertranspose} applied to the $i$-th factor of
$\End(V_{q})$. We note that these partial supertranspose operations acting on different factors mutually commute
\big($\smash{\bigl(X^{\st_{i}}\bigr)^{\st_{j}} }\allowbreak= \smash{\bigl(X^{\st_{j}}\bigr)^{\st_{i}}}$ for $i \neq j$\big) and satisfy (cf.~\eqref{eq:tensor braid action})
\begin{gather}\label{eq:partial supertranspose braiding}
 \tau^{\sigma}\bigl(X^{\st_{i}}\bigr) = \bigl(\tau^{\sigma}(X)\bigr)^{\st_{\sigma(i)}} \qquad \text{for any} \ \sigma \in S_{s}.
\end{gather}
The \emph{full supertranspose}, obtained by applying the partial supertranspose to all $s$ matrix factors, is denoted by
$X^{\st} \coloneqq X^{\st_{1} \cdots \st_{s}}$. By~\eqref{eq:partial supertranspose braiding} and the mutual commutativity
of partial supertransposes, we have
\begin{gather}\label{eq:full supertranspose braiding}
 \tau^{\sigma}\bigl(X^{\st}\bigr) = \bigl(\tau^{\sigma}(X)\bigr)^{\st}.
\end{gather}
Moreover, for homogeneous elements $X, Y \in \CA \otimes \End(V_{q})^{\otimes s}$ with respect to the total $\BZ_{2}$-grading,
whose entries satisfy the supercommutativity condition~\eqref{eq:matrix coeff supercommute}, the full supertranspose obeys
the exact same identity~\eqref{eq:supertranspose product}.
Furthermore, for any elements $X, Y \in \CA \otimes \End(V_{q})$ (not necessarily satisfying~\eqref{eq:matrix coeff supercommute}),
we have in $\CA \otimes \End(V_{q})^{\otimes s}$
\begin{gather}\label{eq:supertranspose different leg}
 (X_{i}Y_{j})^{\st} = X_{i}^{\st}Y_{j}^{\st} \qquad \forall 1 \leq i \neq j \leq s.
\end{gather}

Now consider the defining relations~\eqref{eq:gl-rtt-rel-1} and~\eqref{eq:gl-rtt-rel-2}
$
 R_{12}L^{\nu}_{1}L^{\eta}_{2} = L^{\eta}_{2}L^{\nu}_{1}R_{12}$,
where $(\nu,\eta) \in \{ (+,+), (-,-), (+,-) \}$. Since $R_{12}$ has coefficients in the base field, its entries supercommute
with those of $L^{\nu}_{1}L^{\eta}_{2}, L^{\eta}_{2}L^{\nu}_{1}$. Thus, applying the identity~\eqref{eq:supertranspose product}
to the full supertranspose of this relation yields
\[
 \bigl(L^{\nu}_{1}L^{\eta}_{2}\bigr)^{\st} (R_{12})^{\st} = (R_{12})^{\st} \bigl(L^{\eta}_{2}L^{\nu}_{1}\bigr)^{\st}.
\]
Furthermore, direct computation using the explicit formula~\eqref{eq:gl-evaluated-R} yields\footnote{While one simply has $(R_{12})^{\st} = R_{21}$ for the general linear type, we introduce the conjugation by $\Theta$ to provide a~uniform treatment that will also apply to orthosymplectic types below, see Remark~\ref{rem:osp-U(R) involution}, where it is absolutely necessary.}
\begin{align}
 (R_{12})^{\st}
 &= \sum_{1\leq i,j \leq N} q^{-(\varepsilon_{i},\varepsilon_{j})} E_{ii} \otimes E_{jj}
 - \bigl(q-q^{-1}\bigr) \sum_{i<j} (-1)^{\overline{i}} E_{ji} \otimes E_{ij}\nonumber\\
 &
 = \Theta_{1}\Theta_{2} \cdot R_{21} \cdot (\Theta_{1}\Theta_{2})^{-1},\label{eq:gl-st-vs-op}
\end{align}
where \smash{$\Theta = \sum_{i=1}^{N} q^{(\varepsilon_{i},\varepsilon_{i})/2} E_{ii}$}, and by~\eqref{eq:supertranspose different leg} we obtain
\begin{gather*}
 \bigl(L^{\nu}_{1}\bigr)^{\st}\bigl(L^{\eta}_{2}\bigr)^{\st} \cdot \Theta_{1}\Theta_{2} \cdot R_{21} \cdot (\Theta_{1}\Theta_{2})^{-1} = \Theta_{1}\Theta_{2} \cdot R_{21} \cdot (\Theta_{1}\Theta_{2})^{-1} \cdot \bigl(L^{\eta}_{2}\bigr)^{\st}\bigl(L^{\nu}_{1}\bigr)^{\st}.
\end{gather*}
Applying the permutation operator $\tau_{12}$ to both sides gives
(cf.~\eqref{eq:matrix notation leg swap}, \eqref{eq:full supertranspose braiding})
\begin{gather*}
 \bigl(L^{\nu}_{2}\bigr)^{\st}\bigl(L^{\eta}_{1}\bigr)^{\st} \cdot \Theta_{1}\Theta_{2} \cdot R_{12} \cdot (\Theta_{1}\Theta_{2})^{-1} = \Theta_{1}\Theta_{2} \cdot R_{12} \cdot (\Theta_{1}\Theta_{2})^{-1} \cdot \bigl(L^{\eta}_{1}\bigr)^{\st}\bigl(L^{\nu}_{2}\bigr)^{\st},
\end{gather*}
and conjugating further by $(\Theta_{1}\Theta_{2})^{-1}$ gives
\begin{gather*}
 (\Theta_{1}\Theta_{2})^{-1} \cdot \bigl(L^{\nu}_{2}\bigr)^{\st}\bigl(L^{\eta}_{1}\bigr)^{\st} \cdot \Theta_{1}\Theta_{2} \cdot R_{12} = R_{12} \cdot (\Theta_{1}\Theta_{2})^{-1} \cdot \bigl(L^{\eta}_{1}\bigr)^{\st}\bigl(L^{\nu}_{2}\bigr)^{\st} \cdot \Theta_{1}\Theta_{2}.
\end{gather*}
As $L^{\pm}_{i}$ and $\Theta_{j}$ commute for $i \neq j$, and $\bigl(L^{+}\bigr)^{\st}$ and $\bigl(L^{-}\bigr)^{\st}$ are lower and upper triangular, respectively, we thus obtain a~superalgebra isomorphism
\begin{gather}\label{eq:transpose-iso}
 \omega_{R}\colon\ U(R)\iso U(R) \qquad \text{given by} \ L^{\pm} \mapsto \Theta^{-1}\bigl(L^{\mp}\bigr)^{\st}\Theta.
\end{gather}
Similarly to Remark~\ref{rem:DrJ involution}, it gives rise to the Hopf superalgebra isomorphism
$\omega_{R}\colon U(R)\iso\allowbreak U(R)^{\copp}$ since for $\Theta \in \End(V_q)$ and $L^{\pm} \in U(R) \otimes \End(V_q)$
\begin{gather*}
 \Delta\bigl(\bigl(L^{\pm}\bigr)^{\st}\bigr) =
 \bigl(\Delta\bigl(L^{\pm}\bigr)\bigr)^{\st} \overset{\eqref{eq:supertranspose product}}{=} \bigl(L_{(2)}^{\pm}\bigr)^{\st}\bigl(L_{(1)}^{\pm}\bigr)^{\st},
\end{gather*}
and therefore we obtain
\begin{align*}
 \Delta\bigl(\Theta^{-1}\bigl(L^{\pm}\bigr)^{\st}\Theta\bigr) &= \Theta^{-1}\Delta\bigl(\bigl(L^{\pm}\bigr)^{\st}\bigr)\Theta
 = \Theta^{-1}\bigl(L_{(2)}^{\pm}\bigr)^{\st}\bigl(L_{(1)}^{\pm}\bigr)^{\st}\Theta\\
 &
 = \Theta^{-1}\bigl(L_{(2)}^{\pm}\bigr)^{\st}\Theta \cdot \Theta^{-1}\bigl(L_{(1)}^{\pm}\bigr)^{\st}\Theta.
\end{align*}
It is immediate that $\omega_{R}$ restricts to superalgebra isomorphisms between the subbialgebras
$U^{+}(R)$ and $U^{-}(R)$ introduced in the next subsection, cf.\ Remark~\ref{rem:DrJ involution subalgebras}.
\end{Remark}

\subsection[Generalized double construction of U(R)]{Generalized double construction of $\boldsymbol{U(R)}$}

Let $U^{+}(R)$ and $U^{-}(R)$ be the superalgebras generated by
$\bigl\{l^{+}_{ij}\bigr\}_{1 \leq i \leq j \leq N}$ and $\bigl\{l^{-}_{ji}\bigr\}_{1 \leq i \leq j \leq N}$, respectively,
 equipped with the $Q$-grading~\eqref{eq:uR-Q-grading} and the compatible
$\BZ_{2}$-grading, subject to the respective relations~\eqref{eq:gl-rtt-rel-1}. The same formulas~\eqref{eq:gl-rtt-coalg}
endow both $U^{\pm}(R)$ with superbialgebra structures. Similarly to the case of $\uqgl$, the superalgebra $U(R)$ can be
realized as a~quotient of the double of the superbialgebra pair $\bigl(U^{+}(R), U^{-}(R)\bigr)$.

\begin{Proposition}\label{prop:RTT-pairing_finite}\quad
\begin{itemize}\itemsep=0pt
\item[$(a)$] There exists a~unique skew-pairing $($cf.~\eqref{eq:skew pairing structural property}$)$
\begin{gather}\label{eq:gl-RTT-skew-pairing}
 (\cdot,\cdot)_{R} \colon\ U^{+}(R) \times U^{-}(R) \to \BC(q)
\end{gather}
satisfying the following property on the generators:
\begin{gather}\label{eq:gl-RTT-skew-pairing-gen}
 \sum_{1\leq i,j,k,l\leq N} (-1)^{(\overline{i}+\overline{j})(\overline{k}+\overline{l})} \bigl(l^{+}_{ij}, l^{-}_{kl}\bigr)_{R} E_{ij} \otimes E_{kl}
 = R.
\end{gather}

\item[$(b)$] The superalgebra $U(R)$ is isomorphic to the quotient of the generalized double
$\CD_{R} \coloneqq \CD\bigl(U^{+}(R),U^{-}(R)\bigr)$
modulo the relations \smash{$l^{\pm}_{ii}l^{\mp}_{ii} = 1$} for all $1 \leq i \leq N$.
\end{itemize}
\end{Proposition}

\begin{Remark}
Using the notation $\sigma \coloneqq (\cdot,\cdot)_{R}$, the formula~\eqref{eq:gl-RTT-skew-pairing-gen} can be written
concisely as a~matrix relation
\begin{gather}\label{eq:gl-RTT-skew-pairing-matrix}
 \sigma \bigl(L^{+}_{(1),1} L^{-}_{(2),2}\bigr) = R,
\end{gather}
where we employ the leg-numbering notation defined in Section~\ref{ssec:leg-numbering}.
We will frequently use this below.
\end{Remark}

\begin{Remark}\label{rem:gl-rtt-convention}
We also note that~\eqref{eq:gl-RTT-skew-pairing-gen} is compatible with the convention $l^{+}_{ij} = 0 = l^{-}_{ji}$
for~${i>j}$, which is evident from the explicit formula for $R$ presented in~\eqref{eq:gl-evaluated-R}.
\end{Remark}

\begin{proof}[Proof of Proposition~\ref{prop:RTT-pairing_finite}]
(a) We first note that there exists a~skew-pairing between the free $\BC(q)$-superalgebras generated by \smash{$\bigl\{l^{+}_{ij}\bigr\}_{1 \leq i \leq j \leq N}$}
and \smash{$\{l^{-}_{ji}\}_{1 \leq i \leq j \leq N}$} respectively, which is defined by
the exact same formula~\eqref{eq:gl-RTT-skew-pairing-gen} on the generators.
It is then enough to show that the relations~\eqref{eq:gl-rtt-rel-1} belong to
the kernel of this skew-pairing, which is a~consequence of the
Yang--Baxter equation~\eqref{eq:YBE} together with Remark~\ref{rem:useful-lemma}.
The proof is analogous to that for the non-super case, which can be easily deduced from~\cite[Section~10.2.5 and Example~16]{ks} (cf.~\cite[Section~3]{frt} for a~similar
pairing between the quantum enveloping algebra $U(R)$ and the quantized coordinate algebra~$A(R)$).

(b) The proof is completely parallel to that of Proposition~\ref{prop:DJ-pairing_finite}\,(b). Following the notation in Proposition~\ref{prop:gen double}, we note again that the double $\CD_{R}$ is the quotient of the free product superalgebra $U^{+}(R) * U^{-}(R)$ by the two-sided ideal $\CI$ generated by the cross-relations~\eqref{eq:gl-DJ-cross-rel}, namely all elements $u_{x,y}$ for $\BZ_{2}$-homogeneous $x \in U^{+}(R)$ and $y \in U^{-}(R)$.
Since $U^{+}(R)$ and~$U^{-}(R)$ are respectively graded by $Q^{+}$ and $-Q^{+}$, the explicit comultiplication
formula~\eqref{eq:gl-rtt-coalg-component}, together with Proposition~\ref{prop:double gen rel}
(see Remark~\ref{rem:reduced_generators}), implies that $\CI$ is generated by the elements $u_{x,y}$ as~$x$ and $y$ range over
the generators of $U^{+}(R)$ and $U^{-}(R)$, respectively.
Applying~\eqref{eq:gl-DJ-cross-rel} to pairs of generators $\bigl(l^{+}_{ij},l^{-}_{kl}\bigr)$ and organizing them into a~matrix yields
\begin{gather*}
\begin{split}
 &\sum_{i,j,k,l} u_{l^{+}_{ij},l^{-}_{kl}} \otimes E_{ij} \otimes E_{kl} \\
& \qquad= \sum_{i,j,k,l} \sum_{a,b} (-1)^{(\overline{b}+\overline{l})(\overline{i}+\overline{j})}
 \cdot (-1)^{(\overline{i}+\overline{a})(\overline{a}+\overline{j})}(-1)^{(\overline{k}+\overline{b})(\overline{b}+\overline{l})}\\
&\phantom{\qquad=}{} \times \left\{
 (-1)^{(\overline{i}+\overline{a})(\overline{k}+\overline{b})} \bigl(l^{+}_{ia},l^{-}_{kb}\bigr)_{R} \cdot l^{+}_{aj}l^{-}_{bl}
 - \bigl(l^{+}_{aj},l^{-}_{bl}\bigr)_{R} \cdot l^{-}_{kb}l^{+}_{ia} \right\} \otimes E_{ij} \otimes E_{kl}\\
&\qquad\! \overset{(\star)}{=} \sum_{i,j,k,l} \sum_{a,b} (-1)^{(\overline{a}+\overline{j})(\overline{k}+\overline{l})} \\
&\phantom{\qquad=}{} \times \left\{ (-1)^{(\overline{i}+\overline{a})(\overline{k}+\overline{b})} \bigl(l^{+}_{ia},l^{-}_{kb}\bigr)_{R} \cdot l^{+}_{aj}l^{-}_{bl}
 - \bigl(l^{+}_{aj},l^{-}_{bl}\bigr)_{R} \cdot l^{-}_{kb}l^{+}_{ia} \right\} \otimes E_{ij} \otimes E_{kl} \\
& \qquad= R_{12} L^{+}_{1}L^{-}_{2} - L^{-}_{2}L^{+}_{1} R_{12},
\end{split}
\end{gather*}
where the simplification of signs $(\star)$ is due to Remark~\ref{rem:pairing-degree-zero},
\begin{align*}
 & (-1)^{(\overline{b}+\overline{l})(\overline{i}+\overline{j})} (-1)^{(\overline{i}+\overline{a})(\overline{a}+\overline{j})} (-1)^{(\overline{k}+\overline{b})(\overline{b}+\overline{l})}
 (-1)^{(\overline{i}+\overline{a})(\overline{k}+\overline{b})} \cdot \bigl(l^{+}_{ia},l^{-}_{kb}\bigr)_{R}\\
 &\qquad= (-1)^{(\overline{b}+\overline{l})(\overline{i}+\overline{j})} (-1)^{(\overline{k}+\overline{b})(\overline{a}+\overline{j})} (-1)^{(\overline{i}+\overline{a})(\overline{b}+\overline{l})}
 (-1)^{(\overline{i}+\overline{a})(\overline{k}+\overline{b})} \cdot \bigl(l^{+}_{ia},l^{-}_{kb}\bigr)_{R}\\
 &\qquad= (-1)^{(\overline{k}+\overline{l})(\overline{a}+\overline{j})} (-1)^{(\overline{i}+\overline{a})(\overline{k}+\overline{b})} \cdot \bigl(l^{+}_{ia},l^{-}_{kb}\bigr)_{R}
\end{align*}
and
\begin{align*}
 & (-1)^{(\overline{b}+\overline{l})(\overline{i}+\overline{j})} (-1)^{(\overline{i}+\overline{a})(\overline{a}+\overline{j})} (-1)^{(\overline{k}+\overline{b})(\overline{b}+\overline{l})} \cdot
 \bigl(l^{+}_{aj},l^{-}_{bl}\bigr)_{R}\\
 &\qquad= (-1)^{(\overline{b}+\overline{l})(\overline{i}+\overline{j})} (-1)^{(\overline{i}+\overline{a})(\overline{b}+\overline{l})} (-1)^{(\overline{k}+\overline{b})(\overline{a}+\overline{j})} \cdot
 \bigl(l^{+}_{aj},l^{-}_{bl}\bigr)_{R}
 = (-1)^{(\overline{k}+\overline{l})(\overline{a}+\overline{j})} \cdot \bigl(l^{+}_{aj},l^{-}_{bl}\bigr)_{R}.
\end{align*}
Thus, imposing $u_{l^{+}_{ij},l^{-}_{kl}} = 0$ is precisely equivalent to imposing relation~\eqref{eq:gl-rtt-rel-2}.

Finally, by taking the quotient of the double $\CD_{R}$ by the two-sided ideal (cf.~\eqref{eq:gl-rtt-diag-rel})
\begin{gather*}
 \CJ \coloneqq \bigl( l^{\pm}_{ii}l^{\mp}_{ii} - 1 \mid 1 \leq i \leq N \bigr),
\end{gather*}
we obtain the isomorphism $\CD_{R}/\CJ \simeq U(R)$.
\end{proof}

\begin{Remark}\label{rem:RTT double full Cartan}
Let $U^{\geq}(R)$ and $U^{\leq}(R)$ be the \emph{Borel subalgebras} of $U(R)$ generated by
\[
\bigl\{l^{+}_{ij}\bigr\}_{1 \leq i \leq j \leq N} \cup \{l^{-}_{kk}\}_{1 \leq k \leq N} \qquad \text{and}\qquad \{l^{-}_{ji}\}_{1 \leq i \leq j \leq N} \cup \{l^{+}_{kk}\}_{1 \leq k \leq N},
\]
respectively. From the double construction
above, one can check that the defining relations of~$U^{\geq}(R)$ consist of the internal relations of~$U^{+}(R)$
(the ``$+$''-sign of~\eqref{eq:gl-rtt-rel-1}) and the cross-relations between \smash{$\bigl\{l^{+}_{ij}\bigr\}_{1 \leq i \leq j \leq N}$} and
\smash{$\{l^{-}_{kk}\}_{1 \leq k \leq N}$} (encompassing~\eqref{eq:gl-rtt-diag-rel}). An analogous statement
holds for $U^{\leq}(R)$. Consequently, it can be shown that the skew-pairing~\eqref{eq:gl-RTT-skew-pairing} extends uniquely
to a~skew-pairing $U^{\geq}(R) \times U^{\leq}(R) \to \BC(q)$, denoted by the same symbol. Furthermore, $U^{\geq}(R)$ and
$U^{\leq}(R)$ are Hopf subalgebras of $U(R)$, as they are closed under the antipode~\eqref{eq:gl-rtt-antipode}. As in
Remark~\ref{rem:DrJ double half Cartan}, the generalized double construction for~$U(R)$ can alternatively be carried out
using these Hopf subalgebras instead of the subbialgebras~$U^{+}(R)$ and~$U^{-}(R)$. Also the isomorphism~$\omega_{R}$ of
\eqref{eq:transpose-iso} restricts to a~superalgebra isomorphism between~$U^{\geq}(R)$ and~$U^{\leq}(R)$.
\end{Remark}

\begin{Remark}\label{rem:RTT pairing order}
As in Remark~\ref{rem:DrJ pairing order}, one may alternatively construct the generalized double of~$U(R)$ by taking
the subbialgebras in the opposite order and utilizing the transposed inverse of the pairing $\sigma = (\cdot,\cdot)_{R}$:
\begin{gather*}
 \widetilde{\sigma}=(-,-)_{\widetilde{R}} \colon
\ U^{-}(R) \times U^{+}(R) \to \BC(q).
\end{gather*}
By direct computation using the convolution product, one can verify that its evaluation on the generators is explicitly given
by the matrix formula (cf.~\eqref{eq:gl-RTT-skew-pairing-matrix})
\begin{gather}\label{eq:gl-RTT-skew-pairing-opposite}
 \widetilde{\sigma}\bigl(L^{-}_{(1),1} L^{+}_{(2),2}\bigr) = R^{-1}_{21},
\end{gather}
using the notation of Section~\ref{ssec:leg-numbering}. Furthermore, as in Remark~\ref{rem:RTT double full Cartan}, this pairing
uniquely extends to a~skew-pairing $(-,-)_{\widetilde{R}} \colon U^{\leq}(R) \times U^{\geq}(R) \to \BC(q)$ between the Borel subalgebras.
\end{Remark}

\subsection[The twisted superalgebra U(R)]{The twisted superalgebra $\boldsymbol{\URtw}$}\label{ssec:sign-twists}

Although the defining relations~\eqref{eq:gl-rtt-rel-1} and~\eqref{eq:gl-rtt-rel-2} are concise in matrix form, extracting
their component-wise formulas can be computationally complicated. This difficulty arises from the Koszul signs generated
when swapping elements of the coefficient superalgebra $\CA$ with the auxiliary matrix algebra $\End(V_{q})$ during the
multiplication of formal matrices like $L^{\pm}_{(a),b}$. To circumvent this issue, we introduce a~\emph{twisted superalgebra}.

Let $G$ be an abelian group, written additively, and let $\CC \coloneqq \Vect_{\BK}^{G}$ be the category of $G$-graded
$\BK$-vector spaces with degree-preserving linear maps, endowed with the standard tensor product of underlying $\BK$-vector
spaces. Consider a~$2$-cochain (that is, an arbitrary function) $\zeta \colon G \times G \to \BK^{\times}$ on $G$. The identity
functor $F$ on $\CC$, together with the identity morphism~\smash{$\BK \overset{\Id}{\to} F(\BK)$} and the natural isomorphism
\begin{gather*}
 J^{\zeta}_{X,Y} \colon\ F(X) \otimes F(Y) \iso F(X \otimes Y),\qquad
 x \otimes y \mapsto \zeta(|x|_{G},|y|_{G}) x \otimes y,
\end{gather*}
defined for any $G$-homogeneous elements $x \in X, y \in Y$ of degrees $|x|_{G},|y|_{G} \in G$, forms a~monoidal endofunctor
on $\CC$ if and only if $\zeta$ is a~\emph{normalized $2$-cocycle}
\begin{gather}
 \zeta(g',g'') \zeta(g,g'+g'') \zeta(g,g')^{-1} \zeta(g+g',g'')^{-1} = 1,\qquad
 \zeta(g,0) = 1 = \zeta(0,g) ,\nonumber\\
 \forall g,g',g'' \in G.\label{eq:2-cocycle}
\end{gather}
Let $A$ be an algebra object in $\CC$ (i.e., a~$G$-graded algebra), and let $\mu_{A}$ denote its multiplication morphism.
Using the above isomorphism $J^{\zeta}$, one can define a~\emph{twisted algebra} $A^{\zeta}$ whose multiplication
is given by
\begin{gather}\label{eq:twisted-mult}
 \mu_{A}^{\zeta} \coloneqq \mu_{A} \circ J^{\zeta}_{A,A} \colon\ a \otimes a' \mapsto \zeta(|a|_{G},|a'|_{G}) aa'.
\end{gather}
Furthermore, if $\phi \colon A \to B$ is an algebra morphism, functoriality ensures that the map $\phi = F(\phi)$ itself
provides an algebra morphism $\phi \colon A^{\zeta} \to B^{\zeta}$ between the twisted algebras.

Now let us consider braided structures on $\CC$. For any $2$-cochain $\beta \colon G \times G \to \BK^{\times}$,
one can define a~natural isomorphism
\[
 c^{\beta}_{X,Y} \colon \ X \otimes Y \iso Y \otimes X, \qquad
 x \otimes y \mapsto \beta(|x|_{G},|y|_{G}) y \otimes x,
\]
for any $G$-homogeneous elements $x \in X$, $y \in Y$ of degrees $|x|_{G},|y|_{G} \in G$. This natural isomorphism equips~$\CC$
with a~braided monoidal structure if and only if $\beta$ is a~\emph{bicharacter}, meaning it is additive in both arguments
\begin{gather*}
 \beta(g+g',h) = \beta(g,h)\beta(g',h), \qquad \beta(g,h+h') = \beta(g,h)\beta(g,h') ,\qquad \forall g,g',h,h' \in G.
\end{gather*}
Equivalently, viewing $G$ as a~$\BZ$-module, a~bicharacter is simply a~$\BZ$-module homomorphism
$G \otimes_{\BZ} G \to \BK^{\times}$.

For any algebra objects $A$ and $B$ in $\CC$, the braiding $c^{\beta}$ endows the tensor product $A \otimes B$ with
an algebra structure, denoted $A \otimes^{\beta} B$, with multiplication given on $G$-homogeneous elements by
(cf.~\eqref{eq:superalgebras tensoring})
\[
 (a \otimes b) * (a' \otimes b') \coloneqq \beta(|b|_{G},|a'|_{G}) aa' \otimes bb'.
\]

\begin{Remark}
The category $\sVect$ is recovered in this framework by setting $G = \BZ_{2}$ and $\beta(g,h) = (-1)^{gh}$.
\end{Remark}

We shall now apply these constructions for $G = \BZ_{2} \times \BZ_{2}$ and $\BK = \BC(q)$. When working
with tensor products of the form $\CA^{\otimes r} \otimes \End(V_{q})^{\otimes s}$, we view the coefficient superalgebra $\CA$
and the matrix algebra $\End(V_{q})$ as objects in $\CC$ by assigning their standard $\BZ_{2}$-degrees to the first and second
components of $G$, respectively. Specifically, homogeneous elements in $\CA$ have degrees in $\BZ_{2} \times \{0\}$, while
those in $\End(V_{q})$ have degrees in $\{0\} \times \BZ_{2}$.
We endow $\CC$ with two different braidings given by bicharacters.

\textbf{$\bullet$ The standard braiding:}
Consider the bicharacter
\begin{gather*}
 \beta(g,h) \coloneqq (-1)^{(g_1 + g_2)(h_1 + h_2)} \in \{\pm 1\},
\end{gather*}
defined for any elements $g = (g_1,g_2)$ and $h = (h_1,h_2)$ in $G = \BZ_{2} \times \BZ_{2}$. For any objects $X$, $Y$ in~$\CC$,
the standard braiding $c \coloneqq c^{\beta}$ evaluates on $G$-homogeneous elements $x \in X$ and $y \in Y$ as
\begin{gather*}
 c_{X,Y}(x \otimes y) \coloneqq \beta(|x|_{G},|y|_{G}) y \otimes x = (-1)^{(|x|_{1}+|x|_{2})(|y|_{1}+|y|_{2})} y \otimes x,
\end{gather*}
where their $G$-degrees are denoted by $|x|_{G} = (|x|_{1},|x|_{2})$ and $|y|_{G} = (|y|_{1},|y|_{2})$, respectively.
Under this braiding, multiplying two formal matrices yields
\begin{gather*}
 L^{\nu}_{1}L^{\eta}_{2} =
 \sum_{1\leq i,j,k,l \leq N} (-1)^{(\overline{i}+\overline{j})(\overline{k}+\overline{l})} l^{\nu}_{ij}l^{\eta}_{kl} \otimes E_{ij} \otimes E_{kl}
 \in \CA \otimes^{\beta} \End(V_{q})^{\otimes 2}.
\end{gather*}
Here, the Koszul sign appears because the matrix unit $E_{ij}$ ``passes through'' the generator~$l^{\eta}_{kl}$. We note that this
braiding is compatible with the standard symmetric monoidal structure of~$\sVect$: the group homomorphism
\[
 G \to \BZ_{2} , \qquad (a, b) \mapsto a+b
\]
induces a~braided monoidal functor from $(\CC,c)$ to $\sVect$. Therefore, computing the matrix relations~\eqref{eq:gl-rtt-rel-1}
and~\eqref{eq:gl-rtt-rel-2} in $(\CC,c)$ yields the exact same component-wise formulas as computing them in $\sVect$.

\textbf{$\bullet$ The twisted braiding:}
Consider the alternative bicharacter
\[
 \beta'(g,h) \coloneqq (-1)^{g_{1}h_{1} + g_{2}h_{2}},
\]
defined for any $g=(g_1,g_2), h=(h_1,h_2) \in G$. For any objects $X$, $Y$ in $\CC$, this yields the twisted braiding
$c' \coloneqq c^{\beta'}$ given by
\begin{gather*}
 c'_{X,Y}(x \otimes y) \coloneqq \beta'(|x|_{G},|y|_{G}) y \otimes x = (-1)^{|x|_{1}|y|_{1} + |x|_{2}|y|_{2}} y \otimes x.
\end{gather*}
Under this twisted braiding, no signs appear when swapping an object concentrated in degree~${\BZ_{2} \times \{0\}}$ with
an object concentrated in degree $\{0\} \times \BZ_{2}$. Consequently, formal matrix multiplication simplifies as follows:
\begin{gather*}
 L^{\nu}_{1}L^{\eta}_{2} = \sum_{1\leq i,j,k,l \leq N} l^{\nu}_{ij}l^{\eta}_{kl} \otimes E_{ij} \otimes E_{kl}
 \in \CA \otimes^{\beta'} \End(V_{q})^{\otimes 2}.
\end{gather*}

Now consider the superalgebra $\URtw$, which has the exact same generators and diagonal relations~\eqref{eq:gl-rtt-diag-rel}
as $U(R)$, and whose remaining defining relations take the identical matrix form as~\eqref{eq:gl-rtt-rel-1}
and~\eqref{eq:gl-rtt-rel-2} but are evaluated in the superalgebra $\CT \otimes^{\beta'} \End(V_{q})^{\otimes 2}$ instead of
${\CT \otimes^{\beta} \End(V_{q})^{\otimes 2}}$.

\begin{Proposition}\label{prop:gl-twisted-rtt}
Let $\zeta \colon \BZ_{2} \times \BZ_{2} \to \BC(q)^{\times}$ be the 2-cochain on $\BZ_{2}$ given by
$\zeta(\bbar{i},\bbar{j}) \coloneqq (-1)^{\bbar{i} \bbar{j}}$. Then, the assignment $l^{\pm}_{ij} \mapsto l^{\pm}_{ij}$
for all $i$, $j$ gives rise to a~superalgebra isomorphism $\URtw \iso\allowbreak U(R)^{\zeta}$.
\end{Proposition}

\begin{proof}
We first note that it is straightforward to verify that $\zeta$ is a~normalized $2$-cocycle, i.e., satisfies~\eqref{eq:2-cocycle}. To avoid confusion, let $*$ and $*'$ denote the multiplications in
$\CT^{\otimes r} \otimes^{\beta} \End(V_{q})^{\otimes 2}$ and $\CT^{\otimes r} \otimes^{\beta'} \End(V_{q})^{\otimes 2}$
for any $r \geq 1$, respectively. We also note that since $R_{12} \in \End(V_{q})^{\otimes 2}$ has $G$-degree $0$,
multiplication by $R_{12}$ does not depend on the choice of $*$ or $*'$.

Consider the RLL defining relations~\eqref{eq:gl-rtt-rel-1}, \eqref{eq:gl-rtt-rel-2} of $U(R)$ written as
$R_{12}\bigl(L^{\nu}_{1} * L^{\eta}_{2}\bigr) = \bigl(L^{\eta}_{2} * L^{\nu}_{1}\bigr)R_{12}$ with $(\nu,\eta) \in \{ (+,+), (-,-), (+,-) \}$.
Using the multiplication morphism $\mu_{\CT}$ of $\CT$, this can be written as follows:
\begin{gather*}
 R_{12} \mu_{\CT}\bigl(L^{\nu}_{(1),1} * L^{\eta}_{(2),2}\bigr) = \mu_{\CT}\bigl(L^{\eta}_{(2),2} * L^{\nu}_{(1),1}\bigr)R_{12}.
\end{gather*}
Since $\mu^{\zeta}_{\CT} = \mu_{\CT} \circ J^{\zeta}_{\CT,\CT}$, this is equivalent to
\begin{gather}\label{eq:twisted rel 1}
 R_{12} \mu^{\zeta}_{\CT} \bigl(\bigl(J^{\zeta}_{\CT,\CT}\bigr)^{-1}\bigl(L^{\nu}_{(1),1} * L^{\eta}_{(2),2}\bigr)\bigr)
 = \mu^{\zeta}_{\CT} \bigl(\bigl(J^{\zeta}_{\CT,\CT}\bigr)^{-1}\bigl(L^{\eta}_{(2),2} * L^{\nu}_{(1),1}\bigr)\bigr)R_{12}.
\end{gather}
Moreover, direct computation yields (cf.~\eqref{eq:matrix notation mult order swap})
\begin{align*}
 \bigl(J^{\zeta}_{\CT,\CT}\bigr)^{-1}\bigl(L^{\nu}_{(1),1} * L^{\eta}_{(2),2}\bigr)
 ={}& \sum_{i,j,k,l=1}^N (-1)^{(\overline{i}+\overline{j})(\overline{k}+\overline{l})}
 \bigl(J^{\zeta}_{\CT,\CT}\bigr)^{-1}\bigl(l^{\nu}_{ij} \otimes l^{\eta}_{kl}\bigr) \otimes E_{ij} \otimes E_{kl}\\
 ={}& L^{\nu}_{(1),1} *' L^{\eta}_{(2),2}, \\
 \bigl(J^{\zeta}_{\CT,\CT}\bigr)^{-1}\bigl(L^{\eta}_{(2),2} * L^{\nu}_{(1),1}\bigr)
 ={}& \sum_{i,j,k,l=1}^N (-1)^{(\overline{i}+\overline{j})(\overline{k}+\overline{l})}
 \bigl(J^{\zeta}_{\CT,\CT}\bigr)^{-1}\bigl(l^{\nu}_{ij} \otimes l^{\eta}_{kl}\bigr) \otimes E_{ij} \otimes E_{kl}\\
 ={}&L^{\eta}_{(2),2} *' L^{\nu}_{(1),1}.
\end{align*}
Substituting these back into~\eqref{eq:twisted rel 1}, we obtain
\begin{gather*}
 R_{12} \mu^{\zeta}_{\CT} \bigl(L^{\nu}_{(1),1} *' L^{\eta}_{(2),2}\bigr)
 = \mu^{\zeta}_{\CT} \bigl(L^{\eta}_{(2),2} *' L^{\nu}_{(1),1}\bigr)R_{12}.
\end{gather*}
This implies that the defining relations of $U(R)^{\zeta}$ coincide exactly with those of $\URtw$
(where the compatibility of the diagonal relations~\eqref{eq:gl-rtt-diag-rel} is trivial due to
$\zeta\bigl(|l^{\pm}_{ii}|,|l^{\mp}_{ii}|\bigr) = 1$), which completes the proof.
\end{proof}

\begin{Remark}
The above isomorphism $\URtw \iso U(R)^{\zeta}$ implies that the original superalgebra~$U(R)$ can be recovered from
$\URtw$ via an \emph{untwisting} procedure (i.e., twisting by $\zeta^{-1}$), that is, $U(R)\iso \URtw^{\zeta^{-1}}$.
\end{Remark}

\subsection{Gauss decomposition and twisted identities: general linear case}\label{ssec:gl-Gauss}

In this subsection, we shall assume that all algebraic manipulations occur within the twisted superalgebra
$\URtw$ or the tensor product $\URtw^{\otimes r} \otimes^{\beta'} \End(V_{q})^{\otimes s}$ (for $r,s \geq 1$),
rather than in their untwisted counterparts.

Any element in $\URtw \otimes^{\beta'} \End(V_{q})$ can be expressed in matrix form as follows:
\begin{gather}\label{eq:gl-formal-matrix}
 \sum_{1 \leq i,j \leq N} a_{ij} \otimes E_{ij} =
 \begin{bsmallmatrix}
 a_{11} & a_{12} & \cdots & a_{1N} \\
 a_{21} & a_{22} & \cdots & a_{2N} \\
 \vdots & \ddots & \ddots & \vdots \\
 a_{N1} & a_{N2} & \cdots & a_{NN}
 \end{bsmallmatrix}
 \in \URtw \otimes^{\beta'} \End(V_{q}).
\end{gather}
Since the Koszul signs no longer arise when swapping elements of $\URtw$ and $\End(V_{q})$, the multiplication in
\smash{$\URtw \otimes^{\beta'} \End(V_{q})$} coincides exactly with the standard matrix multiplication:
\begin{gather}\label{eq:gl-matrix-mult}
 \bigg(\sum_{1\leq i,j \leq N} a_{ij} \otimes E_{ij}\bigg)
 \bigg(\sum_{1\leq i,j \leq N} b_{ij} \otimes E_{ij}\bigg)
 = \sum_{1\leq i,j \leq N} \bigg(\sum_{1\leq k \leq N} a_{ik}b_{kj}\bigg) \otimes E_{ij}.
\end{gather}
We will frequently utilize this matrix form. In particular, $L^{+}$ and $L^{-}$ are genuinely upper and lower
triangular matrices, respectively. As $l^\pm_{ii}$ are invertible by~\eqref{eq:gl-rtt-diag-rel}, we obtain the following
\emph{Gauss decompositions}:
\begin{gather}
 L^{+} =
 \begin{bsmallmatrix}
 l^{+}_{11} & 0 & \cdots & 0 \\
 0 & l^{+}_{22} & \cdots & 0 \\
 \vdots & \ddots & \ddots & \vdots \\
 0 & \cdots & 0 & l^{+}_{NN}
 \end{bsmallmatrix}
 \begin{bsmallmatrix}
 1 & (l^{+}_{11})^{-1}l^{+}_{12} & \cdots & (l^{+}_{11})^{-1}l^{+}_{1N}\\
 0 & 1 & \ddots & \vdots\\
 \vdots & \ddots & \ddots & (l^{+}_{N-1,N-1})^{-1}l^{+}_{N-1,N}\\
 0 & \cdots & 0 & 1
 \end{bsmallmatrix},\nonumber\\
 L^{-} =
 \begin{bsmallmatrix}
 1 & 0 & \cdots & 0\\
 l^{-}_{21}(l^{-}_{11})^{-1} & 1 & \ddots & \vdots\\
 \vdots & \ddots & \ddots & 0\\
 l^{-}_{N1}(l^{-}_{11})^{-1} & \cdots & l^{-}_{N,N-1}(l^{-}_{N-1,N-1})^{-1} & 1
 \end{bsmallmatrix}
 \begin{bsmallmatrix}
 l^{-}_{11} & 0 & \cdots & 0 \\
 0 & l^{-}_{22} & \cdots & 0 \\
 \vdots & \ddots & \ddots & \vdots \\
 0 & \cdots & 0 & l^{-}_{NN}
 \end{bsmallmatrix}.\label{eq:gl-Gauss-L}
\end{gather}
For convenience, we define the elements $\{\bfe_{ij}, \bff_{ji}\}_{1 \leq i < j \leq N}$ in $\URtw$ via the following formulas (cf.~\eqref{eq:weyl-vector}):
\begin{gather}
 \bigl(l_{ii}^{+}\bigr)^{-1} l_{ij}^{+} =
 (-1)^{\overline{i}\, \overline{j}} q^{-(\varepsilon_{i},\varepsilon_{i})} q^{-(\rho,\varepsilon_{i}-\varepsilon_{j})} \bigl(q - q^{-1}\bigr) \cdot \bfe_{ij},\nonumber\\
 l_{ji}^{-} \bigl(l_{ii}^{-}\bigr)^{-1} =
 -(-1)^{\overline{i} + \cdots + \overline{j-1}} (-1)^{\overline{i} \, \overline{j}} q^{(\varepsilon_{i},\varepsilon_{i})} q^{(\rho,\varepsilon_{i}-\varepsilon_{j})} \bigl(q - q^{-1}\bigr) \cdot \bff_{ji}.\label{eq:RTT ef gen}
\end{gather}
We also introduce \smash{$\bigl\{q^{\pm \bfH_{k}}\bigr\}_{k=1}^{N}$} in $\URtw$ via \smash{$q^{\pm \bfH_k} = l_{kk}^{\pm}$} and
define \smash{$q^{\pm \bfH_{ij}} \coloneqq q^{\pm \bfH_i} \bigl(q^{\pm \bfH_j}\bigr)^{-1}$} for $1 \leq i < j \leq N$.

We now exhibit a~list of identities satisfied by the above elements of \smash{$\URtw$}. These identities are extracted by comparing
the appropriate matrix components of the defining relations~\eqref{eq:gl-rtt-rel-1} and~\eqref{eq:gl-rtt-rel-2} of \smash{$\URtw$}.
We illustrate this methodology with a~few selected computations. To prevent ambiguity, we use $[-,-]^{\zeta}$ and
$[\![-,-]\!]^{\zeta}$ to denote the supercommutator and the $q$-supercommutator of elements in \smash{$\URtw$},
cf.~\eqref{eq:super commutator}, \eqref{eq:q-gl-superbracket}.
\begin{itemize}\itemsep=0pt
\item[1.]\textbf{Commutativity of Cartan generators:} $\bigl[q^{\bfH_{i}}, q^{\bfH_{j}}\bigr]^{\zeta} = 0$.

We compute the $(E_{ii} \otimes E_{jj})$-component of~\eqref{eq:gl-rtt-rel-1} for the ``$+$''-sign case (recalling that $L^{+}$ is upper triangular)
\begin{itemize}\itemsep=0pt

\item[$\bullet$]\textbf{LHS:}
 $\bigl( q^{-(\varepsilon_{i},\varepsilon_{j})} E_{ii} \otimes E_{jj} \bigr)
 \bigl( l_{ii}^{+}l_{jj}^{+} E_{ii} \otimes E_{jj} \bigr)$;

\item[$\bullet$]\textbf{RHS:}
 $\bigl( l_{jj}^{+}l_{ii}^{+} E_{ii} \otimes E_{jj} \bigr)
 \bigl( q^{-(\varepsilon_{i},\varepsilon_{j})} E_{ii} \otimes E_{jj} \bigr)$.

\end{itemize}
Equating both sides yields $l_{ii}^{+}l_{jj}^{+} = l_{jj}^{+}l_{ii}^{+}$, which gives
$q^{\bfH_{i}}q^{\bfH_{j}} = q^{\bfH_{j}}q^{\bfH_{i}}$.

\item[2.]\textbf{Inverses:} $q^{\pm \bfH_{i}}q^{\mp \bfH_{i}} = 1$.

This follows directly from the diagonal identity~\eqref{eq:gl-rtt-diag-rel}.

\item[3.]\textbf{Cartan action on root vectors:}
$q^{\bfH_{a}} \bfe_{ij} q^{-\bfH_{a}} = q^{(\varepsilon_{a}, \varepsilon_{i}-\varepsilon_{j})} \bfe_{ij}$.

We compute the $(E_{aa} \otimes E_{ij})$-component of~\eqref{eq:gl-rtt-rel-1} for the ``$+$''-sign case
(recalling that $L^+$ is upper triangular)
\begin{itemize}\itemsep=0pt

\item[$\bullet$]\textbf{LHS:}
 $\bigl( q^{-(\varepsilon_{a},\varepsilon_{i})} E_{aa} \otimes E_{ii} \bigr)
 \bigl( l_{aa}^{+}l_{ij}^{+} E_{aa} \otimes E_{ij} \bigr)$;

\item[$\bullet$]\textbf{RHS:}
 $\bigl( l_{ij}^{+} l_{aa}^{+} E_{aa} \otimes E_{ij} \bigr)
 \bigl( q^{-(\varepsilon_{a},\varepsilon_{j})} E_{aa} \otimes E_{jj} \bigr)$.

\end{itemize}
Equating both sides yields
 \smash{$q^{-(\varepsilon_{a},\varepsilon_{i})}l_{aa}^{+}l_{ij}^{+} = q^{-(\varepsilon_{a},\varepsilon_{j})}l_{ij}^{+}l_{aa}^{+}$},
which simplifies to $\smash{l_{aa}^{+}l_{ij}^{+}\bigl(l_{aa}^{+}\bigr)^{-1} }= \smash{q^{(\varepsilon_{a},\varepsilon_{i}-\varepsilon_{j})}l_{ij}^{+}}$.
This establishes the desired identity
 \smash{$q^{\bfH_{a}}\bfe_{ij}q^{-\bfH_{a}} = q^{(\varepsilon_{a},\varepsilon_{i}-\varepsilon_{j})}\bfe_{ij}$}.
\end{itemize}

To exhibit the rest of aforementioned identities, we summarize the relevant matrix components and the
resulting twisted identities in Table~\ref{table:A1} (the identities verified above are also
included for ease of~reference).

\begin{table}[t]
\centering
\small
\renewcommand{\arraystretch}{1.2}
{\thickmuskip=1mu
\begin{tabularx}{\textwidth}{|l l >{\raggedright\arraybackslash}X r|}
 \hline
 Matrix Component & Conditions & Identity & \\
 \hline
 $E_{ii} \otimes E_{jj}$
 & $1 \leq i,j \leq N$
 & $\bigl[q^{\bfH_{i}},q^{\bfH_{j}}\bigr]^{\zeta} = 0$
 & \AddLabel{eq:A-twisted-1} \\
 \quad
 & $1 \leq i \leq N$
 & $q^{\pm \bfH_{i}}q^{\mp \bfH_{i}} = 1$
 & \AddLabel{eq:A-twisted-2} \\
 $E_{aa} \otimes E_{ij}$
 & $1 \leq i < j \leq N$, $1 \leq a \leq N$
 & $q^{\bfH_{a}} \bfe_{ij} q^{-\bfH_{a}} = q^{(\varepsilon_{a}, \varepsilon_{i}-\varepsilon_{j})} \bfe_{ij}$
 & \AddLabel{eq:A-twisted-3} \\
 $E_{ij} \otimes E_{j,j+1}$
 & $1 \leq i < j < N$
 & $\bfe_{i,j+1} = (-1)^{(\overline{i}+\overline{j})(\overline{j}+\overline{j+1})} [\![ \bfe_{ij}, \bfe_{j,j+1} ]\!]^{\zeta}$
 & \AddLabel{eq:A-twisted-4} \\
 $E_{i,i+1} \otimes E_{j,j+1}$
 & $1 \leq i < j < N$, $j-i > 1$
 & $[\![ \bfe_{i,i+1}, \bfe_{j,j+1} ]\!]^{\zeta} = 0$
 & \AddLabel{eq:A-twisted-5} \\
 $E_{i,i+1} \otimes E_{i,i+1}$
 & $1 \leq i < N$, $|\bfe_{i,i+1}| = \bbar{1}$
 & $[\![ \bfe_{i,i+1}, \bfe_{i,i+1} ]\!]^{\zeta} = 0$
 & \AddLabel{eq:A-twisted-6} \\
 $E_{i,i+1} \otimes E_{i,i+2}$
 & $1 \leq i \leq N-2$
 & $\bigl[\!\bigl[ \bfe_{i,i+1}, [\![ \bfe_{i,i+1}, \bfe_{i+1,i+2} ]\!]^{\zeta}\bigr]\!\bigr]^{\zeta} = 0$
 & \AddLabel{eq:A-twisted-7} \\
 $E_{i,i+1} \otimes E_{i-1,i+1}$
 & $2 \leq i \leq N-1$
 & $\bigl[\!\bigl[[\![ \bfe_{i-1,i}, \bfe_{i,i+1} ]\!]^{\zeta}, \bfe_{i,i+1} \bigr]\!\bigr]^{\zeta} = 0$
 & \AddLabel{eq:A-twisted-8} \\
 $E_{i,i+1} \otimes E_{i-1,i+2}$
 & $2 \leq i \leq N-2$
 & $\bigl[\!\bigl[\bigl[\!\bigl[[\![ \bfe_{i-1,i}, \bfe_{i,i+1} ]\!]^{\zeta}, \bfe_{i+1,i+2} \bigr]\!\bigr]^{\zeta}, \bfe_{i,i+1} \bigr]\!\bigr]^{\zeta} = 0$
 & \AddLabel{eq:A-twisted-9} \\
 $E_{ij} \otimes E_{ij}$
 & $1 \leq i < j \leq N$, $|\bfe_{ij}| = \bbar{1}$
 & $[\![ \bfe_{ij}, \bfe_{ij} ]\!]^{\zeta} = 0$
 & \AddLabel{eq:A-twisted-10} \\
 \hline
\end{tabularx}}
\caption{$A$-type, identities in $\URtw$ derived from $RL_{1}^{+}L_{2}^{+} = L_{2}^{+}L_{1}^{+}R$ and~\eqref{eq:gl-rtt-diag-rel}.}
\label{table:A1}
\end{table}

\begin{Remark}\label{rem:A-rel-nilpotent}While~\eqref{eq:A-twisted-6} is a~special case of~\eqref{eq:A-twisted-10}, it is recorded separately as it corresponds to the
Serre relations. We also note that~\eqref{eq:A-twisted-10} implies $\bfe_{ij}^{2} = 0$ for $|\bfe_{ij}| = \bbar{1}$.
\end{Remark}

\begin{Remark}\label{rem:table-higher-order-rel}The relations~\eqref{eq:A-twisted-1}--\eqref{eq:A-twisted-4} are applied repeatedly to derive the higher-order relations
in Table~\ref{table:A1}. For instance, the $E_{i,i+1} \otimes E_{i,i+2}$ component of~\eqref{eq:gl-rtt-rel-1} yields
${[\![ \bfe_{i,i+1}, \bfe_{i,i+2} ]\!]^{\zeta} = 0}$, which directly recovers~\eqref{eq:A-twisted-7} upon substituting
the iterative formula~\eqref{eq:A-twisted-4} and removing the overall sign.
\end{Remark}

\begin{Remark}Similar relations for $\bff_{ij}$'s can be obtained by applying Remark~\ref{rem:U(R) involution} to Table~\ref{table:A1}.
We note that while
\[
(-1)^{(\overline{i}+\overline{i+1})(\overline{j}+\overline{j+1})} [\bfe_{i,i+1}, \bff_{j+1,j}]^{\zeta}
 = \delta_{ij} \frac{q^{\bfH_{i,i+1}} - q^{-\bfH_{i,i+1}}}{q-q^{-1}}
\]
can be derived alike, it will not be needed.
\end{Remark}

\subsection{Homomorphism theorem: general linear case}\label{ssec:homomorphism thm}

The main result of this section is the identification of the Hopf superalgebras $\uqgl^{\CF}$ and~$U(R)$.

\begin{Theorem}\label{thm:finite DrJ to RTT}\quad
\begin{itemize}\itemsep=0pt
\item[$(a)$] The assignment
\begin{gather}\label{eq:xi_efh}
 e_{i} \mapsto \bfe_{i,i+1},\qquad f_{i} \mapsto \bff_{i+1,i},\qquad q^{\pm H_{i}} \mapsto q^{\pm\bfH_{i}}
\end{gather}
gives rise to a~$Q$-graded Hopf superalgebra isomorphism \smash{$\xi \colon \uqgl^{\CF} \iso U(R)$}.

\item[$(b)$] The images of the quantum root vectors from Section~{\rm\ref{ssec:universal R}} under $\xi$ are given by
$
 \xi(e_{\gamma_{ij}}) = \bfe_{ij}$, $ \xi(f_{\gamma_{ij}}) = \bff_{ji}$.

\item[$(c)$] The isomorphism $\xi$ is compatible with the pairings of
Remarks~{\rm\ref{rem:wt-pairing-Dr-twist}} and {\rm\ref{rem:RTT double full Cartan}}
\begin{gather*}
 (\xi(x),\xi(y))_{R} = (x,y)^{\CF}_{\widetilde{\DrJ}} \qquad \text{for all} \ x \in U_{q}^{\geq}(\gl(V))^{\CF},\ y \in U_{q}^{\leq}(\gl(V))^{\CF}.
\end{gather*}
\end{itemize}
\end{Theorem}

\begin{Remark}
Since the elements $l^{\pm}_{ii}$ are even, the defining formulas~\eqref{eq:RTT ef gen} for $\bfe_{ij}$ and $\bff_{ji}$
still hold when the multiplication in the left-hand sides is replaced by that of $U(R)$, instead of that of $\URtw$.
\end{Remark}

\begin{Remark}\label{rem:omega_R_action}
For the later use in the proof of Theorem~\ref{thm:finite DrJ to RTT}, we record here how the superalgebra isomorphism
$\omega_{R}$ of~\eqref{eq:transpose-iso} acts on the elements $\bfe_{ij}$, $\bff_{ji}$ of~\eqref{eq:RTT ef gen}, and $q^{\pm \bfH_{k}}$.
We have~${\omega_{R}\bigl(q^{\pm \bfH_{k}}\bigr) = q^{\mp \bfH_{k}}}$ and
\begin{align*}
 \omega_{R}(\bfe_{ij}) &=
 -(-1)^{\overline{i}+\cdots+\overline{j-1}}(-1)^{\overline{i}(\overline{i}+\overline{j})} q^{(\varepsilon_{i}, \varepsilon_{i})/2}q^{(\varepsilon_{j}, \varepsilon_{j})/2}q^{2(\rho, \varepsilon_{i}-\varepsilon_{j})} \bff_{ji},\\
 \omega_{R}(\bff_{ji})
 &= -(-1)^{\overline{i}+\cdots+\overline{j-1}}(-1)^{\overline{j}(\overline{i}+\overline{j})} q^{-(\varepsilon_{i}, \varepsilon_{i})/2}q^{-(\varepsilon_{j}, \varepsilon_{j})/2}q^{-2(\rho, \varepsilon_{i}-\varepsilon_{j})} \bfe_{ij}.
\end{align*}
In particular, we have (cf.~\eqref{eq:weyl-vector-prop})
\begin{align*}
 \omega_{R}(\bfe_{i,i+1}) &= -(-1)^{\overline{i}\, \overline{i+1}} q^{3(\rho, \alpha_{i})}\bff_{i+1,i},\\
 \omega_{R}(\bff_{i+1,i}) &= -(-1)^{\overline{i}\, \overline{i+1}}(-1)^{\overline{i}+\overline{i+1}} q^{-3(\rho, \alpha_{i})} \bfe_{i,i+1}.
\end{align*}
\end{Remark}

\begin{proof}[Proof of Theorem~\ref{thm:finite DrJ to RTT}]
We first prove that the assignment~\eqref{eq:xi_efh} gives rise to a~morphism of Hopf superalgebras
$U_{q}^{\geq}(\gl(V))^{\CF} \to U^{\geq}(R)$. By untwisting the relations in Table~\ref{table:A1} via (cf.~\eqref{eq:twisted-mult})
\begin{gather}\label{eq:untwisting}
 [\![ x,y ]\!]^{\zeta} = (-1)^{|x||y|} [\![ x,y ]\!],
\end{gather}
we obtain the following:
\begin{itemize}\itemsep=0pt

\item
equations~\eqref{eq:A-twisted-1}--\eqref{eq:A-twisted-3} imply the Chevalley-type
relations~\eqref{eq:q-gl-chevalley-rel-HH}--\eqref{eq:q-gl-chevalley-rel-He} for $U_{q}^{\geq}(\gl(V))^{\CF}$;

\item
equations~\eqref{eq:A-twisted-5}--\eqref{eq:A-twisted-9} imply the Serre
relations~\eqref{eq:q-gl-serre-rel-standard-1}--\eqref{eq:q-gl-serre-rel-standard-3} for $U_{q}^{\geq}(\gl(V))^{\CF}$.

\end{itemize}
This proves that $\xi$ gives a~superalgebra morphism $U_{q}^{\geq}(\gl(V))^{\CF} \to U^{\geq}(R)$.
Moreover, the following computations show that $\xi$ also intertwines
the comultiplications:
\begin{gather*}
 \Delta\bigl(q^{\bfH_{k}}\bigr)
 = \Delta\bigl(l^{+}_{kk}\bigr) = l^{+}_{kk} \otimes l^{+}_{kk} = q^{\bfH_{k}} \otimes q^{\bfH_{k}},\\
 \Delta(\bfe_{i,i+1})
 = (-1)^{\overline{i}\, \overline{i+1}} q^{(\varepsilon_{i},\varepsilon_{i})} q^{(\rho, \varepsilon_{i} - \varepsilon_{i+1})} \bigl(q-q^{-1}\bigr)^{-1} \cdot \Delta\bigl(l^{-}_{ii}\bigr) \Delta\bigl(l^{+}_{i,i+1}\bigr)\\
\phantom{\Delta(\bfe_{i,i+1})}{} = (-1)^{\overline{i}\, \overline{i+1}} q^{(\varepsilon_{i},\varepsilon_{i})} q^{(\rho, \varepsilon_{i} - \varepsilon_{i+1})} \bigl(q-q^{-1}\bigr)^{-1} \cdot \bigl(1 \otimes l^{-}_{ii}l^{+}_{i,i+1}
 + l^{-}_{ii}l^{+}_{i,i+1} \otimes l^{-}_{ii}l^{+}_{i+1,i+1}\bigr)\\
 \phantom{\Delta(\bfe_{i,i+1})}{}= 1 \otimes \bfe_{i,i+1} + \bfe_{i,i+1} \otimes q^{-\bfH_{i,i+1}}.
\end{gather*}
Furthermore, untwisting~\eqref{eq:A-twisted-4} via~\eqref{eq:untwisting} gives
$\bfe_{i,j+1} = [\![ \bfe_{ij}, \bfe_{j,j+1} ]\!]$, which is compatible with \eqref{eq:quantum root vectors},
thus establishing part (b) for the positive Borel subalgebras. We note that this iterative formula
also proves the surjectivity of $\xi \colon U_{q}^{\geq}(\gl(V))^{\CF} \to U^{\geq}(R)$.

To derive the analogous results for the negative Borel subalgebras, we utilize the isomorphisms from
Remarks~\ref{rem:DrJ involution} and~\ref{rem:U(R) involution}. Specifically, using Remark~\ref{rem:omega_R_action},
one can check that the composed morphism
\begin{gather*}
 U_{q}^{\leq}(\gl(V))^{\CF} \xrightarrow{\omega_{\DrJ}} \bigl(U_{q}^{\geq}(\gl(V))^{\CF}\bigr)^{\copp} \xrightarrow{\xi} \bigl(U^{\geq}(R)\bigr)^{\copp}
 \xrightarrow{\omega_{R}^{-1}} U^{\leq}(R)
\end{gather*}
coincides with $\xi$ on the generators. This proves that the assignment~\eqref{eq:xi_efh} gives rise to a~well-defined Hopf
superalgebra morphism $U_{q}^{\leq}(\gl(V))^{\CF} \to U^{\leq}(R)$. Moreover, this morphism is automatically surjective since
both $\omega_{\DrJ}$ and $\omega_{R}$ restrict to isomorphisms between the respective positive and negative Borel subalgebras
(cf.~Remarks~\ref{rem:DrJ involution subalgebras} and~\ref{rem:RTT double full Cartan}). Furthermore, we obtain an analogous
iterative formula
 \smash{$\bff_{j+1,i} =
 - (-1)^{(\overline{i}+\overline{j})(\overline{j}+\overline{j+1})} q^{(\varepsilon_{j},\varepsilon_{j})} [\![ \bff_{ji}, \bff_{j+1,j} ]\!]$},
which is also compatible with~\eqref{eq:quantum root vectors}. This completes the proof of part (b) for the negative Borel
subalgebras.

Let us now compare the pairings. Since both pairings are of $Q$-degree zero, cf.~\eqref{eq:skew-pairing-degree-zero}, we get
\begin{gather*}
 \bigl(\bfe_{i,i+1},q^{\pm \bfH_{k}}\bigr)_{R} = 0 = \bigl(e_{i},q^{\pm H_{k}}\bigr)^{\CF}_{\widetilde{\DrJ}},\qquad
 \bigl(q^{\pm \bfH_{k}},\bff_{i+1,i}\bigr)_{R} = 0 = \bigl(q^{\pm H_{k}},f_{i}\bigr)^{\CF}_{\widetilde{\DrJ}}.
\end{gather*}
Moreover, we have
\begin{gather*}
 \bigl(q^{\bfH_{i}},q^{-\bfH_{j}}\bigr)_{R} = \bigl(l^{+}_{ii},l^{-}_{jj}\bigr)_{R} = q^{-(\varepsilon_{i},\varepsilon_{j})} =
 \bigl(q^{H_{i}},q^{-H_{j}}\bigr)^{\CF}_{\widetilde{\DrJ}}.
\end{gather*}
Finally,
\begin{align*}
 \bigl( l_{ii}^{-} l_{i,i+1}^{+}, l_{j+1,j}^{-} l_{jj}^{+} \bigr)_{R}
 & = \bigl( l_{ii}^{-} \otimes l_{i,i+1}^{+}, \Delta\bigl(l_{j+1,j}^{-}\bigr) \Delta\bigl(l_{jj}^{+}\bigr) \bigr)_{R} \\
\overset{\eqref{eq:skew-pairing-degree-zero}}&{=}
 \bigl( l_{ii}^{-} \otimes l_{i,i+1}^{+}, l_{j+1,j+1}^{-} l_{jj}^{+} \otimes l_{j+1,j}^{-} l_{jj}^{+} \bigr)_{R}\\
 & = \bigl( l_{ii}^{-}, l_{j+1,j+1}^{-} l_{jj}^{+} \bigr)_{R} \bigl( l_{i,i+1}^{+}, l_{j+1,j}^{-} l_{jj}^{+} \bigr)_{R}
 = q^{-(\varepsilon_{i},\alpha_{j})} \bigl( l_{i,i+1}^{+}, l_{j+1,j}^{-} l_{jj}^{+} \bigr)_{R}\\
\overset{\eqref{eq:skew-pairing-degree-zero}}&{=}
 q^{-(\varepsilon_{i},\alpha_{j})} \bigl( l_{i,i+1}^{+} \otimes l_{ii}^{+}, l_{j+1,j}^{-} \otimes l_{jj}^{+} \bigr)_{R}
 = q^{(\varepsilon_{i},\varepsilon_{j+1})} \bigl( l_{i,i+1}^{+}, l_{j+1,j}^{-} \bigr)_{R}\\
 & = -\delta_{ij} q^{(\varepsilon_{i},\varepsilon_{j+1})} (-1)^{\overline{i}+\overline{i+1}} \cdot (-1)^{\overline{i+1}} \bigl(q-q^{-1}\bigr)
 = -\delta_{ij} (-1)^{\overline{i}} \bigl(q-q^{-1}\bigr),
\end{align*}
which implies
\[
 (\bfe_{i,i+1},\bff_{j+1,j})_{R} = \delta_{ij} \bigl(q-q^{-1}\bigr)^{-1} = (e_{i},f_{j})^{\CF}_{\widetilde{\DrJ}}.
 \]
This establishes the compatibility of the pairings on the generators, which completes the proof of part~(c).

Thus, $\xi$ defines a~surjective Hopf superalgebra morphism between the corresponding doubles
\begin{gather*}
 \xi \colon\ \CD^{\CF}_{\DrJ} = U^\geq_{q}(\gl(V))^{\CF} \otimes U^\leq_{q}(\gl(V))^{\CF}
 \longrightarrow U^{\geq}(R) \otimes U^{\leq}(R) = \CD_{R},
\end{gather*}
see Remark~\ref{rem:wt-pairing-Dr-twist}. Since the pairing $(-,-)^{\CF}_{\widetilde{\DrJ}}$ is non-degenerate,
$\xi \colon \CD^{\CF}_{\DrJ} \to \CD_{R}$ must be injective on each Borel subalgebra, and so is an isomorphism.
Finally, the identification of the Cartan subalgebras yields an isomorphism $\xi \colon \uqgl^{\CF} \iso U(R)$,
thus completing the proof of part~(a).\looseness=-1
\end{proof}

\subsection{Inverse morphism: general linear case}\label{ssec:inverse morphism}

Following~\cite{df}, we shall now construct a~morphism in the opposite direction
(whose construction requires working in the $\hbar$-adic completion, as emphasized in Remark~\ref{rem:h-adic-completion-avoidance}).

\begin{Proposition}\label{prop:gl-inverse-morphism}
The assignment
\begin{gather}\label{eq:DF morphism}
 L^{+} \mapsto (\Id \otimes \varrho)(\fR_{(21)}),\qquad
 L^{-} \mapsto (\Id \otimes \varrho)\bigl(\fR^{-1}\bigr),
\end{gather}
gives rise to a~$Q$-graded Hopf superalgebra morphism $\phi^{\DF} \colon U(R)^{\copp} \to \uqgl$.
\end{Proposition}

\begin{proof}
We first verify that the assignment~\eqref{eq:DF morphism} defines a~superalgebra morphism. The explicit
formula~\eqref{eq:gl-universal-Rs Ru} (together with an $\hbar$-adic computation, see Remark~\ref{rem:h-adic-completion-avoidance})
ensures that~$\phi^{\DF}\bigl(L^{+}\bigr)$ and~$\phi^{\DF}\bigl(L^{-}\bigr)$ are indeed upper and lower triangular matrices, respectively.
Furthermore, since the diagonal components of $(\Id \otimes \varrho)((\fR_{u})_{(21)})$ and
$(\Id \otimes \varrho)\bigl(\fR_{u}^{-1}\bigr)$ reduce to the identity matrix, the diagonal components of
$(\Id \otimes \varrho)(\fR_{(21)})$ and $(\Id \otimes \varrho)\bigl(\fR^{-1}\bigr)$ simplify respectively to $(\Id \otimes \varrho)((\fR_{s})_{(21)}) = (\Id \otimes \varrho)(\fR_{s})$ and $(\Id \otimes \varrho)\bigl(\fR_{s}^{-1}\bigr)$,
verifying~\eqref{eq:gl-rtt-diag-rel}.

From the Yang--Baxter equation~\eqref{eq:YBE}, we deduce the following equalities:
\begin{gather}\label{eq:YBE variants}
 \fR_{(23)}\fR_{(21)}\fR_{(31)} = \fR_{(31)}\fR_{(21)}\fR_{(23)},\qquad
 \fR_{(23)}\fR_{(12)}^{-1}\fR_{(13)}^{-1} = \fR_{(13)}^{-1}\fR_{(12)}^{-1}\fR_{(23)},
\end{gather}
with the former obtained by applying $\tau_{(12)}\tau_{(23)}$ to~\eqref{eq:YBE}.
Applying $\Id \otimes \varrho^{\otimes 2}$ to~\eqref{eq:YBE variants} verifies~\eqref{eq:gl-rtt-rel-1}
\begin{gather*}
 R_{12} \phi^{\DF}\bigl(L^{\pm}\bigr)_{1} \phi^{\DF}\bigl(L^{\pm}\bigr)_{2} = \phi^{\DF}\bigl(L^{\pm}\bigr)_{2} \phi^{\DF}\bigl(L^{\pm}\bigr)_{1} R_{12}.
\end{gather*}
We note that because $\Id \otimes \varrho^{\otimes 2}$ maps the second and third tensor factors of $\uqgl^{\hat{\otimes} 3}$ to
the last two factors in \smash{$U(R)\otimes \End(V_{q})^{\otimes 2}$}, the element $\fR_{(23)}$ is sent to $R_{12}$.

Likewise, applying \smash{$\tau_{(12)}$} to~\eqref{eq:YBE} and multiplying \smash{$\fR^{-1}_{(13)}$} on both sides, we get
$\smash{\fR_{(23)}\fR_{(21)}\fR_{(13)}^{-1} }= \fR_{(13)}^{-1}\fR_{(21)}\fR_{(23)}$.
Applying \smash{$\Id \otimes \varrho^{\otimes 2}$} to the latter equality then verifies~\eqref{eq:gl-rtt-rel-2}
\begin{gather*}
 R_{12} \phi^{\DF}\bigl(L^{+}\bigr)_{1} \phi^{\DF}\bigl(L^{-}\bigr)_{2} = \phi^{\DF}\bigl(L^{-}\bigr)_{2} \phi^{\DF}\bigl(L^{+}\bigr)_{1} R_{12}.
\end{gather*}

To verify that $\phi^{\DF}$ also intertwines the respective comultiplications, we note that~\eqref{eq:R comult} implies
\[
 (\Delta \otimes \Id)\fR_{(21)} = \fR_{(32)}\fR_{(31)},\qquad
 (\Delta \otimes \Id)\fR^{-1} = \fR_{(23)}^{-1}\fR_{(13)}^{-1}.
 \]
Applying \smash{$\Id^{\otimes 2} \otimes \varrho$} to this yields
 \smash{$(\Delta \otimes \Id)\phi^{\DF}\bigl(L^{\pm}\bigr) = \phi^{\DF}\bigl(L^{\pm}\bigr)_{(2)}\phi^{\DF}\bigl(L^{\pm}\bigr)_{(1)}$}.
This completes the proof.
\end{proof}

To establish the relation between $\phi^{\DF}$ and $\xi$, we start with the following explicit computation.

\begin{Lemma}\label{lem:DF gl image}
The images of the generators $\{l^{\pm}_{ii}\}_{i=1}^{N}$ and $\{l^{+}_{i,i+1}, l^{-}_{i+1,i} \}_{i=1}^{N-1}$ under
$\phi^{\DF}$ are as follows:
\begin{gather*}
 \phi^{\DF}\colon \
 l^{\pm}_{ii} \mapsto q^{\mp H_{i}},\qquad
 l^{+}_{i,i+1} \mapsto -\bigl(q-q^{-1}\bigr) f_{i}q^{-H_{i+1}},\\
 l^{-}_{i+1,i} \mapsto (-1)^{\overline{i+1}} \bigl(q-q^{-1}\bigr) q^{H_{i+1}} e_{i}.
\end{gather*}
\end{Lemma}

\begin{proof}
We begin by evaluating the semisimple part. A~direct $\hbar$-adic computation yields (cf.~\eqref{eq:gl-Lie-action-Cartan})
\begin{gather*}
 (\Id \otimes \varrho)(\fR_{s}) = \sum_{1\leq i\leq N} q^{-H_{i}} \otimes E_{ii},
\end{gather*}
so that $\phi^{\DF}\bigl(l^{\pm}_{ii}\bigr) = q^{\mp H_{i}}$.
Next, we evaluate the unipotent part. We first determine the $E_{i,i+1}$-component of
\begin{gather*}
 (\Id \otimes \varrho)((\fR_{u})_{(21)})
 = \prod_{\gamma \in \Phi^{+}}^{\leftarrow} \left( \sum^{k\geq 0}_{k \leq 1 \text{ if } \gamma \in \Phi_{\bbar{1}} }
 (-1)^{|e_{\gamma}^{k}|} \frac{f_\gamma^k \otimes \varrho\bigl(e_\gamma^k\bigr)}{\bigl(f_\gamma^k, e_\gamma^k\bigr)_{\DrJ}} \right).
\end{gather*}
Since $\varrho$ preserves the $Q$-degree and $\deg(E_{i,i+1}) = \alpha_{i}$ is a~simple root, the only term contributing
to this component is
\begin{gather*}
 (-1)^{|e_{\alpha_{i}}|} \frac{f_{\alpha_{i}} \otimes \varrho(e_{\alpha_{i}})}{(f_{\alpha_{i}}, e_{\alpha_{i}})_{\DrJ}}
 = -\bigl(q-q^{-1}\bigr) f_{i} \otimes E_{i,i+1},
\end{gather*}
which implies that $\phi^{\DF}\bigl(l^{+}_{i,i+1}\bigr) = -\bigl(q-q^{-1}\bigr) f_{i}q^{-H_{i+1}}$.
Likewise, let us compute the $E_{i+1,i}$-component of
\begin{align*}
 (\Id \otimes \varrho)\bigl(\fR_{u}^{-1}\bigr)
 &= \prod_{\gamma \in \Phi^{+}}^{\rightarrow} \left( \sum^{k\geq 0}_{k \leq 1 \text{ if } \gamma \in \Phi_{\bbar{1}} }
 \frac{e_\gamma^k \otimes \varrho\bigl(f_\gamma^k\bigr)}{\bigl(f_\gamma^k, e_\gamma^k\bigr)_{\DrJ}} \right)^{-1}
\\
& = \prod_{\gamma \in \Phi^{+}}^{\rightarrow}
 \sum_{r \geq 0} \left( -\sum^{k\geq 1}_{k \leq 1 \text{ if } \gamma \in \Phi_{\bbar{1}} }
 \frac{e_\gamma^k \otimes \varrho\bigl(f_\gamma^k\bigr)}{\bigl(f_\gamma^k, e_\gamma^k\bigr)_{\DrJ}} \right)^{r},
\end{align*}
where we use the fact that
\[
\sum^{k\geq 1}_{k \leq 1 \text{ if } \gamma \in \Phi_{\bbar{1}} }
 \frac{e_\gamma^k \otimes \varrho\big(f_\gamma^k\big)}{\big(f_\gamma^k, e_\gamma^k\big)_{\DrJ}}
 \]
 is nilpotent.
As above, the only contributing term is thus
\begin{gather*}
 - \frac{e_{\alpha_{i}} \otimes \varrho(f_{\alpha_{i}})}{(f_{\alpha_{i}}, e_{\alpha_{i}})_{\DrJ}}
 = (-1)^{\overline{i}+\overline{i+1}} \bigl(q-q^{-1}\bigr) e_{i} \otimes (-1)^{\overline{i}} E_{i+1,i}
 = (-1)^{\overline{i+1}} \bigl(q-q^{-1}\bigr) e_{i} \otimes E_{i+1,i},
\end{gather*}
which implies that \smash{$\phi^{\DF}\bigl(l^{-}_{i+1,i}\bigr) = (-1)^{\overline{i+1}} \bigl(q-q^{-1}\bigr) q^{H_{i+1}} e_{i}$}.
This completes the proof.
\end{proof}

We shall now exhibit $\phi^{\DF}$ as the inverse of $\xi$, up to canonical superalgebra automorphisms.

\begin{Proposition}\label{prop:almost-inverse-gl}
The composition
\begin{gather*}
 \uqgl \xrightarrow[\eqref{eq:twist-identification}]{\phi_{\CF}^{-1}} \uqgl^{\CF}
 \xrightarrow[\eqref{eq:xi_efh}]{\xi} U(R)
 \xrightarrow[\eqref{eq:transpose-iso}]{\omega_{R}^{-1}} U(R)^{\copp}
 \xrightarrow[\eqref{eq:DF morphism}]{\phi^{\DF}} \uqgl
\end{gather*}
is the identity morphism.
\end{Proposition}

\begin{proof}
This is verified by direct computation (cf.~\eqref{eq:weyl-vector-prop} and Remark~\ref{rem:omega_R_action})
\begin{align*}
 & q^{\pm H_{k}} \xmapsto{\phi_{\CF}^{-1}} q^{\pm H_{k}} \xmapsto{\xi} q^{\pm \bfH_{k}}
 \xmapsto{\omega_{R}^{-1}} q^{\mp \bfH_{k}} \xmapsto{\phi^{\DF}} q^{\pm H_{k}},\\
 & e_{i} \xmapsto{\phi_{\CF}^{-1}} e_{i}q^{h_{i}} \xmapsto{\xi} \bfe_{i,i+1} q^{\bfH_{i,i+1}} \xmapsto{\omega_{R}^{-1}}
 -(-1)^{\overline{i}\, \overline{i+1}} (-1)^{\overline{i}+\overline{i+1}} q^{3(\rho, \alpha_{i})} \bff_{i+1,i} q^{-\bfH_{i,i+1}}
 \xmapsto{\phi^{\DF}} e_{i},\\
 & f_{i} \xmapsto{\phi_{\CF}^{-1}} q^{-h_{i}}f_{i} \xmapsto{\xi} q^{-\bfH_{i,i+1}} \bff_{i+1,i} \xmapsto{\omega_{R}^{-1}}
 -(-1)^{\overline{i}\, \overline{i+1}} q^{-3(\rho, \alpha_{i})} q^{\bfH_{i,i+1}} \bfe_{i,i+1}
 \xmapsto{\phi^{\DF}} f_{i}. \tag*{\qed}
\end{align*}\renewcommand{\qed}{}
\end{proof}

\section{Orthosymplectic Lie superalgebras and quantum supergroups}\label{sec:osp-DJ}

\subsection{Orthosymplectic Lie superalgebras}

As in Section~\ref{ssec:gl-Chevalley-Serre}, we fix a~superspace $V$ of dimension $(m,n)$, and its homogeneous basis
$\{v_{1}, \ldots, v_{N}\}$, where $N = m + n$. We assume that $n$ is even, and define \smash{$s = \bigl\lfloor \frac{N}{2} \bigr\rfloor$}.
Consider the involution $'$ on the index set $\mathbb{I} = \{1,2,\ldots,N\}$ defined by
\[
 i \mapsto i' \coloneqq N+1-i.
 \]
We shall assume that $\overline{i} = \overline{i'}$ for all $i \in \mathbb{I}$. In particular, $v_{s+1}$ is necessarily even when $N$
is odd. We also choose a~sequence $\vartheta_{V} \coloneqq (\vartheta_{1},\vartheta_{2},\ldots,\vartheta_{N}) \in \{\pm 1\}^{N}$
such that $\vartheta_{i} = 1$ if $\overline{i} = \bbar{0}$ and $\vartheta_{i} = -\vartheta_{i'}$ if~${\overline{i} = \bbar{1}}$
(we do not assume $\vartheta_{i} = 1$ for $i \leq s$). From the conditions above, we immediately have
\[
 \vartheta_{i}^{2} = 1 \qquad \text{and} \qquad \vartheta_{i}\vartheta_{i'} = (-1)^{\overline{i}}
 \qquad \text{for all $ i \in \mathbb{I}$}.
\]
Consider a~bilinear form $B_{G} \colon V \times V \to \BC$ defined by the anti-diagonal matrix
\smash{$G = \sum_{i=1}^{N} \vartheta_{i} E_{i'i}$}.
The \emph{orthosymplectic Lie superalgebra} $\fosp(V)$ is defined as the Lie subalgebra of $\gl(V)$ consisting of
all matrices $X$ that preserve the bilinear form $B_{G}$, meaning (see notation~\eqref{eq:supertranspose})
\[
 X^{\st}G + GX = 0.
\]
Explicitly, $\fosp(V)$ is spanned by the elements
\[
X_{ij} = E_{ij} - (-1)^{\overline{i}(\overline{i}+\overline{j})} \vartheta_{i}\vartheta_{j} E_{j'i'}.
\]
Since \smash{$X_{j'i'} = -(-1)^{\overline{i}(\overline{i}+\overline{j})} \vartheta_{i}\vartheta_{j} \cdot X_{ij}$},
the following elements form a~basis of $\fosp(V)$:
\begin{gather*}
 \bigl\{X_{ij} \mid i+j \leq N \bigr\} \cup \bigl\{X_{ii'} \mid \overline{i} = \bbar{1} \bigr\}.
\end{gather*}
We also introduce a~closely related object, the \emph{extended orthosymplectic Lie superalgebra}\footnote{Our choice of
notation $\fgosp$ follows that of~\cite{sss}.}
\begin{gather*}
 \fgosp(V) \coloneqq \fosp(V) \oplus \BC \cdot \ID \subset \gl(V),
\end{gather*}
where $\ID$ denotes the identity matrix.
We note that both $\fosp(V)$ and $\fgosp(V)$ inherit the $\BZ_{2}$-grading from $\gl(V)$. Since $[\ID, X] = 0$ for all
$X \in \fosp(V)$, the orthosymplectic Lie superalgebra~$\fosp(V)$ forms a~codimension $1$ ideal of $\fgosp(V)$.
To avoid confusion, the central element $\ID$ will frequently be denoted by $C$.

Similarly to $\gl(V)$, we choose the Borel subalgebras of $\fosp(V)$ and $\fgosp(V)$ consisting of all upper triangular matrices.
The Cartan subalgebra $\fh_{\fosp}$ of $\fosp(V)$ consists of all diagonal matrices and has basis $\{X_{ii}\}_{i=1}^{s}$,
while the Cartan subalgebra $\fh_{\fgosp}$ of $\fgosp(V)$ is the direct sum~${\fh_{\fosp} \oplus \BC \cdot C}$. Recall the
linear functionals $\{\varepsilon_{i}\}_{i=1}^{N}$ on $\gl(V)$, and let
$\Tilde{\varepsilon}_{i} \coloneqq \varepsilon_{i}|_{\fh_{\fgosp}}$ be their restrictions to the Cartan subalgebra
$\fh_{\fgosp}$ of $\fgosp(V)$. Then
$\Tilde{\varepsilon}_{i} + \Tilde{\varepsilon}_{i'} = \Tilde{\varepsilon}_{j} + \Tilde{\varepsilon}_{j'}$
for all $1 \leq i,j \leq N$.
Thus $\Tilde{\varepsilon}_{C} \coloneqq (\Tilde{\varepsilon}_{i} + \Tilde{\varepsilon}_{i'})/2$ is well-defined
independently from the choice of $i$, and $\Tilde{\varepsilon}_{C} = \Tilde{\varepsilon}_{s+1}$ for odd~$N$.
The set $\{\Tilde{\varepsilon}_{i} - \Tilde{\varepsilon}_{C}\}_{i=1}^{s} \cup \{\Tilde{\varepsilon}_{C}\}$ forms a~basis
of $\fh_{\fgosp}^{*}$ dual to the basis $\{X_{ii}\}_{i=1}^{s} \cup \{C\}$ of $\fh_{\fgosp}$. Moreover, as
$\Tilde{\varepsilon}_{C}|_{\fh_{\fosp}} = 0$, we note that
$\varepsilon_{i}|_{\fh_{\fosp}} = (\Tilde{\varepsilon}_{i} - \Tilde{\varepsilon}_{C})|_{\fh_{\fosp}}$ for all
$1 \leq i \leq N$, and therefore
\[
 \varepsilon_{i}|_{\fh_{\fosp}} = -\varepsilon_{i'}|_{\fh_{\fosp}} \quad \text{for all $i$}, \qquad
 \varepsilon_{s+1}|_{\fh_{\fosp}} = 0 \quad \text{for odd $ N$}.
\]
This implies that $\{\varepsilon_{i}|_{\fh_{\fosp}}\}_{i=1}^{s}$ forms a~basis of $\fh_{\fosp}^{*}$ dual to the basis
$\{X_{ii}\}_{i=1}^{s}$ of $\fh_{\fosp}$.
The functionals $\varepsilon_{i}|_{\fh_{\fosp}} \in \fh_{\fosp}^{*}$ and their extension to $\fh_{\fgosp}^{*}$ by zero on
$C$ will be simply denoted by $\varepsilon_{i}$, so that
\begin{gather}\label{eq:osp-epsilon-def}
 \varepsilon_{i} = \Tilde{\varepsilon}_{i} - \Tilde{\varepsilon}_{C} \in \fh_{\fgosp}^{*} \qquad \forall 1 \leq i \leq N.
\end{gather}

In complete parallel with the $\gl(V)$ case, a~direct computation shows that
$[X_{ii},X_{ab}] = (\varepsilon_{a}-\varepsilon_{b})(X_{ii})X_{ab}$, and thus $X_{ab}$ is a~root vector corresponding to the
root $\varepsilon_{a} - \varepsilon_{b} = \Tilde{\varepsilon}_{a} - \Tilde{\varepsilon}_{b}$. This yields the root space
decomposition $\fosp(V) = \fh_{\fosp} \oplus \bigoplus_{\alpha \in \Phi} \fosp(V)_{\alpha}$ with the root system given~by
\begin{gather*}
 \Phi = \bigl\{\varepsilon_{a} - \varepsilon_{b} \mid a+b \leq N,\, a \neq b \bigr\} \cup
 \bigl\{2\varepsilon_{a} \mid \overline{a} = \bbar{1} \bigr\},
\end{gather*}
and the corresponding polarization
\begin{align*}
 \Phi^{+} &=
 \bigl\{\varepsilon_{a} - \varepsilon_{b} \mid a<b<a' \bigr\} \cup
 \bigl\{2\varepsilon_{a} \mid \overline{a} = \bbar{1},\, a \leq s \bigr\},\\
 \Phi^{-} &=
 \bigl\{\varepsilon_{a} - \varepsilon_{b} \mid b<a<b' \bigr\} \cup
 \bigl\{2\varepsilon_{a} \mid \overline{a} = \bbar{1},\, a' \leq s \bigr\}.
\end{align*}
In particular, $\Phi$ contains any (equivalently all) of $\{\pm \varepsilon_i\}_{i=1}^N$
if and only if $N$ is odd. We use $\Phi_{\bbar{0}}$ and $\Phi_{\bbar{1}}$ to denote the sets of even and odd roots, respectively.
Lastly, we consider the following \emph{reduced} root system:
\begin{gather}\label{eq:osp-reduced-root-sys}
 \bar{\Phi} = \bigl\{ \gamma \in \Phi \mid \tfrac{1}{2}\gamma \notin \Phi \bigr\},\qquad
 \bar{\Phi}_{\bbar{0}} = \bar{\Phi} \cap \Phi_{\bbar{0}},\qquad
 \bar{\Phi}_{\bbar{1}} = \bar{\Phi} \cap \Phi_{\bbar{1}}.
\end{gather}
We also note that $\fgosp(V)$ shares the same root system $\Phi$ and admits the corresponding root space decomposition
\begin{gather*}
 \fgosp(V) = \fh_{\fgosp} \oplus \bigoplus_{\alpha \in \Phi} \fgosp(V)_{\alpha} \qquad \text{with} \quad
 \fgosp(V)_{\alpha} = \fosp(V)_{\alpha} \qquad \forall \alpha \in \Phi.
\end{gather*}

\subsection{Chevalley--Serre-type presentation}

With respect to the above polarization, the simple roots and corresponding root vectors in $\fosp(V)$ are as follows:
\begin{itemize}\itemsep=0pt

\item[$\bullet$]
\textbf{Case 1 ($B$-type):} $m$ is odd.
\begin{gather}
 \alpha_{i} = \varepsilon_{i} - \varepsilon_{i+1}, \qquad \sse_{i} = X_{i,i+1},\quad
 \ssf_{i} = (-1)^{\overline{i}} X_{i+1,i} \qquad \text{for} \ 1 \leq i \leq s, \nonumber\\
 \ssh_{i} =
 (-1)^{\overline{i}} X_{ii} - (-1)^{\overline{i+1}} X_{i+1,i+1} \qquad \text{for} \ 1 \leq i \leq s.\label{eq:Lie-action-case-B}
\end{gather}

\item[$\bullet$]
\textbf{Case 2 ($C$-type):} $m$ is even and $\overline{s} = \bbar{1}$.
\begin{gather}
 \alpha_{i} =
 \begin{cases}
 \varepsilon_{i} - \varepsilon_{i+1} & \text{if } 1 \leq i < s ,\\
 2\varepsilon_{s} & \text{if } i = s,
 \end{cases} \nonumber\\
 \sse_{i} =
 \begin{cases}
 X_{i,i+1} & \text{if } 1 \leq i < s ,\\
 E_{ss'} & \text{if } i = s,
 \end{cases} \qquad
 \ssf_{i} =
 \begin{cases}
 (-1)^{\overline{i}} X_{i+1,i} & \text{if } 1 \leq i < s ,\\
 -2E_{s's} & \text{if } i = s,
 \end{cases} \nonumber\\
 \ssh_{i} =
 \begin{cases}
 (-1)^{\overline{i}} X_{ii} - (-1)^{\overline{i+1}} X_{i+1,i+1} & \text{if } 1 \leq i < s, \\
 -2X_{ss} & \text{if } i = s.
 \end{cases}\label{eq:Lie-action-case-C}
\end{gather}

\item[$\bullet$]
\textbf{Case 3 ($D$-type):} $m$ is even and $\overline{s} = \bbar{0}$.
\begin{gather}
 \alpha_{i} =
 \begin{cases}
 \varepsilon_{i} - \varepsilon_{i+1} & \text{if } 1 \leq i < s ,\\
 \varepsilon_{s-1} + \varepsilon_{s} & \text{if } i = s,
 \end{cases} \nonumber\\
 \sse_{i} =
 \begin{cases}
 X_{i,i+1} & \text{if } 1 \leq i < s ,\\
 X_{s-1,s'} & \text{if } i = s,
 \end{cases} \qquad
 \ssf_{i} =
 \begin{cases}
 (-1)^{\overline{i}} X_{i+1,i} & \text{if } 1 \leq i < s, \\
 (-1)^{\overline{s-1}} X_{s',s-1} & \text{if } i = s,
 \end{cases} \nonumber\\
 \ssh_{i} =
 \begin{cases}
 (-1)^{\overline{i}} X_{ii} - (-1)^{\overline{i+1}} X_{i+1,i+1} & \text{if } 1 \leq i < s ,\\
 (-1)^{\overline{s-1}} X_{s-1,s-1} + (-1)^{\overline{s}} X_{ss} & \text{if } i = s.
 \end{cases}\label{eq:Lie-action-case-D}
\end{gather}

\end{itemize}

While $\{\ssh_{i}\}_{i=1}^{s}$ is a~basis for $\fh_{\fosp}$, it is often more convenient to employ an alternative basis.
To this end, we consider the following elements:
\[
 \ssH_i \coloneqq (-1)^{\overline{i}} X_{ii} = (-1)^{\overline{i}} (E_{ii} - E_{i'i'})
 \qquad \text{for $ 1 \leq i \leq N$}.
\]
Note that $\ssH_{i} = -\ssH_{i'}$ for all $1 \leq i \leq N$, and $\ssH_{s+1} = 0$ when $N$ is odd.
The set $\{\ssH_{i}\}_{i=1}^{s}$ also forms a~basis of $\fh_{\fosp}$, and in terms of this alternative basis,
the elements $\ssh_i$ from~\eqref{eq:Lie-action-case-B}--\eqref{eq:Lie-action-case-D} take the following form:
\begin{gather}\label{eq:h-elements}
 \ssh_i =
 \begin{cases}
 \ssH_i - \ssH_{i+1} & \text{if } 1 \leq i < s \text{ \ (in all $BCD$-types)}, \\
 \ssH_s & \text{if } i = s \text{ \ ($B$-type)}, \\
 2\ssH_s & \text{if } i = s \text{ \ ($C$-type)}, \\
 \ssH_{s-1} + \ssH_s & \text{if } i = s \text{ \ ($D$-type)}.
 \end{cases}
\end{gather}
Furthermore, the lattices $P^{\vee} \coloneqq \BZ C \oplus \bigoplus_{i=1}^{s} \BZ \ssH_{i}$ and
$Q^{\vee} \coloneqq \bigoplus_{i=1}^{s} \BZ \ssh_{i}$ constitute the coweight and coroot lattices of $\fgosp(V)$,
respectively. We also define the lattice
\begin{gather*}
 P \coloneqq \sum_{i=1}^{N} \BZ \Tilde{\varepsilon}_{i} =
 \begin{cases}
 \BZ \Tilde{\varepsilon}_{C} \oplus \bigoplus_{i=1}^{s} \BZ \Tilde{\varepsilon}_{i} & \text{($B$-type)}, \\
 \BZ (2\Tilde{\varepsilon}_{C}) \oplus \bigoplus_{i=1}^{s} \BZ \Tilde{\varepsilon}_{i} & \text{($CD$-types)},
 \end{cases}
\end{gather*}
the root lattice $Q \coloneqq \bigoplus_{i=1}^{s} \BZ\alpha_{i} \subseteq P$, and
$Q^{+} \coloneqq \bigoplus_{i=1}^{s} \BZ_{\geq 0}\alpha_{i}$, the latter giving rise to the partial order on~$P$.
One can verify that $\fgosp(V)$ is graded by $Q$ (and hence also by $P$) via
\begin{gather}\label{eq:osp-gen-Q-grading}
 \deg(\sse_i) = \alpha_i, \qquad \deg(\ssf_i) = -\alpha_i,\qquad
 \deg(\ssh_i) = 0, \qquad \deg(\ssH_i) = 0, \qquad \deg(C) = 0,
\end{gather}
which restricts to a~$Q$-grading on $\fosp(V)$, and is compatible with the $\BZ_{2}$-grading via the group homomorphism
\begin{gather}\label{eq:gosp-grading-compatible}
 P \to \BZ_{2}, \qquad \Tilde{\varepsilon}_{i} \mapsto \overline{i} \qquad \textrm{for all} \ 1 \leq i \leq N,
\end{gather}
cf.~\eqref{eq:gl-grading-compatible}. Specifically, this implies $\Tilde{\varepsilon}_{C} \mapsto \bbar{0}$
in $B$-type, and $2\Tilde{\varepsilon}_{C} \mapsto \bbar{0}$ in $CD$-types.

We define a~non-degenerate bilinear form $(\cdot,\cdot) \colon \fgosp(V) \times \fgosp(V) \to \BC$ by
\begin{gather}\label{eq:osp-str-form}
 (X,Y) = \sfrac{1}{2} \mathrm{sTr}(XY), \qquad
 (C,X) = 0 = (X,C), \qquad
 (C,C) = 1,
\end{gather}
for all $X,Y\in \fosp(V)$. As in the $\gl(V)$ case, its restriction to $\fh_{\fgosp}$ is non-degenerate.
Through the identification $\fh_{\fgosp}^* \simeq \fh_{\fgosp}$ via $\varepsilon_{i} \leftrightarrow \ssH_{i}$
and $\Tilde{\varepsilon}_{C} \leftrightarrow C$, this induces a~bilinear form
$(-,-) \colon \fh_{\fgosp}^{*} \times \fh_{\fgosp}^{*} \to \BC$ satisfying
(for all $1 \leq i,j \leq s$)
\begin{gather}\label{eq:osp-epsilon-pairing}
 (\varepsilon_{i}, \varepsilon_{j}) = \delta_{ij} (-1)^{\overline{i}},\qquad
 (\Tilde{\varepsilon}_{C}, \varepsilon_{i}) = 0 = (\varepsilon_{i}, \Tilde{\varepsilon}_{C}),\qquad
 (\Tilde{\varepsilon}_{C}, \Tilde{\varepsilon}_{C}) = 1.
\end{gather}
We note that the factor $1/2$ in~\eqref{eq:osp-str-form} is used to ensure
the first formula above akin to~\eqref{eq:gl-epsilon-pairing}.

We define the \emph{symmetrized Cartan matrix} by $a_{ij} = (\alpha_{i},\alpha_{j})$ for $1 \leq i,j \leq s$.
The elements~${\{\sse_{i},\ssf_i,\ssh_i\}_{i=1}^s}$ generate $\fosp(V)$ and satisfy the exact same Chevalley-type
relations~\eqref{eq:sl-chevalley-rel} as in the $\ssl(V)$ case. By~\cite[Main Theorem]{z}, this set of generators and relations
provides a~complete presentation of $\fosp(V)$ when augmented with the appropriate Serre relations (which we omit here).

\begin{Remark}
We note that our choice of generators corresponds to a~rescaled version of those used in~\cite{y,z},
as we work with the symmetrized Cartan matrix rather than the non-symmetrized one.
\end{Remark}

\begin{Remark}\label{rem:lattice generator-osp}
As in Remark~\ref{rem:lattice generator}, instead of using Cartan generators $\ssh_{i}$ or $\ssH_{i}$, one can use
an additive subgroup $\Gamma \subseteq \fh_{\fgosp}$ containing $\{\ssh_{i}\}_{i=1}^{s}$ to provide a~uniform Chevalley--Serre
type presentation for both $\fosp(V)$ and $\fgosp(V)$. Consider the Lie superalgebra $\fg(\Gamma)$ over $\BC$ generated by
$\{\sse_{i}, \ssf_{i}\}_{i=1}^{s}$ and $\Gamma$, subject to the Serre relations together with the Chevalley-type relations
as in Remark~\ref{rem:lattice generator}, but for all $\ssH,\ssH' \in \Gamma$ and $1 \leq i,j \leq s$. It is clear that
$\fg\bigl(Q^{\vee}\bigr) = \fosp(V)$ and~${\fg(\Gamma_{\fgosp}) = \fgosp(V)}$, where
\begin{gather*}
 \Gamma_{\fgosp} =
 \begin{cases}
 \BZ C \oplus \bigoplus_{i=1}^{s} \BZ\ssH_{i} & \text{($B$-type)}, \\
 \BZ (2C) \oplus \bigoplus_{i=1}^{s} \BZ(\ssH_{i}+C) & \text{($CD$-types)},
 \end{cases}
\end{gather*}
is the image of $P$ inside $P^{\vee}$ under the identification $\fh_{\fgosp} \simeq \fh_{\fgosp}^*$ induced by the above
non-degenerate form. One can uniformly choose a~different basis of the above lattice
$\Gamma_{\fgosp} = \bigoplus_{i=1}^{s+1} \BZ \Tilde{\ssH}_{i}$ via
\begin{gather}\label{eq:Tilde H}
 \Tilde{\ssH}_i \coloneqq \ssH_{i} + C \qquad \text{for} \ 1 \leq i \leq N.
\end{gather}
\end{Remark}

\subsection{Drinfeld--Jimbo-type orthosymplectic quantum supergroups}

Similarly to the general linear case in Section~\ref{ssec:q-gl}, we define a~class of quantum supergroups~$U_q(\Gamma)$
parameterized by an additive subgroup $\Gamma \subseteq P^{\vee}$ containing the coroot lattice $Q^{\vee}$, see
Remark~\ref{rem:lattice generator-osp}.

\begin{Remark}
As in Remark~\ref{rem:Cartan-notation-gl}, we shall denote the elements $\ssh_i, \ssH_i, \Tilde{\ssH}_i \in \fh_{\fgosp}$ by
$h_i$, $H_i$, $\Tilde{H}_i$, respectively, in the quantum supergroup setting in order to distinguish them from their Lie superalgebra counterparts.
\end{Remark}

\begin{Definition}\label{def:quantum-osp-lattice}
For a~fixed additive subgroup $\Gamma \subseteq P^{\vee}$ containing $Q^{\vee}$, the quantum supergroup~$U_q(\Gamma)$ is
the $\BC(q)$-superalgebra generated by $\{e_i, f_i\}_{i=1}^{s}$ and the formal symbols $\bigl\{q^{\pm H} \mid H \in \Gamma\bigr\}$.
It is equipped with the $Q$-grading (cf.~\eqref{eq:osp-gen-Q-grading})
\begin{gather}\label{eq:uqGamma-osp-Q-grading}
 \deg(e_i) = \alpha_i,\qquad \deg(f_i) = -\alpha_i,\qquad \deg\bigl(q^{\pm H}\bigr) = 0,
\end{gather}
and the compatible $\BZ_2$-grading, see~\eqref{eq:gosp-grading-compatible}. These generators are subject to the Serre
relations specified later in this subsection, alongside the Chevalley-type
relations~\eqref{eq:q-gamma-chevalley-rel-HH}--\eqref{eq:q-gamma-chevalley-rel-ef} as in the $\gl(V)$ case, but for all
$H,H' \in \Gamma$ and $1 \leq i,j \leq s$, where $\alpha_j(H)$ denotes the values of root~$\alpha_j$ on the Cartan element
$H \in P^{\vee}$.
\end{Definition}

To express the Serre relations, we recall the $q$-supercommutator $[\![ -,- ]\!]$ from~\eqref{eq:q-gl-superbracket}.
First of all, we impose~\eqref{eq:q-gl-serre-rel-standard-1} for all $1 \leq i,j \leq s$ as well
as~\eqref{eq:q-gl-serre-rel-standard-2}, \eqref{eq:q-gl-serre-rel-standard-3} provided all featured indices are strictly
less than $s$. The remaining Serre relations involving the index $s$ are specified by the following type-dependent list:
\begin{itemize}\itemsep=0pt

\item \textbf{$B$-type:}
\begin{align}
 & [\![e_{s-1},[\![e_{s-1},e_{s}]\!]]\!]=0, \qquad [\![f_{s-1},[\![f_{s-1},f_{s}]\!]]\!]=0
 \qquad \text{if} \ \alpha_{s-1}\in \Phi_{\bbar{0}},\nonumber \\
 & [\![[\![[\![e_{s-1},e_s]\!],e_s]\!],e_s]\!]=0, \qquad [\![[\![[\![f_{s-1},f_s]\!],f_s]\!],f_s]\!]=0
 \qquad \text{if} \ \alpha_{s}\in \Phi_{\bbar{0}}, \nonumber\\
 & [\![[\![[\![e_{s-2},e_{s-1}]\!],e_{s}]\!],e_{s-1}]\!] = 0,\qquad
 [\![[\![[\![f_{s-2},f_{s-1}]\!],f_{s}]\!],f_{s-1}]\!] = 0
 \quad \text{if } \alpha_{s-1}\in \Phi_{\bbar{1}}.\label{eq:q-osp-serre-B}
\end{align}

\item \textbf{$C$-type:}
\begin{align}
 & [\![e_{s-1},[\![e_{s-1},[\![e_{s-1},e_{s}]\!]]\!]]\!]=0, \qquad [\![f_{s-1},[\![f_{s-1},[\![f_{s-1},f_{s}]\!]]\!]]\!]=0
 \qquad \text{if } \alpha_{s-1}\in \Phi_{\bbar{0}}, \nonumber\\
 & [\![[\![e_{s-1},e_{s}]\!],e_{s}]\!]=0, \qquad [\![[\![f_{s-1},f_{s}]\!],f_{s}]\!]=0, \nonumber\\
 & \begin{cases}
 [\![[\![[\![[\![e_{s-2},e_{s-1}]\!],e_{s}]\!], [\![e_{s-2},e_{s-1}]\!]]\!], e_{s-1}]\!]=0 \\
 [\![[\![[\![[\![f_{s-2},f_{s-1}]\!],f_{s}]\!], [\![f_{s-2},f_{s-1}]\!]]\!], f_{s-1}]\!]=0
 \end{cases}
 \qquad \text{if } \alpha_{s-2},\alpha_{s-1}\in \Phi_{\bbar{1}}, \label{eq:q-osp-serre-C}\\
 & \begin{cases}
 [\![[\![[\![[\![[\![[\![e_{s-3},e_{s-2}]\!],e_{s-1}]\!],e_{s}]\!],e_{s-1}]\!],e_{s-2}]\!],e_{s-1}]\!]=0 \\
 [\![[\![[\![[\![[\![[\![f_{s-3},f_{s-2}]\!],f_{s-1}]\!],f_{s}]\!],f_{s-1}]\!],f_{s-2}]\!],f_{s-1}]\!]=0
 \end{cases}
 { \quad \text{if} \ \alpha_{s-2}\in \Phi_{\bbar{0}},\ \alpha_{s-1}\in \Phi_{\bbar{1}}.}\nonumber
\end{align}

\item \textbf{$D$-type:}
\begin{align}
 & [\![e_{s-2},[\![e_{s-2},e_{s}]\!]]\!]=0, \qquad [\![f_{s-2},[\![f_{s-2},f_{s}]\!]]\!]=0
 \qquad \text{if } \alpha_{s-2}\in \Phi_{\bbar{0}},\nonumber \\
 & [\![[\![e_{s-2},e_{s}]\!],e_{s}]\!]=0, \qquad [\![[\![f_{s-2},f_{s}]\!],f_{s}]\!]=0
 \qquad \text{if } \alpha_{s}\in \Phi_{\bbar{0}}, \nonumber\\
 & [\![[\![[\![e_{s-3},e_{s-2}]\!],e_{s}]\!],e_{s-2}]\!] = 0,\quad
 [\![[\![[\![f_{s-3},f_{s-2}]\!],f_{s}]\!],f_{s-2}]\!] = 0
 \qquad \text{if } \alpha_{s-2}\in \Phi_{\bbar{1}},\nonumber \\
 & \begin{cases}
 [\![[\![e_{s-2},e_{s-1}]\!],e_{s}]\!] &= [\![[\![e_{s-2},e_{s}]\!],e_{s-1}]\!], \\
 [\![[\![f_{s-2},f_{s-1}]\!],f_{s}]\!] &= [\![[\![f_{s-2},f_{s}]\!],f_{s-1}]\!]
 \end{cases}
 {\qquad \text{if } \alpha_{s} \in \Phi_{\bbar{1}}.}\label{eq:q-osp-serre-D}
\end{align}

\end{itemize}
The Hopf superalgebra structure on $U_q(\Gamma)$ is naturally defined by the
formulas~\eqref{eq:gl-comult}--\eqref{eq:gl-counit}.

\begin{Definition}\label{def:quantum-osp-standard}\quad
\begin{itemize}\itemsep=0pt
\item[(a)] The quantum supergroup $U_q(\Gamma_{\fgosp})$ is called the \emph{extended Drinfeld--Jimbo-type orthosymplectic
quantum supergroup} and is denoted by $\uqVe$. Using the basis \smash{$\big\{\Tilde{\ssH}_{i}\big\}_{i=1}^{s+1}$} of~$\Gamma_{\fgosp}$,
see~\eqref{eq:Tilde H}, it admits an equivalent presentation with finitely many generators. Specifically, it is the
$\BC(q)$-superalgebra generated by
\[
\{e_i, f_i\}_{i=1}^{s} \cup \big\{q^{\pm \Tilde{H}_i}\big\}_{i=1}^{s+1},
\]
 equipped with the same
$Q$-grading as in~\eqref{eq:uqGamma-osp-Q-grading} and compatible $\BZ_2$-grading. These generators are subject to the Serre
relations specified above, alongside the Chevalley-type
relations~\eqref{eq:q-gamma-chevalley-rel-HH}--\eqref{eq:q-gamma-chevalley-rel-ef} evaluated on the chosen
basis of $\Gamma_{\fgosp}$. Here, the elements \smash{$q^{\pm C} = q^{\pm \Tilde{H}_{s+1}}$} for $B$-type
(resp.\ \smash{$q^{\pm 2C} = q^{\pm \Tilde{H}_{s}}q^{\pm \Tilde{H}_{s+1}}$} for $CD$-types) are central, and
the elements $q^{\pm h_i}$ in~\eqref{eq:q-gamma-chevalley-rel-ef} are described through the chosen Cartan generators
\smash{$\bigl\{q^{\pm \Tilde{H}_k}\bigr\}_{k=1}^{s+1}$} akin to~\eqref{eq:h-elements} via
\begin{gather*}
 q^{\pm h_i} =
 \begin{cases}
 q^{\pm \Tilde{H}_i} q^{\mp \Tilde{H}_{i+1}} & \text{if } 1 \leq i < s \text{ \ (in all $BCD$-types)}, \\
 q^{\pm \Tilde{H}_s} q^{\mp \Tilde{H}_{s+1}} & \text{if } i = s \text{ \ ($BC$-types)}, \\
 q^{\pm \Tilde{H}_{s-1}} q^{\mp \Tilde{H}_{s+1}} & \text{if } i = s \text{ \ ($D$-type)}.
 \end{cases}
\end{gather*}

\item[(b)] Analogously, the \emph{Drinfeld--Jimbo-type orthosymplectic quantum supergroup} is defined as $\uqV \coloneqq U_q\bigl(Q^{\vee}\bigr)$.
Equivalently, it is the $\BC(q)$-superalgebra generated by $\{e_i, f_i\}_{i=1}^{s}$ and the Cartan generators
$\{q^{\pm h_i}\}_{i=1}^{s}$, subject to the same $Q$-grading and $\BZ_2$-grading. These generators satisfy the Serre relations
specified above, alongside the following explicit Chevalley-type relations (for all $1 \leq i,j \leq s$):
\begin{gather*}
 \bigl[q^{h_i}, q^{h_j}\bigr] = 0, \qquad
 q^{\pm h_i} q^{\mp h_i} = 1, \qquad
 q^{h_i} e_j q^{-h_i} = q^{(\alpha_i, \alpha_j)} e_j, \\
 q^{h_i} f_j q^{-h_i} = q^{-(\alpha_i, \alpha_j)} f_j, \qquad
 [e_i, f_j] = \delta_{ij} \frac{q^{h_i} - q^{-h_i}}{q - q^{-1}}.\end{gather*}
By definition, there is a canonical Hopf superalgebra morphism from $\uqV$ to\linebreak $\uqVe$. This morphism can be shown
to be an embedding via the triangular decomposition~\eqref{eq:osp-DJ-tri-decomp} discussed below.
\end{itemize}
\end{Definition}

\begin{Remark}
Our choice of the denominator $q-q^{-1}$ in the Chevalley-type relations follows the conventions of~\cite{y},
and may therefore differ from some other literature through a~scalar rescaling of the $f_j$'s.
\end{Remark}

\begin{Remark}\label{rem:osp-DrJ involution}
Analogously to Remark~\ref{rem:DrJ involution}, the assignment \smash{$\omega_{\DrJ}\bigl(q^{\pm \Tilde{H}_k}\bigr) = q^{\mp \Tilde{H}_k}$}
$(1 \leq k \leq s+1)$ and
\begin{gather*}
 \omega_{\DrJ}(e_i) =
 \begin{cases}
 \frac{1}{q^{5}(1+q^2)}f_s & \text{for $C$-type and } i=s, \\
 -(-1)^{\overline{i}\, \overline{i+1}} q^{3(\rho, \alpha_{i})} f_i
 & \text{otherwise},
 \end{cases}\\
 \omega_{\DrJ}(f_i) =
 \begin{cases}
 q^{5}\bigl(1+q^2\bigr) e_s & \text{for $C$-type and } i=s, \\
 -(-1)^{|f_{i}|}(-1)^{\overline{i}\, \overline{i+1}} q^{-3(\rho, \alpha_{i})} e_i & \text{otherwise},
 \end{cases}
\end{gather*}
gives rise to a~Hopf superalgebra isomorphism $\omega_{\DrJ} \colon \uqVe \iso \uqVe^{\copp}$,
where $\rho$ is the \emph{Weyl vector} of $\Phi$ given by the formula~\eqref{eq:weyl-vector} and satisfying property~\eqref{eq:weyl-vector-prop} for simple roots~$\alpha_i$ of~\eqref{eq:Lie-action-case-B}--\eqref{eq:Lie-action-case-D}.
\end{Remark}

\subsection[Drinfeld double and Drinfeld twist of U\_q(gosp(V))]{Drinfeld double and Drinfeld twist of $\boldsymbol{\uqVe}$}\label{ssec:osp-double-twist}

In complete analogy with the $\gl(V)$ case (cf.~\eqref{eq:DJ-tri-decomp}), $\uqVe$ admits a~\emph{triangular decomposition}
\begin{gather}\label{eq:osp-DJ-tri-decomp}
 U^<_{q}(\fgosp(V)) \otimes U^0_{q}(\fgosp(V)) \otimes U^>_{q}(\fgosp(V)) \iso \uqVe,
\end{gather}
where the factors are the subalgebras generated by \smash{$\{f_i\}_{i=1}^{s}$}, \smash{$\bigl\{q^{\pm H} \mid H \in \Gamma_{\fgosp}\bigr\}$},
and $\{e_i\}_{i=1}^{s}$, respectively.
This triangular decomposition yields the following isomorphisms of underlying superspaces for the positive and negative Borel Hopf subalgebras $U^{\geq}_{q}(\fgosp(V))$ and $U^{\leq}_{q}(\fgosp(V))$:
\begin{gather*}
 U^>_{q}(\fgosp(V)) \otimes U^0_{q}(\fgosp(V)) \iso U^{\geq}_{q}(\fgosp(V)),\\
 U^<_{q}(\fgosp(V)) \otimes U^0_{q}(\fgosp(V)) \iso U^{\leq}_{q}(\fgosp(V)).
\end{gather*}

The following result establishes the Drinfeld double realization of $\uqVe$. This is a~direct analogue of
Proposition~\ref{prop:DJ-pairing_finite} (we omit the proof which is identical).

\begin{Proposition}\label{prop:osp-DJ-pairing_finite}\quad
\begin{itemize}\itemsep=0pt
\item[$(a)$] There exists a~unique skew-pairing $($cf.~\eqref{eq:skew pairing structural property}$)$
\begin{gather}\label{eq:osp-skew-pairing}
 (\cdot,\cdot)_{\DrJ}\colon\ U^\leq_{q}(\fgosp(V)) \times U^\geq_{q}(\fgosp(V)) \to \BC(q)
\end{gather}
satisfying the following properties on the generators for $1 \leq i,j \leq s$ and $H,H' \in \Gamma_{\fgosp}$:
\begin{gather*}
 \bigl(f_i, q^{H}\bigr)_{\DrJ} = 0, \qquad \bigl(q^{H}, e_i\bigr)_{\DrJ} = 0,\qquad
 (f_i, e_j)_{\DrJ} = \frac{\delta_{ij} (-1)^{|f_i||e_j|}}{q^{-1} - q},\\
 \bigl(q^{H}, q^{H'}\bigr)_{\DrJ} = q^{-(H,H')},
\end{gather*}
where the pairing in the last equality is evaluated via~\eqref{eq:osp-str-form}. Moreover, this pairing is non-degenerate,
as is its restriction to the Cartan subalgebras $U^0_{q}(\fgosp(V)) \times U^0_{q}(\fgosp(V))$.

\item[$(b)$] The superalgebra $\uqVe$ is isomorphic to the quotient of the Drinfeld double $\CD_{\DrJ}$, formed by the pair of Hopf subalgebras $U^\leq_q(\fgosp(V))$ and $U^\geq_q(\fgosp(V))$ with respect to \eqref{eq:osp-skew-pairing}, obtained by identifying the Cartan subalgebras $U^0_q(\fgosp(V))$ from both factors.
\end{itemize}
\end{Proposition}

\begin{Remark}\label{rem:osp-DrJ pairing order}
Similarly to Remark~\ref{rem:DrJ pairing order}, one may alternatively construct the Drinfeld double by taking the Borel
subalgebras in the opposite order. This utilizes the transposed inverse~${(-,-)_{\widetilde{\DrJ}} \colon U^\geq_{q}(\fgosp(V)) \times U^\leq_{q}(\fgosp(V)) \to \BC(q)}$, which is also non-degenerate
and evaluates on the generators as follows:
\begin{gather*}
 \bigl(e_i, q^{H}\bigr)_{\widetilde{\DrJ}} = 0, \qquad \bigl(q^{H}, f_i\bigr)_{\widetilde{\DrJ}} = 0,\qquad
 (e_i, f_j)_{\widetilde{\DrJ}} = \frac{\delta_{ij}}{q - q^{-1}},\qquad
 \bigl(q^{H}, q^{H'}\bigr)_{\widetilde{\DrJ}} = q^{(H, H')}.
\end{gather*}
\end{Remark}

We now introduce a~specific Drinfeld twist of $\uqVe$ using the element $\CF = \fR_{s} \coloneqq \smash{q^{- C \otimes C} q^{-\sum_{i=1}^{s} (-1)^{\overline{i}} H_i \otimes H_i}}$, the semisimple part of the universal
$R$-matrix, see Section~\ref{ssec:osp-universal-R}. This yields the exact same formulas for the twisted comultiplication
$\Delta^{\CF}$ and antipode $S^{\CF}$ on the generators $e_i$, $f_i$, $q^{H}$ as in the $\gl(V)$ case from
Section~\ref{ssec:drinfeld-twist}:
\begin{align*}
 & \Delta^{\CF}(e_i) = 1 \otimes e_i + e_i \otimes q^{-h_i},\qquad
 \Delta^{\CF}(f_i) = q^{h_i} \otimes f_i + f_i \otimes 1,\qquad
 \Delta^{\CF}\bigl(q^{H}\bigr) = q^{H} \otimes q^{H}, \\
 & S^{\CF}(e_i) = -e_i q^{h_i},\qquad
 S^{\CF}(f_i) = -q^{-h_i} f_i,\qquad
 S^{\CF}\bigl(q^{H}\bigr) = q^{-H}.
\end{align*}
The following structural properties are completely parallel to Proposition~\ref{prop:drinfeld-twist-iso}.

\begin{Proposition}\label{prop:osp-drinfeld-twist-iso}\quad
\begin{itemize}\itemsep=0pt
\item[$(a)$] The assignment $\phi_{\CF}(e_i) = e_i q^{-h_i}$, $\phi_{\CF}(f_i) = q^{h_i} f_i$, $\phi_{\CF}\bigl(q^{H}\bigr) = q^{H}$
uniquely extends to a~$Q$-graded Hopf superalgebra isomorphism
\[\phi_{\CF} \colon\ \uqVe^{\CF} \iso \uqVe.\]

\item[$(b)$] The isomorphism $\phi_{\CF}$ preserves the Borel subalgebras $U^{\geq}_q(\fgosp(V))$ and $U^{\leq}_q(\fgosp(V))$.

\item[$(c)$] The pullback skew-pairing $(x,y)^{\CF}_{\DrJ} \coloneqq (\phi_{\CF}(x), \phi_{\CF}(y))_{\DrJ}$ evaluates identically to
the skew-pairing $(-,-)_{\DrJ}$ on the generators. Consequently, $\uqVe^{\CF}$ can also be realized as the Drinfeld double
of $U^{\leq}_q(\fgosp(V))^{\CF}$ and $U^{\geq}_q(\fgosp(V))^{\CF}$ with respect to $(-,-)^{\CF}_{\DrJ}$ {\rm(}and similarly for
\smash{$(-,-)^{\CF}_{\widetilde{\DrJ}}$} with the opposite order, see Remark~{\rm\ref{rem:osp-DrJ pairing order}}$)$.
\end{itemize}
\end{Proposition}

\section[RLL-realization of extended orthosymplectic quantum supergroups]{$\boldsymbol{\RLL}$-realization of extended orthosymplectic\\ quantum supergroups}\label{sec:RTT_osp}

\subsection[The universal R-matrix]{The universal $\boldsymbol{R}$-matrix}\label{ssec:osp-universal-R}

Following the combinatorial approach developed in~\cite{ht1} (based on~\cite{chw}), one defines the corresponding
\emph{quantum root vectors} $\{e_\gamma,f_\gamma\}_{\gamma \in \bar{\Phi}^{+}}$ recursively as in~\cite[equation~(5.20)]{ht1},
using iterated $q$-supercommutators~\eqref{eq:q-gl-superbracket} and combinatorics of dominant Lyndon words. As in the $\gl(V)$
case, let $\gamma_{ij} \coloneqq \varepsilon_{i} - \varepsilon_{j}$. Since $\gamma_{ij} = \gamma_{j'i'}$, the reduced positive
root system $\bar{\Phi}^{+}$ of~\eqref{eq:osp-reduced-root-sys} is uniquely parameterized by a~subset of
$\{(i,j) \mid i < j \leq i'\}$ via
\begin{gather}\label{eq:gosp-reduced-roots}
 \bar{\Phi}^{+} = \begin{cases}
 \{ \gamma_{ij} \mid i < j < i' \} & \text{($B$-type)}, \\
 \{ \gamma_{ij} \mid i < j < i' \} \cup
 \{ \gamma_{ii'} \mid \overline{i} = \bbar{1} \} & \text{($CD$-types)}.
 \end{cases}
\end{gather}
The set $\sL^+$ of dominant Lyndon words depends on the order of simple roots, and we choose it as in~\cite{ht1}:
\begin{gather}\label{eq:alphabet-order}
 \alpha_1 \prec \alpha_2 \prec \cdots \prec \alpha_s.
\end{gather}
Using the aforementioned bijection $\rl \colon \sL^{+} \iso \bar{\Phi}^{+}$, the lexicographic order on $\sL^+$ induced
by~\eqref{eq:alphabet-order} gives rise to a \emph{lexicographic} order $\prec$ on $\bar{\Phi}^{+}$. Explicitly, $\prec$
is obtained from the lexicographic order on the pairs $\{(i,j) \mid i < j < i'\}$ (cf.~\eqref{eq:gl-convex-order}) by
imposing $\gamma_{is} \prec \gamma_{ii'} \prec \gamma_{is'}$ for all $1 \leq i < s$ satisfying $\overline{i} = \bbar{1}$.
For simple roots $\alpha_i$, we set $e_{\alpha_i} = e_i$
and $f_{\alpha_i} = f_i$. For non-simple roots $\gamma \in \bar{\Phi}^{+}$, the root vectors are defined inductively via
(cf.~\eqref{eq:quantum root vectors})
\[
 e_{\gamma} = [\![ e_{\alpha}, e_{\beta} ]\!], \qquad
 f_{\gamma} = - (-1)^{|f_\alpha| |f_\beta|} q^{-(\alpha,\beta)} [\![ f_{\alpha}, f_{\beta} ]\!],
 \]
where the pair of roots $(\alpha, \beta)$ corresponds to the so-called costandard factorization $\ell = \ell_1 \ell_2$
of the dominant Lyndon word $\ell = \rl^{-1}(\gamma)$ into $\ell_1 = \rl^{-1}(\alpha)$ and $\ell_2 = \rl^{-1}(\beta)$.
The costandard factorization for the order~\eqref{eq:alphabet-order} is generically given by the uniform rule
$\gamma_{ij} \mapsto (\gamma_{i, j-1}, \gamma_{j-1, j})$ for $i < j \leq i'$, which applies exhaustively for $B$-type,
while the type-dependent exceptions for $CD$-types are as follows:
\begin{itemize}\itemsep=0pt

\item \textbf{$C$-type:}
$\gamma_{ii'} \mapsto (\gamma_{is}, \gamma_{is'})$ for $i < s$;

\item \textbf{$D$-type:}
$\gamma_{is'} \mapsto (\gamma_{i,s-1}, \gamma_{s-1,s'})$ for $i < s-1$,
and $\gamma_{ii'} \mapsto (\gamma_{is}, \gamma_{is'})$ for $i < s$.

\end{itemize}

In complete parallel with the $\gl(V)$ case (cf.\ Theorem~\ref{thm:gl-PBW-general}), these quantum root vectors provide
PBW-type bases for the strictly positive and negative subalgebras $U^>_q(\fgosp(V))$ and $U^<_q(\fgosp(V))$, which are
orthogonal with respect to the skew-pairing $(-,-)_{\DrJ}$ from~\eqref{eq:osp-skew-pairing}.

\begin{Theorem}\quad
\begin{itemize}\itemsep=0pt
\item[$(a)$] The sets of ordered monomials
\begin{gather*}
 \Bigg\{ \prod_{\gamma \in \bar{\Phi}^{+}}^{\leftarrow} e_{\gamma}^{m_{\gamma}} \;\Bigg|\;
 \substack{m_{\gamma} \in \BZ_{\ge 0} \\ m_\gamma \leq 1 \text{ if } \gamma \in \bar{\Phi}_{\bbar{1}} \text{ isotropic}} \Bigg\}
 \qquad \text{and} \qquad \Bigg\{ \prod_{\gamma \in \bar{\Phi}^{+}}^{\leftarrow} f_{\gamma}^{m_{\gamma}} \;\Bigg|\;
 \substack{m_{\gamma} \in \BZ_{\ge 0} \\ m_\gamma \leq 1 \text{ if } \gamma \in \bar{\Phi}_{\bbar{1}} \text{ isotropic}} \Bigg\}
\end{gather*}
form bases for $U^>_q(\fgosp(V))$ and $U^<_q(\fgosp(V))$, respectively.
Here, the arrow $\leftarrow$ over the product signs refers to the lexicographic order $\prec$ on $\bar{\Phi}^+$, thus ordering the elements of $\bar{\Phi}^+$ in decreasing order.

\item[$(b)$] These bases are orthogonal with respect to the skew-pairing $(-,-)_{\DrJ}$ of~\eqref{eq:osp-skew-pairing}.
\end{itemize}
\end{Theorem}

Using these dual PBW bases, the universal $R$-matrix $\fR$ in the completion of the tensor square of $\uqVe$
(cf.\ Remark~\ref{rem:h-adic-completion-avoidance}) factors exactly as it does in the $\gl(V)$ case (cf.~\cite[Theorem~5.18]{ht1})
\begin{gather*}
 \fR = \fR_u \fR_s,\qquad
 \fR_s = q^{- C \otimes C} q^{-\sum_{i=1}^{s} (-1)^{\overline{i}} H_i \otimes H_i},\\
 \fR_u = \prod_{\gamma \in \bar{\Phi}^{+}}^{\leftarrow}
 \left( \sum^{k\geq 0}_{k \leq 1 \text{ if } \gamma \in \bar{\Phi}_{\bbar{1}} \text{ isotropic}}
 \frac{e_\gamma^k \otimes f_\gamma^k}{\bigl(f_\gamma^k, e_\gamma^k\bigr)_{\DrJ}} \right).
\end{gather*}
We now evaluate $\fR$ on the \emph{first fundamental representation} $\varrho \colon\! U_{\!q}(\fgosp(V)\!) \to \End(V_{q})$
(cf.~\cite[Proposition 3.1]{ht1}), where $V_{q} \coloneqq \BC(q) \otimes_{\BC} V$ as before. The representation $\varrho$
is defined on the generators by $\varrho(e_i) = \sse_i$ and $\varrho(f_i) = \kappa_i\ssf_i$ for $1 \leq i \leq s$, where
\begin{gather}\label{eq:kappa-const}
 \kappa_{i} =
 \begin{cases}
 \frac{q+q^{-1}}{2} & \text{for $C$-type and } i=s, \\
 1 & \text{otherwise},
 \end{cases}
\end{gather}
while on the Cartan generators it is given by \smash{$\varrho\bigl(q^{\pm \Tilde{H}_i}\bigr) = q^{\pm (\ssH_i + \ID)}$} for $1 \leq i \leq s+1$,
where $\ID$ denotes the identity matrix and the exponentials of diagonal matrices are evaluated as in~\eqref{eq:q-D}.

The evaluation of the semisimple part $\varrho^{\otimes 2}(\fR_{s})$ follows from a~direct $\hbar$-adic
computation (cf.~\eqref{eq:osp-epsilon-def} and~\eqref{eq:osp-epsilon-pairing}), while that of the unipotent part
$\varrho^{\otimes 2}(\fR_{u})$ is derived in~\cite[Propositions~5.23 and 5.28]{ht1},
\begin{gather*}
 R_{s}\coloneqq \varrho^{\otimes 2}(\fR_s) =
 q^{-(\Tilde{\varepsilon}_{C},\Tilde{\varepsilon}_{C})} \sum_{1\leq i,j\leq N}
 q^{-(\varepsilon_{i},\varepsilon_{j})} E_{ii} \otimes E_{jj}
 = \sum_{1\leq i,j\leq N} q^{-(\Tilde{\varepsilon}_{i},\Tilde{\varepsilon}_{j})} E_{ii} \otimes E_{jj},\\
 R_{u} \coloneqq \varrho^{\otimes 2}(\fR_u) = \ID -
 \bigl(q - q^{-1}\bigr)\!\sum_{i < j}\!(-1)^{\overline{j}} E_{ij} \\
 \phantom{R_{u} \coloneqq \varrho^{\otimes 2}(\fR_u) =}{}\otimes \bigl( q^{(\varepsilon_{i},\varepsilon_{j})} E_{ji} -
 (-1)^{\overline{j}(\overline{i}+\overline{j})} \vartheta_{i}\vartheta_{j} q^{(\rho, \varepsilon_{i} - \varepsilon_{j})}
 q^{-(\varepsilon_{i},\varepsilon_{i})/2} q^{-(\varepsilon_{j},\varepsilon_{j})/2} E_{i'j'} \bigr),
\end{gather*}
where $\rho$ is the Weyl vector.
Combining these we get explicit formulas for $R \coloneqq \varrho^{\otimes 2}(\fR)$ and its
inverse $R^{-1}$ in $\End(V_{q})^{\otimes 2}$, cf.~\cite[Remark~4.7]{ht1}
\begin{gather}
 R = q^{-(\Tilde{\varepsilon}_{C},\Tilde{\varepsilon}_{C})} \cdot \biggl\{ \sum_{1\leq i,j\leq N}
 q^{-(\varepsilon_{i},\varepsilon_{j})} E_{ii} \otimes E_{jj} - \bigl(q - q^{-1}\bigr)
 \sum_{i < j} (-1)^{\overline{j}} E_{ij}\nonumber \\
 \phantom{R = q^{-(\Tilde{\varepsilon}_{C},\Tilde{\varepsilon}_{C})} \cdot \biggl\{}{} \otimes \bigl( E_{ji} -
 (-1)^{\overline{j}(\overline{i}+\overline{j})} \vartheta_{i}\vartheta_{j} q^{(\rho, \varepsilon_{i} - \varepsilon_{j})}
 q^{-(\varepsilon_{i},\varepsilon_{i})/2} q^{(\varepsilon_{j},\varepsilon_{j})/2} E_{i'j'} \bigr) \biggr\},\nonumber\\
 R^{-1}
 = q^{(\Tilde{\varepsilon}_{C},\Tilde{\varepsilon}_{C})} \cdot \biggl\{ \sum_{1\leq i,j\leq N}
 q^{(\varepsilon_{i},\varepsilon_{j})} E_{ii} \otimes E_{jj} + \bigl(q - q^{-1}\bigr) \sum_{i < j} (-1)^{\overline{j}} E_{ij} \nonumber\\
 \phantom{R^{-1}
 = q^{(\Tilde{\varepsilon}_{C},\Tilde{\varepsilon}_{C})} \cdot \biggl\{}{} \otimes \bigl( E_{ji} -
 (-1)^{\overline{j}(\overline{i}+\overline{j})} \vartheta_{i}\vartheta_{j} q^{-(\rho, \varepsilon_{i} - \varepsilon_{j})}
 q^{-(\varepsilon_{i},\varepsilon_{i})/2} q^{(\varepsilon_{j},\varepsilon_{j})/2} E_{i'j'} \bigr) \biggr\}.\label{eq:osp-evaluated-R}
\end{gather}

\begin{Remark}\label{rem:osp-YBE}
We note that $\fR$ (as well as its evaluation $R$) satisfies the exact same Yang--Baxter equation, comultiplication properties,
and braid relations as those detailed in Remark~\ref{rem:YBE} for the $\gl(V)$ case.
\end{Remark}

\subsection[The superalgebra U(R)]{The superalgebra $\boldsymbol{U(R)}$}

Following the general linear case, we define the $\RLL$-realization $U(R)$ of the extended orthosymplectic quantum supergroup
using the corresponding $R$-matrix~\eqref{eq:osp-evaluated-R}. The overarching algebraic structure is defined in exactly the
same way as in the $\gl(V)$ setting, with the only difference being the specific choice of $R$.

\begin{Definition}
The $\BC(q)$-superalgebra $U(R)$ is generated by the elements \smash{$\{l^{+}_{ij}, l^{-}_{ji}\}_{1 \leq i \leq j \leq N}$}.
Adopting the convention $l^{+}_{ij} = 0 = l^{-}_{ji}$ for $i>j$, we organize these generators into formal matrices
\smash{$L^{\pm} \coloneqq \sum_{i,j=1}^{N} l^{\pm}_{ij} \otimes E_{ij}$}. The superalgebra $U(R)$ is defined as the quotient of the
free superalgebra $\CT \coloneqq \BC(q)\langle l^{+}_{ij}, l^{-}_{ji} | 1 \leq i \leq j \leq N \rangle$ by the
relations~\eqref{eq:gl-rtt-diag-rel}, together with all matrix coefficients in $\CT$ of the matrix equations~\eqref{eq:gl-rtt-rel-1}
and~\eqref{eq:gl-rtt-rel-2} evaluated in \smash{$\CT \otimes \End(V_{q})^{\otimes 2}$}. The $\BZ_{2}$-grading is
compatible via~\eqref{eq:gosp-grading-compatible} with the $Q$-grading determined by
\begin{gather}\label{eq:osp-uR-Q-grading}
 \deg\bigl(l^{\pm}_{ij}\bigr) = \varepsilon_{i} - \varepsilon_{j}.
\end{gather}
\end{Definition}

As the defining relations are $Q$-homogeneous, we note that $U(R)$ is indeed a~$Q$-graded superalgebra. Furthermore, the exact
same formulas~\eqref{eq:gl-rtt-coalg}, \eqref{eq:gl-rtt-antipode} endow $U(R)$ with the structure of a~Hopf superalgebra.

\begin{Remark}\label{rem:osp-U(R) involution}
As follows from the explicit formula~\eqref{eq:osp-evaluated-R}, we have
$(R_{12})^{\st} = \Theta_{1}\Theta_{2} \cdot R_{21} \cdot (\Theta_{1}\Theta_{2})^{-1}$ where $\Theta = \sum_{i=1}^{N} q^{(\varepsilon_{i},\varepsilon_{i})/2} E_{ii}$,
similarly to~\eqref{eq:gl-st-vs-op}. Thus, arguing exactly as in Remark~\ref{rem:U(R) involution}, we note that
applying the full supertranspose to the defining relations of $U(R)$ yields an isomorphism of Hopf superalgebras
\begin{gather}\label{eq:osp-transpose-iso}
 \omega_{R}\colon\ U(R)\iso U(R)^{\copp} \qquad \text{given by} \ L^{\pm} \mapsto \Theta^{-1}\bigl(L^{\mp}\bigr)^{\st}\Theta.
\end{gather}
This $\omega_{R}$ induces a~superalgebra isomorphism between $U^{+}(R)$ and $U^{-}(R)$ introduced below.
\end{Remark}

Let $U^{+}(R)$ and $U^{-}(R)$ be the superalgebras generated respectively by \smash{$\bigl\{l^{+}_{ij}\bigr\}_{1 \leq i \leq j \leq N}$} and
$\{l^{-}_{ji}\}_{1 \leq i \leq j \leq N}$, equipped with the $Q$-grading~\eqref{eq:osp-uR-Q-grading} and the compatible
$\BZ_2$-grading, subject to the respective single-sign relations~\eqref{eq:gl-rtt-rel-1}. The same
formulas~\eqref{eq:gl-rtt-coalg} endow both $U^{\pm}(R)$ with superbialgebra structures. As in the general linear setting,
the superalgebra $U(R)$ can be realized as a~quotient of the generalized double of the superbialgebra pair $\bigl(U^{+}(R), U^{-}(R)\bigr)$.

\begin{Proposition}\label{prop:osp-RTT-pairing_finite}\quad
\begin{itemize}\itemsep=0pt
\item[$(a)$] There exists a~unique skew-pairing
\begin{gather}\label{eq:osp-RTT-skew-pairing}
 (\cdot,\cdot)_{R} \colon\ U^{+}(R) \times U^{-}(R) \to \BC(q)
\end{gather}
satisfying the following property on the generators:
\begin{gather}\label{eq:osp-RTT-skew-pairing-gen}
 \sum_{1\leq i,j,k,l\leq N} (-1)^{(\overline{i}+\overline{j})(\overline{k}+\overline{l})} \bigl(l^{+}_{ij}, l^{-}_{kl}\bigr)_{R} E_{ij} \otimes E_{kl}
 = R.
\end{gather}

\item[$(b)$] The superalgebra $U(R)$ is isomorphic to the quotient of the generalized double
$\CD_{R} \coloneqq \CD\bigl(U^{+}(R),U^{-}(R)\bigr)$ modulo the relations $l^{\pm}_{ii}l^{\mp}_{ii} = 1$ for all $1 \leq i \leq N$.
\end{itemize}
\end{Proposition}

We omit the proof of Proposition~\ref{prop:osp-RTT-pairing_finite}, as it relies entirely on the structural compatibilities
governed by the Yang--Baxter equation (cf.\ Remark~\ref{rem:osp-YBE}) and is thus completely analogous to that of
Proposition~\ref{prop:RTT-pairing_finite}. We note that
\begin{gather*}
 \bigl(l^{+}_{ij}, l^{-}_{kl}\bigr)_{R} \ne 0 \iff
 (i,k)=(j,l) \qquad \text{or} \qquad (i,k)=(l,j) \qquad \text{or} \qquad k=i' \ \& \ l=j',
\end{gather*}
which in view of~\eqref{eq:osp-uR-Q-grading} implies that the skew-pairing~\eqref{eq:osp-RTT-skew-pairing} is of $Q$-degree zero,
and consequently of $\BZ_2$-degree zero. The latter is crucial for $(\star)$ used in the proof of
Proposition~\ref{prop:RTT-pairing_finite}.

\begin{Remark}
As in Remark~\ref{rem:gl-rtt-convention}, the evaluation~\eqref{eq:osp-RTT-skew-pairing-gen} is compatible with
the convention $l^{+}_{ij} = 0 = l^{-}_{ji}$ for $i>j$, as evident from the explicit formula~\eqref{eq:osp-evaluated-R}.
\end{Remark}

\begin{Remark}\label{rem:osp-RTT double full Cartan}
Following Remark~\ref{rem:RTT double full Cartan}, let $U^{\geq}(R)$ and $U^{\leq}(R)$ denote the Hopf subalgebras of~$U(R)$
generated by \smash{$\bigl\{l^{+}_{ij}\bigr\}_{1 \leq i \leq j \leq N} \cup \{l^{-}_{kk}\}_{1 \leq k \leq N}$} and
\smash{$\{l^{-}_{ji}\}_{1 \leq i \leq j \leq N} \cup \{l^{+}_{kk}\}_{1 \leq k \leq N}$}, respectively.
The skew-pairing~\eqref{eq:osp-RTT-skew-pairing} extends uniquely to a~skew-pairing $U^{\geq}(R) \times U^{\leq}(R) \to \BC(q)$.
Thus, the generalized double construction of $U(R)$ can equivalently be carried out using these Hopf subalgebras. Furthermore,
the isomorphism $\omega_{R}$ from~\eqref{eq:osp-transpose-iso} restricts to a~superalgebra isomorphism between
$U^{\geq}(R)$ and $U^{\leq}(R)$.
\end{Remark}

\begin{Remark}
As in Remark~\ref{rem:RTT pairing order}, the generalized double construction for $U(R)$ can equivalently be carried out by taking
the subbialgebras in the opposite order and utilizing the transposed inverse of $(-,-)_{R}$:
\[
 \widetilde{\sigma}=(-,-)_{\widetilde{R}} \colon \ U^{-}(R) \times U^{+}(R) \to \BC(q).
 \]
By direct computation, its evaluation on the generators is given by the matrix
formula~\eqref{eq:gl-RTT-skew-pairing-opposite}. As in Remark~\ref{rem:osp-RTT double full Cartan}, this pairing uniquely extends
to a~skew-pairing $(-,-)_{\widetilde{R}} \colon U^{\leq}(R) \times U^{\geq}(R) \to \BC(q)$.
\end{Remark}

Finally, the twisting procedure used to bypass the emergence of Koszul signs during formal matrix multiplication carries
over verbatim, yielding the analogue of Proposition~\ref{prop:gl-twisted-rtt}.

\begin{Proposition}\label{prop:osp-twisted-rtt}
Let $\zeta \colon \BZ_{2} \times \BZ_{2} \to \BC(q)^{\times}$ be the normalized $2$-cocycle given by
\smash{$\zeta(\bbar{i},\bbar{j}) \coloneqq (-1)^{\bbar{i} \bbar{j}}$}. Then, the assignment $l^{\pm}_{ij} \mapsto l^{\pm}_{ij}$
for all $1 \leq i \leq j \leq N$ gives rise to a~superalgebra isomorphism~\smash{$\URtw \iso U(R)^{\zeta}$}.
\end{Proposition}

\subsection{Gauss decomposition and twisted identities: orthosymplectic case}\label{ssec:osp-Gauss}

As in Section~\ref{ssec:gl-Gauss}, we shall assume that all algebraic manipulations occur within the twisted superalgebra
$\URtw$ or the tensor product $\URtw^{\otimes r} \otimes^{\beta'} \End(V_{q})^{\otimes s}$ (for $r,s \geq 1$), rather than
in their untwisted counterparts. Again, the elements in \smash{$\URtw^{\otimes r} \otimes^{\beta'} \End(V_{q})^{\otimes s}$} will be
considered as formal matrices~\eqref{eq:gl-formal-matrix}, following the standard matrix multiplication~\eqref{eq:gl-matrix-mult}.
Consequently, $L^{+}$ and~$L^{-}$ are genuinely upper and lower triangular matrices, respectively, and
using~\eqref{eq:gl-rtt-diag-rel} we obtain the exact same \emph{Gauss decompositions}~\eqref{eq:gl-Gauss-L}.

Similarly to the general linear setup, we define the elements \smash{$\{\bfe_{ij}, \bff_{ji}\}_{1 \leq i < j \leq N}$} in $\URtw$ via
the exact same formulas~\eqref{eq:RTT ef gen}. We also introduce \smash{$q^{\pm \Tilde{\bfH}_k} \coloneqq l_{kk}^{\pm}$} in $\URtw$ for
$1 \leq k \leq N$ and define~\smash{$q^{\pm \Tilde{\bfH}_{ij}} \coloneqq q^{\pm \Tilde{\bfH}_i} \bigl(q^{\pm \Tilde{\bfH}_j}\bigr)^{-1}$} for all
$1 \leq i < j \leq N$.

By evaluating the specific matrix components of the defining relations~\eqref{eq:gl-rtt-rel-1} and~\eqref{eq:gl-rtt-rel-2},
we deduce a~set of identities governing the elements $\bfe_{ij}$, $\bff_{ji}$, and \smash{$q^{\pm \Tilde{\bfH}_{k}}$} within \smash{$\URtw$}.
As before, $[-,-]^{\zeta}$ and $[\![-,-]\!]^{\zeta}$ denote the supercommutator and the $q$-supercommutator
(cf.~\eqref{eq:super commutator} and~\eqref{eq:q-gl-superbracket}) of elements in $\URtw$.

We start with type-independent identities, the first three of which are derived exactly as in Section~\ref{ssec:gl-Gauss}:

\begin{table}[H]
\centering
\small
\renewcommand{\arraystretch}{1.2}
\setlength{\tabcolsep}{5pt}
{\thickmuskip=2mu\medmuskip=1.5mu
\begin{tabularx}{\textwidth}{|l l >{\raggedright\arraybackslash}X r|}
 \hline
 Matrix Component & Conditions & Identity & \\
 \hline
 $E_{ii} \otimes E_{jj}$
 & $1 \leq i,j \leq N$
 & $\bigl[q^{\Tilde{\bfH}_{i}},q^{\Tilde{\bfH}_{j}}\bigr]^{\zeta} = 0$
 & \AddLabel{eq:indep-twisted-1} \\
 \quad
 & $1 \leq i \leq N$
 & $q^{\pm \Tilde{\bfH}_{i}}q^{\mp \Tilde{\bfH}_{i}} = 1$
 & \AddLabel{eq:indep-twisted-2} \\
 $E_{aa} \otimes E_{ij}$
 & $1 \leq i < j \leq N$, $1 \leq a \leq N$
 & $q^{\Tilde{\bfH}_{a}} \bfe_{ij} q^{-\Tilde{\bfH}_{a}} =
 q^{(\varepsilon_{a}, \varepsilon_{i}-\varepsilon_{j})} \bfe_{ij}$
 & \AddLabel{eq:indep-twisted-3} \\
 $E_{ij} \otimes E_{ij}$
 & $1 \leq i < j \leq N$, $\varepsilon_{i} - \varepsilon_{j}$ odd isotropic
 & $[\![ \bfe_{ij}, \bfe_{ij} ]\!]^{\zeta} = 0$
 & \AddLabel{eq:indep-twisted-4} \\
 \hline
\end{tabularx}}\vspace{-3mm}
\caption{Type-independent identities in $\URtw$.}\label{table:osp-indep}
\end{table}

We now proceed to list the type-dependent identities:

\begingroup
\small
\renewcommand{\arraystretch}{1.2}
\setlength{\tabcolsep}{5pt}
\thickmuskip=1mu
\begin{xltabular}{\textwidth}{|l l >{\raggedright\arraybackslash}X r|}
 \hline
 Matrix Component & Conditions & Identity & \\
 \hline
 \endfirsthead

 \hline
 Matrix Component & Conditions & Identity & \\
 \hline
 \endhead

 \hline
 \endfoot

 \hline
 \multicolumn{4}{@{}c@{}}{%
 \rule{0pt}{18pt}
 \textbf{Table \thetable.} $B$-type, identities in $\URtw$ derived from $RL_{1}^{+}L_{2}^{+} = L_{2}^{+}L_{1}^{+}R$.%
 \label{table:B1}%
 } \\
 \endlastfoot

 $E_{ij} \otimes E_{i'j'}$
 & $1 \leq i < j \leq s+1$
 & $q^{\Tilde{\bfH}_{i}}q^{\Tilde{\bfH}_{i'}} = q^{\Tilde{\bfH}_{j}}q^{\Tilde{\bfH}_{j'}}$
 & \AddLabel{eq:B-twisted-1} \\
 $E_{i,i+1} \otimes E_{i'i'}$
 & $1 \leq i \leq s$
 & $\bfe_{i,i+1} = -(-1)^{\overline{i}(\overline{i}+\overline{i+1})} \vartheta_{i}\vartheta_{i+1} \bfe_{(i+1)'i'}$
 & \AddLabel{eq:B-twisted-2} \\
 $E_{ij} \otimes E_{j,j+1}$
 & $i < j < (i+1)'$
 & $\bfe_{i,j+1} = (-1)^{(\overline{i}+\overline{j})(\overline{j}+\overline{j+1})} [\![ \bfe_{ij}, \bfe_{j,j+1} ]\!]^{\zeta}$
 & \AddLabel{eq:B-twisted-3} \\
 $E_{i,(i+1)'} \otimes E_{(i+1)'i'}$
 & $1 \leq i \leq s$
 & \begin{tabular}{@{}l@{}}
 $\bfe_{ii'} = (-1)^{\overline{i}+\overline{i+1}}\bfe_{i,(i+1)'}\bfe_{(i+1)'i'}$\\
 \qquad $-q^{-(\varepsilon_{i},\varepsilon_{i})}q^{-(\varepsilon_{i+1},\varepsilon_{i+1})}
 \bfe_{(i+1)'i'}\bfe_{i,(i+1)'}$
 \end{tabular}
 & \AddLabel{eq:B-twisted-4} \\
 $E_{(j+1)'j'} \otimes E_{j'i'}$
 & $i < j < (i+1)'$
 & $\bfe_{(j+1)'i'} = (-1)^{(\overline{i}+\overline{j})(\overline{j}+\overline{j+1})} [\![ \bfe_{(j+1)'j'}, \bfe_{j'i'} ]\!]^{\zeta}$
 & \AddLabel{eq:B-twisted-5} \\
 $E_{i,i+1} \otimes E_{i+1,i'}$
 & $1 \leq i \leq s$
 & \begin{tabular}{@{}l@{}}
 $\bfe_{ii'} = (-1)^{\overline{i}+\overline{i+1}}\bfe_{i,i+1}\bfe_{i+1,i'}$\\
 \qquad $-q^{-(\varepsilon_{i},\varepsilon_{i})}q^{-(\varepsilon_{i+1},\varepsilon_{i+1})} \bfe_{i+1,i'}\bfe_{i,i+1}$
 \end{tabular}
 & \AddLabel{eq:B-twisted-6} \\
 $E_{i,i+1} \otimes E_{j,j+1}$
 & $1 \leq i < j \leq s$, $j-i > 1$
 & $[\![ \bfe_{i,i+1}, \bfe_{j,j+1} ]\!]^{\zeta} = 0$
 & \AddLabel{eq:B-twisted-7} \\
 $E_{i,i+1} \otimes E_{i,i+1}$
 & $1 \leq i < s$, $|\bfe_{i,i+1}| = \bbar{1}$
 & $[\![ \bfe_{i,i+1}, \bfe_{i,i+1} ]\!]^{\zeta} = 0$
 & \AddLabel{eq:B-twisted-8} \\
 $E_{i,i+1} \otimes E_{i,i+2}$
 & $1 \leq i < s$
 & $\bigl[\!\bigl[ \bfe_{i,i+1}, [\![ \bfe_{i,i+1}, \bfe_{i+1,i+2} ]\!]^{\zeta}\bigr]\!\bigr]^{\zeta} = 0$
 & \AddLabel{eq:B-twisted-9} \\
 $E_{i,i+1} \otimes E_{i-1,i+1}$
 & $1 < i < s$
 & $\bigl[\!\bigl[[\![ \bfe_{i-1,i}, \bfe_{i,i+1} ]\!]^{\zeta}, \bfe_{i,i+1} \bigr]\!\bigr]^{\zeta} = 0$
 & \AddLabel{eq:B-twisted-10} \\
 $E_{s,s+1} \otimes E_{s-1,s'}$
 &
 & $\bigl[\!\bigl[\bigl[\!\bigl[[\![ \bfe_{s-1,s}, \bfe_{s,s+1} ]\!]^{\zeta}, \bfe_{s,s+1} \bigr]\!\bigr]^{\zeta}, \bfe_{s,s+1} \bigr]\!\bigr]^{\zeta} = 0$
 & \AddLabel{eq:B-twisted-11} \\
 $E_{i,i+1} \otimes E_{i-1,i+2}$
 & $1 < i < s$
 & $\bigl[\!\bigl[\bigl[\!\bigl[[\![ \bfe_{i-1,i}, \bfe_{i,i+1} ]\!]^{\zeta}, \bfe_{i+1,i+2} \bigr]\!\bigr]^{\zeta}, \bfe_{i,i+1} \bigr]\!\bigr]^{\zeta} = 0$
 & \AddLabel{eq:B-twisted-12} \\
\end{xltabular}
\endgroup

\begingroup
\small
\renewcommand{\arraystretch}{1.2}
\setlength{\tabcolsep}{5pt}
\thickmuskip=1mu
\begin{xltabular}{\textwidth}{|l l >{\raggedright\arraybackslash}X r|}
 \hline
 Matrix Component & Conditions & Identity & \\
 \hline
 \endfirsthead

 \hline
 Matrix Component & Conditions & Identity & \\
 \hline
 \endhead

 \hline
 \endfoot

 \hline
 \multicolumn{4}{@{}c@{}}{%
 \rule{0pt}{18pt}%
 \textbf{Table \thetable.} $C$-type, identities in $\URtw$ derived from $RL_{1}^{+}L_{2}^{+} = L_{2}^{+}L_{1}^{+}R$.%
 \label{table:C1}%
 } \\
 \endlastfoot

 $E_{ij} \otimes E_{i'j'}$
 & $1 \leq i < j \leq s$
 & $q^{\Tilde{\bfH}_{i}}q^{\Tilde{\bfH}_{i'}} = q^{\Tilde{\bfH}_{j}}q^{\Tilde{\bfH}_{j'}}$
 & \AddLabel{eq:C-twisted-1} \\
 $E_{i,i+1} \otimes E_{i'i'}$
 & $1 \leq i \leq s-1$
 & $\bfe_{i,i+1} = -(-1)^{\overline{i}(\overline{i}+\overline{i+1})} \vartheta_{i}\vartheta_{i+1} \bfe_{(i+1)'i'}$
 & \AddLabel{eq:C-twisted-2} \\
 $E_{ij} \otimes E_{j,j+1}$
 & $i < j < (i+1)'$, $j \neq s$
 & $\bfe_{i,j+1} = (-1)^{(\overline{i}+\overline{j})(\overline{j}+\overline{j+1})} [\![ \bfe_{ij}, \bfe_{j,j+1} ]\!]^{\zeta}$
 & \AddLabel{eq:C-twisted-3} \\
 $E_{ss'} \otimes E_{is}$
 & $1 \leq i \leq s-1$
 & $\bigl(1+q^{2}\bigr)\bfe_{is'} = [\![ \bfe_{is}, \bfe_{ss'} ]\!]^{\zeta}$
 & \AddLabel{eq:C-twisted-4} \\
 $E_{i,(i+1)'} \otimes E_{(i+1)'i'}$
 & $1 \leq i \leq s-1$
 & \begin{tabular}{@{}l@{}}
 $\bfe_{ii'} = (-1)^{\overline{i}+\overline{i+1}}\bfe_{i,(i+1)'}\bfe_{(i+1)'i'}$\\
 \qquad $-q^{-(\varepsilon_{i},\varepsilon_{i})}q^{-(\varepsilon_{i+1},\varepsilon_{i+1})}
 \bfe_{(i+1)'i'}\bfe_{i,(i+1)'}$
 \end{tabular}
 & \AddLabel{eq:C-twisted-5} \\
 $E_{(j+1)'j'} \otimes E_{j'i'}$
 & $i < j < (i+1)'$, $j \neq s$
 & $\bfe_{(j+1)'i'} = (-1)^{(\overline{i}+\overline{j})(\overline{j}+\overline{j+1})} [\![ \bfe_{(j+1)'j'}, \bfe_{j'i'} ]\!]^{\zeta}$
 & \AddLabel{eq:C-twisted-6} \\
 $E_{ss'} \otimes E_{s'i'}$
 & $1 \leq i \leq s-1$
 & $\bigl(1+q^{2}\bigr)\bfe_{si'} = [\![ \bfe_{ss'}, \bfe_{s'i'} ]\!]^{\zeta}$
 & \AddLabel{eq:C-twisted-7} \\
 $E_{i,i+1} \otimes E_{i+1,i'}$
 & $1 \leq i \leq s-1$
 & \begin{tabular}{@{}l@{}}
 $\bfe_{ii'} = (-1)^{\overline{i}+\overline{i+1}}\bfe_{i,i+1}\bfe_{i+1,i'}$\\
 \qquad $-q^{-(\varepsilon_{i},\varepsilon_{i})}q^{-(\varepsilon_{i+1},\varepsilon_{i+1})} \bfe_{i+1,i'}\bfe_{i,i+1}$
 \end{tabular}
 & \AddLabel{eq:C-twisted-8} \\
 $E_{i,i+1} \otimes E_{j,j+1}$
 & $1 \leq i < j \leq s$, $j-i > 1$
 & $[\![ \bfe_{i,i+1}, \bfe_{j,j+1} ]\!]^{\zeta} = 0$
 & \AddLabel{eq:C-twisted-9} \\
 $E_{i,i+1} \otimes E_{i,i+1}$
 & $1 \leq i < s$, $|\bfe_{i,i+1}| = \bbar{1}$
 & $[\![ \bfe_{i,i+1}, \bfe_{i,i+1} ]\!]^{\zeta} = 0$
 & \AddLabel{eq:C-twisted-10} \\
 $E_{i,i+1} \otimes E_{i,i+2}$
 & $1 \leq i \leq s-2$
 & $\bigl[\!\bigl[ \bfe_{i,i+1}, [\![ \bfe_{i,i+1}, \bfe_{i+1,i+2} ]\!]^{\zeta}\bigr]\!\bigr]^{\zeta} = 0$
 & \AddLabel{eq:C-twisted-11} \\
 $E_{i,i+1} \otimes E_{i-1,i+1}$
 & $1 < i \leq s$
 & $\bigl[\!\bigl[[\![ \bfe_{i-1,i}, \bfe_{i,i+1} ]\!]^{\zeta}, \bfe_{i,i+1} \bigr]\!\bigr]^{\zeta} = 0$
 & \AddLabel{eq:C-twisted-12} \\
 $E_{i,i+1} \otimes E_{i-1,i+2}$
 & $1 < i < s-1$
 & $\bigl[\!\bigl[\bigl[\!\bigl[[\![ \bfe_{i-1,i}, \bfe_{i,i+1} ]\!]^{\zeta}, \bfe_{i+1,i+2} \bigr]\!\bigr]^{\zeta}, \bfe_{i,i+1} \bigr]\!\bigr]^{\zeta} = 0$
 & \AddLabel{eq:C-twisted-13} \\
 $E_{s-1,s} \otimes E_{s-3,(s-2)'}$
 &
 & \begin{tabular}{@{}l@{}}
 $[\![[\![[\![[\![[\![[\![ \bfe_{s-3,s-2}, \bfe_{s-2,s-1} ]\!]^{\zeta}, \bfe_{s-1,s} ]\!]^{\zeta}$,\\
\hspace{40pt}$\bfe_{ss'} ]\!]^{\zeta}, \bfe_{s-1,s} ]\!]^{\zeta}, \bfe_{s-2,s-1} ]\!]^{\zeta}, \bfe_{s-1,s} ]\!]^{\zeta} = 0$
 \end{tabular}
 & \AddLabel{eq:C-twisted-14}
\end{xltabular}
\endgroup

\begingroup
\small
\renewcommand{\arraystretch}{1.2}
\setlength{\tabcolsep}{3pt}
\thickmuskip=2mu
\begin{xltabular}{\textwidth}{|l l >{\raggedright\arraybackslash}X r|}
 \hline
 Matrix Component & Conditions & Identity & \\
 \hline
 \endfirsthead

 \hline
 Matrix Component & Conditions & Identity & \\
 \hline
 \endhead

 \hline
 \endfoot

 \hline
 \multicolumn{4}{@{}c@{}}{%
 \rule{0pt}{18pt}%
 \textbf{Table \thetable.} $D$-type, identities in $\URtw$ derived from $RL_{1}^{+}L_{2}^{+} = L_{2}^{+}L_{1}^{+}R$.%
 \label{table:D1}%
 } \\
 \endlastfoot

 $E_{ij} \otimes E_{i'j'}$
 & $1 \leq i < j \leq s$
 & $q^{\Tilde{\bfH}_{i}}q^{\Tilde{\bfH}_{i'}} = q^{\Tilde{\bfH}_{j}}q^{\Tilde{\bfH}_{j'}}$
 & \AddLabel{eq:D-twisted-1} \\
 $E_{i,i+1} \otimes E_{i'i'}$
 & $1 \leq i \leq s$
 & $\bfe_{i,i+1} = -(-1)^{\overline{i}(\overline{i}+\overline{i+1})} \vartheta_{i}\vartheta_{i+1} \bfe_{(i+1)'i'}$
 & \AddLabel{eq:D-twisted-2} \\
 $E_{s-1,s'} \otimes E_{(s-1)'(s-1)'}$
 &
 & $\bfe_{s-1,s'} = -(-1)^{\overline{s-1}(\overline{s-1}+\overline{s})} \vartheta_{s-1}\vartheta_{s'} \bfe_{s,(s-1)'}$
 & \AddLabel{eq:D-twisted-3} \\
 $E_{ij} \otimes E_{j,j+1}$
 & $i < j < (i+1)'$, $j \neq s$
 & $\bfe_{i,j+1} = (-1)^{(\overline{i}+\overline{j})(\overline{j}+\overline{j+1})} [\![ \bfe_{ij}, \bfe_{j,j+1} ]\!]^{\zeta}$
 & \AddLabel{eq:D-twisted-4} \\
 $E_{i,s-1} \otimes E_{s-1,s'} $
 & $1 \leq i \leq s-2$
 & $\bfe_{is'} = (-1)^{(\overline{i}+\overline{s-1})(\overline{s-1}+\overline{s})} [\![ \bfe_{i,s-1}, \bfe_{s-1,s'} ]\!]^{\zeta}$
 & \AddLabel{eq:D-twisted-5} \\
 $E_{is} \otimes E_{s,(s-1)'}$
 & $1 \leq i \leq s-2$
 & $\bfe_{i,(s-1)'} = (-1)^{(\overline{i}+\overline{s})(\overline{s-1}+\overline{s})} [\![ \bfe_{is}, \bfe_{s,(s-1)'} ]\!]^{\zeta}$
 & \AddLabel{eq:D-twisted-6} \\
 $E_{i,(i+1)'} \otimes E_{(i+1)'i'}$
 & $1 \leq i \leq s-1$
 & \begin{tabular}{@{}l@{}}
 $\bfe_{ii'} = (-1)^{\overline{i}+\overline{i+1}}\bfe_{i,(i+1)'}\bfe_{(i+1)'i'}$\\
 \qquad $-q^{-(\varepsilon_{i},\varepsilon_{i})}q^{-(\varepsilon_{i+1},\varepsilon_{i+1})} \bfe_{(i+1)'i'}\bfe_{i,(i+1)'}$
 \end{tabular}
 & \AddLabel{eq:D-twisted-7} \\
 $E_{(j+1)'j'} \otimes E_{j'i'}$
 & $i < j < (i+1)'$, $j \neq s$
 & $\bfe_{(j+1)'i'} = (-1)^{(\overline{i}+\overline{j})(\overline{j}+\overline{j+1})} [\![ \bfe_{(j+1)'j'}, \bfe_{j'i'} ]\!]^{\zeta}$
 & \AddLabel{eq:D-twisted-8} \\
 $E_{s,(s-1)'} \otimes E_{(s-1)'i'}$
 & $1 \leq i \leq s-2$
 & $\bfe_{si'} = (-1)^{(\overline{i}+\overline{s-1})(\overline{s-1}+\overline{s})} [\![ \bfe_{s,(s-1)'}, \bfe_{(s-1)'i'} ]\!]^{\zeta}$
 & \AddLabel{eq:D-twisted-9} \\
 $E_{s-1,s'} \otimes E_{s'i'}$
 & $1 \leq i \leq s-2$
 & $\bfe_{s-1,i'} = (-1)^{(\overline{i}+\overline{s})(\overline{s-1}+\overline{s})} [\![ \bfe_{s-1,s'}, \bfe_{s'i'} ]\!]^{\zeta}$
 & \AddLabel{eq:D-twisted-10} \\
 $E_{i,i+1} \otimes E_{i+1,i'}$
 & $1 \leq i \leq s-1$
 & \begin{tabular}{@{}l@{}}
 $\bfe_{ii'} = (-1)^{\overline{i}+\overline{i+1}}\bfe_{i,i+1}\bfe_{i+1,i'}$\\
 \qquad $-q^{-(\varepsilon_{i},\varepsilon_{i})}q^{-(\varepsilon_{i+1},\varepsilon_{i+1})} \bfe_{i+1,i'}\bfe_{i,i+1}$
 \end{tabular}
 & \AddLabel{eq:D-twisted-11} \\
 $E_{i,i+1} \otimes E_{j,j+1}$
 & $1 \leq i < j < s$, $j-i > 1$
 & $[\![ \bfe_{i,i+1}, \bfe_{j,j+1} ]\!]^{\zeta} = 0$
 & \AddLabel{eq:D-twisted-12} \\
 $E_{i,i+1} \otimes E_{s-1,s'}$
 & $1 \leq i \leq s-3$
 & $[\![ \bfe_{i,i+1}, \bfe_{s-1,s'} ]\!]^{\zeta} = 0$
 & \AddLabel{eq:D-twisted-13} \\
 $E_{i,i+1} \otimes E_{i,i+1}$
 & $1 \leq i < s$, $|\bfe_{i,i+1}| = \bbar{1}$
 & $[\![ \bfe_{i,i+1}, \bfe_{i,i+1} ]\!]^{\zeta} = 0$
 & \AddLabel{eq:D-twisted-14} \\
 $E_{s-1,s'} \otimes E_{s-1,s'}$
 & $|\bfe_{s-1,s'}| = \bbar{1}$
 & $[\![ \bfe_{s-1,s'}, \bfe_{s-1,s'} ]\!]^{\zeta} = 0$
 & \AddLabel{eq:D-twisted-15} \\
 $E_{i,i+1} \otimes E_{i,i+2}$
 & $1 \leq i \leq s-2$
 & $\bigl[\!\bigl[ \bfe_{i,i+1}, [\![ \bfe_{i,i+1}, \bfe_{i+1,i+2} ]\!]^{\zeta}\bigr]\!\bigr]^{\zeta} = 0$
 & \AddLabel{eq:D-twisted-16} \\
 $E_{s-2,s-1} \otimes E_{s-2,s'}$
 &
 & $\bigl[\!\bigl[ \bfe_{s-2,s-1}, [\![ \bfe_{s-2,s-1}, \bfe_{s-1,s'} ]\!]^{\zeta}\bigr]\!\bigr]^{\zeta} = 0$
 & \AddLabel{eq:D-twisted-17} \\
 $E_{i,i+1} \otimes E_{i-1,i+1}$
 & $1 < i < s$
 & $\bigl[\!\bigl[[\![ \bfe_{i-1,i}, \bfe_{i,i+1} ]\!]^{\zeta}, \bfe_{i,i+1} \bigr]\!\bigr]^{\zeta} = 0$
 & \AddLabel{eq:D-twisted-18} \\
 $E_{s-1,s'} \otimes E_{s-2,s'}$
 &
 & $\bigl[\!\bigl[[\![ \bfe_{s-2,s-1}, \bfe_{s-1,s'} ]\!]^{\zeta}, \bfe_{s-1,s'} \bigr]\!\bigr]^{\zeta} = 0$
 & \AddLabel{eq:D-twisted-19} \\
 $E_{i,i+1} \otimes E_{i-1,i+2}$
 & $1 < i < s-1$
 & $\bigl[\!\bigl[\bigl[\!\bigl[[\![ \bfe_{i-1,i}, \bfe_{i,i+1} ]\!]^{\zeta}, \bfe_{i+1,i+2} \bigr]\!\bigr]^{\zeta}, \bfe_{i,i+1} \bigr]\!\bigr]^{\zeta} = 0$
 & \AddLabel{eq:D-twisted-20} \\
 $E_{s-2,s-1} \otimes E_{s-3,s'}$
 &
 & ${\bigl[\!\bigl[\bigl[\!\bigl[[\![ \bfe_{s-3,s-2}, \bfe_{s-2,s-1} ]\!]^{\zeta}, \bfe_{s-1,s'} \bigr]\!\bigr]^{\zeta}, \bfe_{s-2,s-1} \bigr]\!\bigr]^{\zeta} =0}$
 & \AddLabel{eq:D-twisted-21}
\end{xltabular}
\vspace{-1mm}
\endgroup

\begin{Remark}\label{rem:BCD-rel-nilpotent}
As in Remark~\ref{rem:A-rel-nilpotent}, while~\eqref{eq:B-twisted-8}, \eqref{eq:C-twisted-10}, \eqref{eq:D-twisted-14}--\eqref{eq:D-twisted-15}
are special cases of~\eqref{eq:indep-twisted-4}, they are recorded separately as they correspond to the Serre relations.
Also~\eqref{eq:indep-twisted-4} implies $\bfe_{ij}^{2} = 0$ for odd isotropic $\varepsilon_{i} - \varepsilon_{j}$.
\end{Remark}

\begin{Remark}
For each type, both the type-independent identities (Table~\ref{table:osp-indep}) and the type-dependent identities listed earlier
for that same type are applied iteratively to establish the subsequent relations, cf.\ Remark~\ref{rem:table-higher-order-rel}.
\end{Remark}

Finally, we record the following exceptional identities for $C$- and $D$-types:

\begin{table}[H]
\centering
\small
\renewcommand{\arraystretch}{1.2}
\setlength{\tabcolsep}{3pt}
{\thickmuskip=1mu
\begin{tabularx}{\textwidth}{|c | c >{\raggedright\arraybackslash}X r|}
 \hline
 Type & Conditions & Identity & \\
 \hline
 \multirow{2}{*}{$C$-type}
 &
 & $\bigl[\!\bigl[ \bfe_{s-1,s}, \bigl[\!\bigl[ \bfe_{s-1,s}, [\![ \bfe_{s-1,s}, \bfe_{ss'} ]\!]^{\zeta} \bigr]\!\bigr]^{\zeta} \bigr]\!\bigr]^{\zeta} = 0$
 & \AddLabel{eq:C-twisted-15} \\
 & $|\bfe_{s-2,s-1}|, |\bfe_{s-1,s}| = \bbar{1}$
 & $\bigl[\!\bigl[ \bigl[\!\bigl[ \bigl[\!\bigl[ [\![ \bfe_{s-2,s-1}, \bfe_{s-1,s} ]\!]^{\zeta}, \bfe_{ss'} \bigr]\!\bigr]^{\zeta},
 [\![ \bfe_{s-2,s-1}, \bfe_{s-1,s} ]\!]^{\zeta} \bigr]\!\bigr]^{\zeta}, \bfe_{s-1,s} \bigr]\!\bigr]^{\zeta} = 0$
 & \AddLabel{eq:C-twisted-16} \\
 \hline
 \multirow{2}{*}{$D$-type}
 & $|\bfe_{s-1,s}| = \bbar{0}$
 & $[\![ \bfe_{s-1,s}, \bfe_{s-1,s'} ]\!]^{\zeta} = 0$
 & \AddLabel{eq:D-twisted-22} \\
 &
 & $\bigl[\!\bigl[ [\![ \bfe_{s-2,s-1}, \bfe_{s-1,s} ]\!]^{\zeta}, \bfe_{s-1,s'} \bigr]\!\bigr]^{\zeta} =
 \bigl[\!\bigl[ [\![ \bfe_{s-2,s-1}, \bfe_{s-1,s'} ]\!]^{\zeta}, \bfe_{s-1,s} \bigr]\!\bigr]^{\zeta}$
 & \AddLabel{eq:D-twisted-23} \\
 \hline
\end{tabularx}}
\caption{Exceptional identities in $\URtw$.}\label{table:CD_exceptional}
\end{table}

The proofs of the exceptional identities are as follows:
\begin{itemize}\itemsep=0pt

\item

Since relation~\eqref{eq:C-twisted-15} obviously holds when \smash{$|\bfe_{s-1,s}| = \bbar{1}$}, due to \smash{$\bfe_{s-1,s}^2 = 0$}
by~\eqref{eq:C-twisted-10}, we shall assume now that \smash{$|\bfe_{s-1,s}| = \bbar{0}$}. Then identity~\eqref{eq:C-twisted-15}
reduces via~\eqref{eq:C-twisted-4} to
$
\smash{ \bigl[\!\bigl[ \bfe_{s-1,s}},\allowbreak\smash{ [\![ \bfe_{s-1,s}, \bfe_{s-1,s'} ]\!]^{\zeta}\bigr]\!\bigr]^{\zeta} = 0}$,
expanding to
\begin{gather}\label{eq:C-twisted-15-1}
 \bfe_{s-1,s}^{2}\bfe_{s-1,s'} - \bigl(1+q^{-2}\bigr)\bfe_{s-1,s}\bfe_{s-1,s'}\bfe_{s-1,s} + q^{-2}\bfe_{s-1,s'}\bfe_{s-1,s}^{2} = 0.
\end{gather}
Evaluating the $E_{s',(s-1)'} \otimes E_{s-1,(s-1)'}$-component of $RL_{1}^{+}L_{2}^{+} = L_{2}^{+}L_{1}^{+}R$
(together with \eqref{eq:indep-twisted-3} and~\eqref{eq:C-twisted-2} for $i=s-1$) yields
\begin{gather}\label{eq:C-twisted-15-2}
 [\bfe_{s-1,s}, \bfe_{s-1,(s-1)'}]^{\zeta} = 0,
\end{gather}
while~\eqref{eq:C-twisted-5} for $i = s-1$ reads
 $\bfe_{s-1,(s-1)'} = \bfe_{s-1,s'}\bfe_{s',(s-1)'} - q^{2}\bfe_{s',(s-1)'}\bfe_{s-1,s'}$.
Substituting the latter into~\eqref{eq:C-twisted-15-2} and applying~\eqref{eq:C-twisted-2} for $i=s-1$ yields a~scalar multiple of~\eqref{eq:C-twisted-15-1}, thus completing the proof of~\eqref{eq:C-twisted-15}.

\item
Using the formulas~\eqref{eq:C-twisted-3} and~\eqref{eq:C-twisted-4}, identity~\eqref{eq:C-twisted-16} reduces to
 $\bigl[\!\bigl[[\![ \bfe_{s-2,s'},\! \bfe_{s-2,s} ]\!]^{\zeta},\! \bfe_{s-1,s} \bigr]\!\bigr]^{\zeta}\allowbreak = 0$,
which by~\eqref{eq:C-twisted-2} is equivalent to
 \smash{$\bigl[\!\bigl[[\![ \bfe_{s-2,s'}, \bfe_{s-2,s} ]\!]^{\zeta}, \bfe_{s',(s-1)'} \bigr]\!\bigr]^{\zeta} = 0$}.
This is further equivalent to
 \smash{$\bigl[\!\bigl[[\![ \bfe_{s-2,s}, \bfe_{s-2,s'} ]\!]^{\zeta}, \bfe_{s',(s-1)'} \bigr]\!\bigr]^{\zeta} = 0$}
as the condition $|\bfe_{s-2,s-1}| = |\bfe_{s-1,s}| = \bbar{1}$ implies~${|\bfe_{s-2,s}| = |\bfe_{s-2,s'}| = \bbar{0}}$ and $(\gamma_{s-2,s}, \gamma_{s-2,s'}) = 0$.
We note the following Jacobi identity:
\begin{gather}
 \bigl[\!\bigl[[\![ \bfe_{s-2,s}, \bfe_{s-2,s'} ]\!]^{\zeta}, \bfe_{s',(s-1)'} \bigr]\!\bigr]^{\zeta}\nonumber\\
 \qquad
 = \bigl[\!\bigl[ \bfe_{s-2,s}, [\![ \bfe_{s-2,s'}, \bfe_{s',(s-1)'} ]\!]^{\zeta}\bigr]\!\bigr]^{\zeta}
 - \bigl[\!\bigl[ \bfe_{s-2,s'}, [\![ \bfe_{s-2,s}, \bfe_{s',(s-1)'} ]\!]^{\zeta}\bigr]\!\bigr]^{\zeta}.\label{eq:C-twisted-16-1}
\end{gather}
The second term in \eqref{eq:C-twisted-16-1} vanishes because $[\![ \bfe_{s-2,s}, \bfe_{s-1,s} ]\!]^{\zeta} = 0$
by \eqref{eq:C-twisted-12} together with \eqref{eq:C-twisted-2} and \eqref{eq:C-twisted-3}. The first term reduces
via \eqref{eq:C-twisted-3} to a~scalar multiple of $[\![ \bfe_{s-2,s},\allowbreak \bfe_{s-2,(s-1)'} ]\!]^{\zeta}$, which can be shown
to vanish by evaluating the $E_{s-2,s} \otimes E_{s-2,(s-1)'}$-compo\-nent of $RL_{1}^{+}L_{2}^{+} = L_{2}^{+}L_{1}^{+}R$
(together with~\eqref{eq:indep-twisted-3}). This completes the verification of~\eqref{eq:C-twisted-16}.

\item
Relation~\eqref{eq:D-twisted-22} follows by comparing the iterative formulas~\eqref{eq:D-twisted-7}
and~\eqref{eq:D-twisted-11} at $i = s-1$, alongside formulas~\eqref{eq:D-twisted-2} and~\eqref{eq:D-twisted-3}
for $i=s-1$.

\item
Applying formulas~\eqref{eq:D-twisted-3} and~\eqref{eq:D-twisted-2} at $i = s-1$ reduces relation~\eqref{eq:D-twisted-23} to
\begin{gather*}
 \bigl[\!\bigl[[\![\bfe_{s-2,s-1},\bfe_{s-1,s}]\!]^{\zeta},\bfe_{s,(s-1)'}\bigr]\!\bigr]^{\zeta} =
 \bigl[\!\bigl[[\![\bfe_{s-2,s-1},\bfe_{s-1,s'}]\!]^{\zeta},\bfe_{s',(s-1)'}\bigr]\!\bigr]^{\zeta}.
\end{gather*}
Formulas~\eqref{eq:D-twisted-4}--\eqref{eq:D-twisted-6} imply that both sides equal
$(-1)^{\overline{s-1}+\overline{s}}\bfe_{s-2,(s-1)'}$, thus completing the proof.
\end{itemize}

\subsection{Homomorphism theorem: orthosymplectic case}\label{ssec:osp-homomorphism thm}

The key result of this section is the identification between the Hopf superalgebras $\uqVe^{\CF}$ and~$U(R)$.

\begin{Theorem}\label{thm:osp DrJ to RTT}\quad
\begin{itemize}\itemsep=0pt
\item[$(a)$] The assignment
\begin{gather}
 q^{\pm \Tilde{H}_{i}} \mapsto q^{\pm \Tilde{\bfH}_{i}} \qquad \text{for } 1 \leq i \leq s+1,\nonumber\\
 e_{i} \mapsto \begin{cases}
 \bfe_{i,i+1} & \text{if } 1 \leq i < s \text{ {\rm(}in all $BCD$-types$)$}, \\
 \bfe_{s,s+1} & \text{if } i = s \text{ {\rm(}$B$-type$)$}, \\
 \bfe_{ss'}/\bigl(1+q^{2}\bigr) & \text{if } i = s \text{ {\rm(}$C$-type$)$}, \\
 \bfe_{s-1,s'} & \text{if } i = s \text{ {\rm(}$D$-type$)$},
 \end{cases}\nonumber\\
 f_{i} \mapsto \begin{cases}
 \bff_{i+1,i} & \text{if } 1 \leq i < s \text{ $($in all $BCD$-types$)$}, \\
 \bff_{s+1,s} & \text{if } i = s \text{ {\rm(}$B$-type$)$}, \\
 q\bff_{s's} & \text{if } i = s \text{ {\rm(}$C$-type$)$}, \\
 \bff_{s',s-1} & \text{if } i = s \text{ {\rm(}$D$-type$)$},
 \end{cases}\label{eq:osp_xi_efh}
\end{gather}
gives rise to a~$Q$-graded Hopf superalgebra isomorphism $\xi \colon \uqVe^{\CF} \iso U(R)$.

\item[$(b)$] Under the isomorphism $\xi$, the images of the quantum root vectors from Section~{\rm\ref{ssec:osp-universal-R}} are given by
\begin{itemize}\itemsep=0pt

\item[$\bullet$] \textbf{$B$-type}:

If $\gamma = \varepsilon_{i} - \varepsilon_{j}$ with $1 \leq i < j \leq s+1$, then
$
 \xi(e_{\gamma}) = \bfe_{ij}$, $
 \xi(f_{\gamma}) = \bff_{ji}$.

If $\gamma = \varepsilon_{i} + \varepsilon_{j}$ with $1 \leq i < j \leq s$, then
\begin{align*}
& \xi(e_{\gamma})=
 \vartheta_{j}\vartheta_{s+1} \cdot \prod_{k=j}^{s} \bigl(-(-1)^{\overline{k}(\overline{k}+\overline{k+1})}\bigr) \cdot \bfe_{ij'},\\
 &\xi(f_{\gamma})=
 \vartheta_{j}\vartheta_{s+1} \cdot \prod_{k=j}^{s} \bigl(-(-1)^{\overline{k}(\overline{k}+\overline{k+1})}\bigr) \cdot \bff_{j'i}.
\end{align*}

\item[$\bullet$] \textbf{$C$-type}:

If $\gamma = \varepsilon_{i} - \varepsilon_{j}$ with $1 \leq i < j \leq s$, then
$
 \xi(e_{\gamma}) = \bfe_{ij}$, $
 \xi(f_{\gamma}) = \bff_{ji}$.

If $\gamma = \varepsilon_{i} + \varepsilon_{j}$ with $1 \leq i < j \leq s$, then
\begin{align*}
 &\xi(e_{\gamma})=
 -\vartheta_{j}\vartheta_{s} \cdot \prod_{k=j}^{s} \bigl(-(-1)^{\overline{k}(\overline{k}+\overline{k+1})}\bigr) \cdot \bfe_{ij'},\\
& \xi(f_{\gamma})=
 -\bigl(q+q^{-1}\bigr)\vartheta_{j}\vartheta_{s} \cdot \prod_{k=j}^{s} \bigl(-(-1)^{\overline{k}(\overline{k}+\overline{k+1})}\bigr) \cdot \bff_{j'i}.
\end{align*}

\item[$\bullet$] \textbf{$D$-type}:

If $\gamma = \varepsilon_{i} - \varepsilon_{j}$ with $1 \leq i < j \leq s$, then
$\xi(e_{\gamma}) = \bfe_{ij}$, $\xi(f_{\gamma}) = \bff_{ji}$.

If $\gamma = \varepsilon_{i} + \varepsilon_{j}$ with $1 \leq i < j \leq s$, then
\begin{align*}
 &\xi(e_{\gamma})=
 -\vartheta_{j}\vartheta_{s} \cdot \prod_{k=j}^{s} \bigl(-(-1)^{\overline{k}(\overline{k}+\overline{k+1})}\bigr) \cdot \bfe_{ij'},\\
& \xi(f_{\gamma})=
 -\vartheta_{j}\vartheta_{s} \cdot \prod_{k=j}^{s} \bigl(-(-1)^{\overline{k}(\overline{k}+\overline{k+1})}\bigr) \cdot \bff_{j'i}.
\end{align*}
\end{itemize}

\item[$(c)$] The isomorphism $\xi$ is compatible with the pairings introduced in Proposition~{\rm\ref{prop:osp-drinfeld-twist-iso}\,(c)}
and Remark~{\rm\ref{rem:osp-RTT double full Cartan}}:
\[
 (\xi(x),\xi(y))_{R} = (x,y)^{\CF}_{\widetilde{\DrJ}} \qquad
 \text{for all $ x \in U^{\geq}_q(\fgosp(V))^{\CF}$, $ y \in U^{\leq}_q(\fgosp(V))^{\CF}$}.
\]
\end{itemize}
\end{Theorem}

\begin{proof}
The proof is completely analogous to that of Theorem~\ref{thm:finite DrJ to RTT}. We first establish that the
assignment~\eqref{eq:osp_xi_efh} defines a~surjective Hopf superalgebra morphism
$\xi\colon U_{q}^{\geq}(\fgosp(V))^{\CF} \to U^{\geq}(R)$. By untwisting the relations in
Tables~\ref{table:osp-indep}--\ref{table:CD_exceptional} via~\eqref{eq:untwisting}, we deduce the following:
\begin{itemize}\itemsep=0pt
\item
The Chevalley-type relations for $U_{q}^{\geq}(\fgosp(V))^{\CF}$ are recovered directly from
identities~\eqref{eq:indep-twisted-1}--\eqref{eq:indep-twisted-3};

\item
The standard Serre relations~\eqref{eq:q-gl-serre-rel-standard-1}--\eqref{eq:q-gl-serre-rel-standard-3}
(for the appropriate indices) and the higher-order Serre relations~\eqref{eq:q-osp-serre-B}--\eqref{eq:q-osp-serre-D} for each
respective type follow from identities~\eqref{eq:B-twisted-7}--\eqref{eq:B-twisted-12} for $B$-type;
\eqref{eq:C-twisted-9}--\eqref{eq:C-twisted-14} along with~(\ref{eq:C-twisted-15})--(\ref{eq:C-twisted-16}) for $C$-type;
and~\eqref{eq:D-twisted-12}--\eqref{eq:D-twisted-21} along with~(\ref{eq:D-twisted-22})--(\ref{eq:D-twisted-23}) for $D$-type;

\item
The surjectivity of $\xi$ restricted to the Cartan subalgebra follows from identities~\eqref{eq:B-twisted-1},
\eqref{eq:C-twisted-1}, and~\eqref{eq:D-twisted-1} for the respective types. The fact that all remaining generators
$\bfe_{ij}$ $(i<j)$ lie in the image is then guaranteed by the iterative formulas:
\eqref{eq:B-twisted-2}--\eqref{eq:B-twisted-6} for $B$-type, \eqref{eq:C-twisted-2}--\eqref{eq:C-twisted-8} for $C$-type,
and~\eqref{eq:D-twisted-2}--\eqref{eq:D-twisted-11} for $D$-type. We also emphasize that for $D$-type, setting $i=s$
in~\eqref{eq:D-twisted-2}, one actually gets
\begin{gather}\label{eq:vanishing-D}
 \bfe_{ss'} = 0;
\end{gather}

\item
In conjunction with~\eqref{eq:quantum root vectors}, these iterative formulas simultaneously establish part~(b);

\item
The fact that $\xi$ intertwines the respective comultiplications follows by the exact same computation as in the $\gl(V)$
setting. The only case deserving a~separate treatment is the generator $\bfe_{s-1,s'}$ for $D$-type, in which case one uses
$l_{ss'}^{+} = 0$, due to~\eqref{eq:vanishing-D}, to derive
 \smash{$\Delta\bigl(l_{s-1,s'}^{+}\bigr) = l_{s-1,s-1}^{+} \otimes l_{s-1,s'}^{+} + l_{s-1,s'}^{+} \otimes l_{s's'}^{+}$}.

\end{itemize}

This confirms that $\xi$ is a~well-defined, surjective Hopf superalgebra morphism of the positive Borel subalgebras.
The corresponding morphism for the negative Borel subalgebras
\[
U_{q}^{\leq}(\fgosp(V))^{\CF} \to U^{\leq}(R)
\]
 follows
by utilizing the superalgebra isomorphisms $\omega_{\DrJ}$ and $\omega_{R}$ from Remarks~\ref{rem:osp-DrJ involution}
and~\ref{rem:osp-U(R) involution}, precisely as executed in the $\gl(V)$ case.

The compatibility of the pairings stated in part (c) is verified by a~direct evaluation on the generators.
Since both pairings are of $Q$-degree zero (cf.~\eqref{eq:skew-pairing-degree-zero}), we have
\begin{gather*}
 \bigl(\bfe_{ij},q^{\pm \Tilde{\bfH}_{k}}\bigr)_{R} = 0
 = \bigl(q^{\pm \Tilde{\bfH}_{k}},\bff_{ji}\bigr)_{R},
\end{gather*}
for all $1 \leq i < j \leq N$, and
\begin{gather*}
 \bigl(q^{\Tilde{\bfH}_{i}},q^{-\Tilde{\bfH}_{j}}\bigr)_{R} = \bigl(l^{+}_{ii},l^{-}_{jj}\bigr)_{R} \overset{\eqref{eq:osp-evaluated-R}}{=}
 q^{-(\Tilde{\varepsilon}_{C},\Tilde{\varepsilon}_{C})}q^{-(\varepsilon_{i},\varepsilon_{j})} =
 q^{-(\Tilde{\varepsilon}_{i},\Tilde{\varepsilon}_{j})} =
 \bigl(q^{\Tilde{H}_{i}},q^{-\Tilde{H}_{j}}\bigr)^{\CF}_{\widetilde{\DrJ}}.
\end{gather*}

Furthermore, for any $1 \leq i < j \leq N$ and $1 \leq k < l \leq N$, expanding the pairing yields
\begin{align}
 \bigl( l_{ii}^{-} l_{ij}^{+}, l_{lk}^{-} l_{kk}^{+} \bigr)_{R}
 & = \bigl( l_{ii}^{-} \otimes l_{ij}^{+}, \Delta\bigl(l_{lk}^{-}\bigr) \Delta\bigl(l_{kk}^{+}\bigr) \bigr)_{R}
 \overset{\eqref{eq:skew-pairing-degree-zero}}{=}
 \bigl( l_{ii}^{-} \otimes l_{ij}^{+}, l_{ll}^{-} l_{kk}^{+} \otimes l_{lk}^{-} l_{kk}^{+} \bigr)_{R}\nonumber\\
 & = \bigl( l_{ii}^{-}, l_{ll}^{-} l_{kk}^{+} \bigr)_{R} \bigl( l_{ij}^{+}, l_{lk}^{-} l_{kk}^{+} \bigr)_{R}
 = q^{-(\Tilde{\varepsilon}_{i},\Tilde{\varepsilon}_{k}-\Tilde{\varepsilon}_{l})}
 \bigl( l_{ij}^{+}, l_{lk}^{-} l_{kk}^{+} \bigr)_{R}\nonumber\\
\overset{\eqref{eq:skew-pairing-degree-zero}}&{=}
 q^{-(\Tilde{\varepsilon}_{i},\Tilde{\varepsilon}_{k}-\Tilde{\varepsilon}_{l})}
 \bigl( l_{ij}^{+} \otimes l_{ii}^{+}, l_{lk}^{-} \otimes l_{kk}^{+} \bigr)_{R}
 = q^{(\Tilde{\varepsilon}_{i},\Tilde{\varepsilon}_{l})} \bigl( l_{ij}^{+}, l_{lk}^{-} \bigr)_{R},\label{eq:pairing-eval}
\end{align}
which vanishes unless $\deg\bigl(l_{ij}^{+}\bigr) = -\deg\bigl(l_{lk}^{-}\bigr)$. Therefore, $(\xi(e_{i}),\xi(f_{j}))_{R} = 0$ for $i \neq j$
due to degree reasons, and it remains to evaluate the pairing for $i=j$. Thus, the proof of part~(c) is concluded by
the following explicit type-dependent evaluations for all $1 \leq i \leq s$ using~\eqref{eq:pairing-eval},
cf.~\eqref{eq:osp-epsilon-pairing}:
\begin{itemize}\itemsep=0pt

\item \emph{$1 \leq i < s$ in $BCD$-types}:
\begin{align*}
 q^{(\Tilde{\varepsilon}_{i},\Tilde{\varepsilon}_{i+1})} \bigl( l_{i,i+1}^{+}, l_{i+1,i}^{-} \bigr)_{R}
 &= -(-1)^{\overline{i+1}}\bigl(q-q^{-1}\bigr)q^{-(\Tilde{\varepsilon}_{C},\Tilde{\varepsilon}_{C})} \cdot (-1)^{\overline{i}+\overline{i+1}}
 q^{(\Tilde{\varepsilon}_{i},\Tilde{\varepsilon}_{i+1})}\\
 &= -(-1)^{\overline{i}}\bigl(q-q^{-1}\bigr),
\end{align*}
and therefore
\begin{align*}
 (\xi(e_{i}),\xi(f_{i}))_{R} &= (\bfe_{i,i+1}, \bff_{i+1,i})_{R} \\
 &= -(-1)^{\overline{i}} \bigl(q-q^{-1}\bigr)^{-2} \bigl( l_{ii}^{-} l_{i,i+1}^{+}, l_{i+1,i}^{-} l_{ii}^{+} \bigr)_{R}
 = \bigl(q-q^{-1}\bigr)^{-1}.
\end{align*}

\item \emph{$i = s$ in $B$-type}:
\begin{align*}
 q^{(\Tilde{\varepsilon}_{s},\Tilde{\varepsilon}_{s+1})} \bigl( l_{s,s+1}^{+}, l_{s+1,s}^{-} \bigr)_{R}
 &= -(-1)^{\overline{s+1}}\bigl(q-q^{-1}\bigr)q^{-(\Tilde{\varepsilon}_{C},\Tilde{\varepsilon}_{C})} \cdot (-1)^{\overline{s}+\overline{s+1}}
 q^{(\Tilde{\varepsilon}_{s},\Tilde{\varepsilon}_{s+1})}
\\
& = -(-1)^{\overline{s}}\bigl(q-q^{-1}\bigr),
\end{align*}
and thus
\begin{align*}
 (\xi(e_{s}),\xi(f_{s}))_{R} &= (\bfe_{s,s+1}, \bff_{s+1,s})_{R} \\
 &= -(-1)^{\overline{s}} \bigl(q-q^{-1}\bigr)^{-2} \bigl( l_{ss}^{-} l_{s,s+1}^{+}, l_{s+1,s}^{-} l_{ss}^{+} \bigr)_{R} = \bigl(q-q^{-1}\bigr)^{-1}.
\end{align*}

\item \emph{$i = s$ in $C$-type}:
\begin{align*}
 q^{(\Tilde{\varepsilon}_{s},\Tilde{\varepsilon}_{s'})} \bigl( l_{ss'}^{+}, l_{s's}^{-} \bigr)_{R}
 &= -\bigl(q-q^{-1}\bigr)q^{-(\Tilde{\varepsilon}_{C},\Tilde{\varepsilon}_{C})}
 \bigl( (-1)^{\overline{s}} - q^{(\rho,2\varepsilon_{s})} \bigr) \cdot q^{(\Tilde{\varepsilon}_{s},\Tilde{\varepsilon}_{s'})} \\
 &= -\bigl(q-q^{-1}\bigr)q^{-(\varepsilon_{s},\varepsilon_{s})} \bigl( (-1)^{\overline{s}} - q^{2(\varepsilon_{s},\varepsilon_{s})} \bigr)
 = \bigl(q-q^{-1}\bigr)\bigl(q+q^{-1}\bigr),
\end{align*}
and hence
\begin{align*}
 (\xi(e_{s}),\xi(f_{s}))_{R} &= \bigl(q+q^{-1}\bigr)^{-1}(\bfe_{ss'}, \bff_{s's})_{R} \\
 &= -(-1)^{\overline{s}} \bigl(q+q^{-1}\bigr)^{-1}\bigl(q-q^{-1}\bigr)^{-2} \bigl( l_{ss}^{-} l_{ss'}^{+}, l_{s's}^{-} l_{ss}^{+} \bigr)_{R} = \bigl(q-q^{-1}\bigr)^{-1}.
\end{align*}

\item \emph{$i = s$ in $D$-type}:
\begin{align*}
 q^{(\Tilde{\varepsilon}_{s-1},\Tilde{\varepsilon}_{s'})} \bigl( l_{s-1,s'}^{+}, l_{s',s-1}^{-} \bigr)_{R}
 &= -(-1)^{\overline{s}}\bigl(q-q^{-1}\bigr)q^{-(\Tilde{\varepsilon}_{C},\Tilde{\varepsilon}_{C})} \cdot (-1)^{\overline{s-1}+\overline{s}}
 q^{(\Tilde{\varepsilon}_{s-1},\Tilde{\varepsilon}_{s'})} \\
 &= -(-1)^{\overline{s-1}}\bigl(q-q^{-1}\bigr),
\end{align*}
and therefore
\begin{align*}
 (\xi(e_{s}),\xi(f_{s}))_{R} &= (\bfe_{s-1,s+1}, \bff_{s+1,s-1})_{R} \\
& = -(-1)^{\overline{s-1}+\overline{s}} \bigl(q-q^{-1}\bigr)^{-2} \bigl( l_{s-1,s-1}^{-} l_{s-1,s'}^{+}, l_{s',s-1}^{-} l_{s-1,s-1}^{+} \bigr)_{R}
 \\
 &= \bigl(q-q^{-1}\bigr)^{-1}.
\end{align*}

\end{itemize}
We conclude that in all cases, the right-hand side evaluates precisely to \smash{$\bigl(q-q^{-1}\bigr)^{-1}=(e_i, f_i)^{\CF}_{\widetilde{\DrJ}}$}.

As in the $\gl(V)$ case, $\xi$ naturally extends to a~surjective Hopf superalgebra morphism between the doubles
\begin{gather*}
 \xi \colon\ \CD^{\CF}_{\DrJ} = U^\geq_{q}(\fgosp(V))^{\CF} \otimes U^\leq_{q}(\fgosp(V))^{\CF}
 \longrightarrow U^{\geq}(R) \otimes U^{\leq}(R) = \CD_{R}.
\end{gather*}
The non-degeneracy of the pairing $(-,-)^{\CF}_{\widetilde{\DrJ}}$ guarantees that $\xi \colon \CD^{\CF}_{\DrJ} \to \CD_{R}$
is injective on each Borel subalgebra, thereby establishing an isomorphism of doubles. Finally, the identification of
the Cartan subalgebras yields the desired isomorphism $\xi \colon \uqVe^{\CF} \iso U(R)$, completing the proof of~(a).
\end{proof}

\subsection{Inverse morphism: orthosymplectic case}\label{ssec:inverse morphism osp}

Following the approach in the general linear setting (see Section~\ref{ssec:inverse morphism}), we shall now construct a~similar morphism for the orthosymplectic quantum supergroups. The proof of the following proposition is completely analogous to that of Proposition~\ref{prop:gl-inverse-morphism}.

\begin{Proposition}
The assignment
\begin{gather}\label{eq:DF morphism osp}
 L^{+} \mapsto (\Id \otimes \varrho)(\fR_{(21)}),\qquad
 L^{-} \mapsto (\Id \otimes \varrho)\bigl(\fR^{-1}\bigr),
\end{gather}
gives rise to a~$Q$-graded Hopf superalgebra morphism $\phi^{\DF} \colon U(R)^{\copp} \to \uqVe$.
\end{Proposition}

We now determine the explicit images of the generators under $\phi^{\DF}$, cf.\ Lemma~\ref{lem:DF gl image}.

\begin{Lemma}
The images of the generators $\{l^{\pm}_{ii}\}_{i=1}^{N}$ under $\phi^{\DF}$ are given by
$\phi^{\DF}\bigl(l^{\pm}_{ii}\bigr) = q^{\mp \Tilde{H}_{i}}$. For the generators $l^{+}_{ij}$, $l^{-}_{ji}$ associated with a simple root $\gamma_{ij} = \varepsilon_{i} - \varepsilon_{j}$ from~\eqref{eq:gosp-reduced-roots}, their images under $\phi^{\DF}$ are as follows:
\begin{align*}
 l^{+}_{ij} \mapsto -\bigl(q-q^{-1}\bigr) f_{\gamma_{ij}}q^{-\Tilde{H}_{j}},\qquad
 l^{-}_{ji} \mapsto
 \begin{cases}
 -\bigl(q^{2}-q^{-2}\bigr) q^{\Tilde{H}_{s'}} e_s & \text{if } \gamma_{ij} = \alpha_{s} \text{ for $C$-type}, \\
 (-1)^{\overline{j}} \bigl(q-q^{-1}\bigr) q^{\Tilde{H}_{j}} e_{\gamma_{ij}} & \text{otherwise}.
 \end{cases}
\end{align*}
\end{Lemma}

\begin{proof}

As in the general linear case, we first evaluate the semisimple component. An explicit $\hbar$-adic computation yields
\[
 (\Id \otimes \varrho)(\fR_{s})
 = \sum_{1\leq i\leq N} q^{-\Tilde{H}_{i}} \otimes E_{ii},
\]
which establishes $\phi^{\DF}\bigl(l^{\pm}_{ii}\bigr) = q^{\mp \Tilde{H}_{i}}$. Proceeding to the unipotent part, for the simple root
$\gamma_{ij} = \varepsilon_{i} - \varepsilon_{j}$, we need to determine the $E_{ij}$-component of
\begin{gather*}
 (\Id \otimes \varrho)((\fR_{u})_{(21)})
 = \prod_{\gamma \in \bar{\Phi}^{+}}^{\leftarrow}
 \left( \sum^{k\geq 0}_{k \leq 1 \text{ if } \gamma \in \bar{\Phi}_{\bbar{1}} \text{ isotropic}}
 (-1)^{|e_{\gamma}^{k}|} \frac{f_\gamma^k \otimes \varrho\bigl(e_\gamma^k\bigr)}{\bigl(f_\gamma^k, e_\gamma^k\bigr)_{\DrJ}} \right).
\end{gather*}
Because the representation $\varrho$ preserves the $Q$-degree, the only contribution to this component arises from
\[
 (-1)^{|e_{\gamma_{ij}}|} \frac{f_{\gamma_{ij}} \otimes \varrho(e_{\gamma_{ij}})}{(f_{\gamma_{ij}}, e_{\gamma_{ij}})_{\DrJ}}
 = -\bigl(q-q^{-1}\bigr) f_{\gamma_{ij}} \otimes \varrho(e_{\gamma_{ij}}).
 \]
As the coefficient of $E_{ij}$ in each~$\varrho(e_{\gamma_{ij}})$ is $1$
by~\eqref{eq:Lie-action-case-B}--\eqref{eq:Lie-action-case-D}, this implies that
$\phi^{\DF}\bigl(l^{+}_{ij}\bigr) = -\bigl(q -\smash{ q^{-1}\bigr) f_{\gamma_{ij}}q^{-\Tilde{H}_{j}}}$.

Let us now determine the $E_{ji}$-component of
\begin{align*}
 (\Id \otimes \varrho)\bigl(\fR_{u}^{-1}\bigr)
 &= \prod_{\gamma \in \bar{\Phi}^{+}}^{\rightarrow} \left(
 \sum^{k\geq 0}_{k \leq 1 \text{ if } \gamma \in \bar{\Phi}_{\bbar{1}} \text{ isotropic}}
 \frac{e_\gamma^k \otimes \varrho\bigl(f_\gamma^k\bigr)}{\bigl(f_\gamma^k, e_\gamma^k\bigr)_{\DrJ}} \right)^{-1} \\
 &= \prod_{\gamma \in \bar{\Phi}^{+}}^{\rightarrow}
 \sum_{r \geq 0} \left( -\sum^{k\geq 1}_{k \leq 1 \text{ if } \gamma \in \bar{\Phi}_{\bbar{1}} \text{ isotropic}}
 \frac{e_\gamma^k \otimes \varrho\bigl(f_\gamma^k\bigr)}{\bigl(f_\gamma^k, e_\gamma^k\bigr)_{\DrJ}} \right)^{r},
\end{align*}
where we utilize the nilpotency of
\smash{$
\sum^{k\geq 1}_{k \leq 1 \text{ if } \gamma \in \bar{\Phi}_{\bbar{1}} \text{ isotropic}}
 \frac{e_\gamma^k \otimes \varrho(f_\gamma^k)}{(f_\gamma^k, e_\gamma^k)_{\DrJ}}$}.
The sole contributing term is
\begin{gather*}
 - \frac{e_{\gamma_{ij}} \otimes \varrho(f_{\gamma_{ij}})}{(f_{\gamma_{ij}}, e_{\gamma_{ij}})_{\DrJ}}
 = (-1)^{\overline{i}+\overline{j}} \bigl(q-q^{-1}\bigr) e_{\gamma_{ij}} \otimes \varrho(f_{\gamma_{ij}}).
\end{gather*}
Within $\varrho(f_{\gamma_{ij}})$, the coefficient of the matrix unit $E_{ji}$ evaluates to $-\bigl(q+q^{-1}\bigr)$ when
$\gamma_{ij} = \alpha_s$ in $C$-type, and to \smash{$(-1)^{\overline{i}}$} in all other cases, see~\eqref{eq:kappa-const} combined
with~\eqref{eq:Lie-action-case-B}--\eqref{eq:Lie-action-case-D}. Therefore,
\begin{gather*}
 \phi^{\DF}\bigl(l^{-}_{ji}\bigr) =
 \begin{cases}
 -\bigl(q^{2}-q^{-2}\bigr) q^{\Tilde{H}_{s'}} e_s & \text{if } \gamma_{ij} = \alpha_{s} \text{ for $C$-type}, \\
 (-1)^{\overline{j}} \bigl(q-q^{-1}\bigr) q^{\Tilde{H}_{j}} e_{\gamma_{ij}} & \text{otherwise}.
 \end{cases}
\end{gather*}
This completes the proof.
\end{proof}

We conclude with the following analogue of Proposition~\ref{prop:almost-inverse-gl}.

\begin{Proposition}\label{prop:almost-inverse-osp}
The composition
\begin{gather*}
 \uqVe \xrightarrow{\phi_{\CF}^{-1}} \uqVe^{\CF}
 \xrightarrow[\eqref{eq:osp_xi_efh}]{\xi} U(R)
 \xrightarrow[\eqref{eq:osp-transpose-iso}]{\omega_{R}^{-1}} U(R)^{\copp}
 \xrightarrow[\eqref{eq:DF morphism osp}]{\phi^{\DF}} \uqVe
\end{gather*}
$($see Proposition~{\rm\ref{prop:osp-drinfeld-twist-iso}} for $\phi_{\CF})$ is the identity morphism.
\end{Proposition}

\begin{proof}
This is verified by direct computation. For the Cartan generators, we have
\begin{gather*}
 q^{\pm \Tilde{H}_{i}} \xmapsto{\phi_{\CF}^{-1}} q^{\pm \Tilde{H}_{i}} \xmapsto{\xi} q^{\pm \Tilde{\bfH}_{i}}
 \xmapsto{\omega_{R}^{-1}} q^{\mp \Tilde{\bfH}_{i}} \xmapsto{\phi^{\DF}} q^{\pm \Tilde{H}_{i}},
\end{gather*}
while the rest is based on the following type-dependent computations:
\begin{itemize}\itemsep=0pt

\item \emph{$1 \leq i < s$ in $BCD$-types}:
\begin{align*}
 e_{i} &\xmapsto{\phi_{\CF}^{-1}} e_{i}q^{h_{i}} \xmapsto{\xi} \bfe_{i,i+1} q^{\Tilde{\bfH}_{i,i+1}} \xmapsto{\omega_{R}^{-1}}
 -(-1)^{\overline{i}\, \overline{i+1}} (-1)^{\overline{i}+\overline{i+1}} q^{3(\rho, \alpha_{i})} \bff_{i+1,i} q^{-\Tilde{\bfH}_{i,i+1}}
 \xmapsto{\phi^{\DF}} e_{i},\\
 f_{i} &\xmapsto{\phi_{\CF}^{-1}} q^{-h_{i}}f_{i} \xmapsto{\xi} q^{-\Tilde{\bfH}_{i,i+1}} \bff_{i+1,i} \xmapsto{\omega_{R}^{-1}}
 -(-1)^{\overline{i}\, \overline{i+1}} q^{-3(\rho, \alpha_{i})} q^{\Tilde{\bfH}_{i,i+1}} \bfe_{i,i+1}
 \xmapsto{\phi^{\DF}} f_{i}.
\end{align*}

\item \emph{$i = s$ in $B$-type}:
\begin{align*}
 e_{s} &\xmapsto{\phi_{\CF}^{-1}} e_{s}q^{h_{s}} \xmapsto{\xi} \bfe_{s,s+1} q^{\Tilde{\bfH}_{s,s+1}} \xmapsto{\omega_{R}^{-1}}
 -(-1)^{\overline{s} \overline{s+1}} (-1)^{\overline{s}+\overline{s+1}} q^{3(\rho, \alpha_{s})} \bff_{s+1,s} q^{-\Tilde{\bfH}_{s,s+1}}\\
 &
 \xmapsto{\phi^{\DF}} e_{s},\\
 f_{s} &\xmapsto{\phi_{\CF}^{-1}} q^{-h_{s}}f_{s} \xmapsto{\xi} q^{-\Tilde{\bfH}_{s,s+1}} \bff_{s+1,s} \xmapsto{\omega_{R}^{-1}}
 -(-1)^{\overline{s} \overline{s+1}} q^{-3(\rho, \alpha_{s})} q^{\Tilde{\bfH}_{s,s+1}} \bfe_{s,s+1}
 \xmapsto{\phi^{\DF}} f_{s}.
\end{align*}

\item \emph{$i = s$ in $C$-type}:
\begin{align*}
 e_{s} &\xmapsto{\phi_{\CF}^{-1}} e_{s}q^{h_{s}} \xmapsto{\xi}
 \frac{1}{1+q^{2}} \bfe_{ss'} q^{\Tilde{\bfH}_{s,s'}} \xmapsto{\omega_{R}^{-1}}
 \frac{1}{q^{4}\bigl(1+q^{2}\bigr)} \bff_{s's} q^{-\Tilde{\bfH}_{s,s'}}
 \xmapsto{\phi^{\DF}} e_{s},\\
 f_{s} &\xmapsto{\phi_{\CF}^{-1}} q^{-h_{s}}f_{s} \xmapsto{\xi}
 q \cdot q^{-\Tilde{\bfH}_{s,s'}} \bff_{s's} \xmapsto{\omega_{R}^{-1}}
 q^{5} \cdot q^{\Tilde{\bfH}_{s,s'}} \bfe_{ss'}
 \xmapsto{\phi^{\DF}} f_{s}.
\end{align*}

\item \emph{$i = s$ in $D$-type}:
\begin{gather*}
 e_{s} \xmapsto{\phi_{\CF}^{-1}} e_{s}q^{h_{s}} \xmapsto{\xi} \bfe_{s-1,s'} q^{\Tilde{\bfH}_{s-1,s'}} \xmapsto{\omega_{R}^{-1}}
 -(-1)^{\overline{s-1} \overline{s}} (-1)^{\overline{s-1}+\overline{s}} q^{3(\rho, \alpha_{s})} \bff_{s',s-1} q^{-\Tilde{\bfH}_{s-1,s'}}\\
 \phantom{e_{s}}{}
 \xmapsto{\phi^{\DF}} e_{s},\\
 f_{s} \xmapsto{\phi_{\CF}^{-1}} q^{-h_{s}}f_{s}\! \xmapsto{\xi} q^{-\Tilde{\bfH}_{s-1,s'}} \bff_{s',s-1}\! \xmapsto{\omega_{R}^{-1}}
 -(-1)^{\overline{s-1} \overline{s'}}\! q^{-3(\rho, \alpha_{s})} q^{\Tilde{\bfH}_{s-1,s'}} \bfe_{s-1,s'}\!
 \xmapsto{\phi^{\DF}}\! f_{s}. \tag*{\qed}
\end{gather*}
\end{itemize}
\renewcommand{\qed}{}
\end{proof}

\appendix

\section{Generalized double construction in the super setting}\label{sec:double-construction}

This appendix is devoted to the generalization of the non-super constructions from~\cite{maj} to the setting of superbialgebras.
Specifically, we systematically incorporate the $\BZ_{2}$-grading and the corresponding super-sign conventions into the standard
theory of matched pairs. Throughout this appendix, we assume for convenience that all elements of superbialgebras we consider
are $\BZ_{2}$-homogeneous.

\subsection{Matched pairs of superbialgebras}

Let $A$ and $B$ be superbialgebras. Our goal is to endow the superspace $A \otimes B$ with a~superbialgebra structure such
that $A$ and $B$ can be identified as subbialgebras via the natural embeddings
\begin{alignat*}{3}
& A \hookrightarrow A \otimes B, \qquad && a \mapsto a \otimes 1 , &\\
& B \hookrightarrow A \otimes B, \qquad && b \mapsto 1 \otimes b.&
\end{alignat*}
Requiring these embeddings to be superbialgebra morphisms (and assuming $A \otimes B$ is endowed with such a~superalgebra
structure) forces the following multiplication rules:
\begin{gather}
 (a \otimes 1)(a' \otimes 1) = (aa' \otimes 1),\qquad
 (1 \otimes b)(1 \otimes b') = (1 \otimes bb'),\nonumber\\
 (a \otimes 1)(1 \otimes b) = (a \otimes b).\label{eq:matched-pair-mult}
\end{gather}
{\samepage Analogously, for the comultiplication\footnote{Throughout this appendix, unless stated otherwise, we omit parentheses in the subscripts and write the Sweedler notation as \smash{$\Delta(a) = \sum_{(a)} a_{1} \otimes a_{2}$} instead of $a_{(1)} \otimes a_{(2)}$.
Additionally, to maintain brevity during iterated comultiplications, we suppress the summation symbols for subsequent components
such as $(a_i)$ or $(b_i)$.}
\begin{align*}
 \Delta(a \otimes b) = \Delta(a \otimes 1)\Delta(1 \otimes b)
 &= \sum_{(a)(b)} (-1)^{|a_{2}||b_{1}|} (a_{1} \otimes 1)(1 \otimes b_{1}) \otimes (a_{2} \otimes 1)(1 \otimes b_{2})\nonumber\\
 &= \sum_{(a)(b)} (-1)^{|a_{2}||b_{1}|} (a_{1} \otimes b_{1}) \otimes (a_{2} \otimes b_{2}),
\end{align*}
and for the counit
$
 \epsilon(a \otimes b) = \epsilon(a \otimes 1)\epsilon(1 \otimes b) = \epsilon(a)\epsilon(b)$.}

The only remaining nontrivial structure is the cross-multiplication $(1 \otimes b)(a \otimes 1)$. To define this,
we introduce morphisms (hence degree-preserving)
\[
 \rhd \colon \ B \otimes A \to A ,\qquad \lhd \colon \ B \otimes A \to B,
\]
and define the multiplication as follows:
\begin{gather}\label{eq:matched-pair-cross-mult}
 (1 \otimes b)(a \otimes 1) = \sum_{(a)(b)} (-1)^{|b_{2}||a_{1}|} (b_{1} \rhd a_{1}) \otimes (b_{2} \lhd a_{2}).
\end{gather}

\begin{Proposition}\label{prop:matched-pair-super}
Suppose that $\rhd$ defines a~left $B$-module supercoalgebra structure on~$A$, and~$\lhd$ defines a~right $A$-module
supercoalgebra structure on $B$. If these maps satisfy the following compatibility conditions:
\begin{gather*}
 b \rhd aa' = \sum_{(b)(a)} (-1)^{|b_{2}||a_{1}|} (b_{1} \rhd a_{1}) \cdot ((b_{2} \lhd a_{2}) \rhd a'),\qquad
 b \rhd 1 = \epsilon(b),\\
 bb' \lhd a = \sum_{(b')(a)} (-1)^{|b'_{2}||a_{1}|} (b \lhd (b'_{1} \rhd a_{1})) \cdot (b'_{2} \lhd a_{2}),\qquad
 1 \lhd a = \epsilon(a),\\
 \sum_{(a)(b)} (-1)^{|b_{2}||a_{1}|} (b_{1} \lhd a_{1}) \otimes (b_{2} \rhd a_{2}) \\
 \qquad =
 \sum_{(a)(b)} (-1)^{|b_{2}||a_{1}|} (-1)^{(|a_{1}|+|b_{1}|)(|a_{2}|+|b_{2}|)}
 (b_{2} \lhd a_{2}) \otimes (b_{1} \rhd a_{1}),
\end{gather*}
then the pair $(A,B)$ is called a~matched pair of superbialgebras. In this case, the preceding
formulas~\eqref{eq:matched-pair-mult}--\eqref{eq:matched-pair-cross-mult} endow $A \otimes B$ with a~well-defined superbialgebra structure, called the double cross product of $A$ and $B$ and denoted by $A \bowtie B$. Moreover, if $A$ and $B$ are Hopf superalgebras with respective antipodes $S_{A}$ and $S_{B}$, then $A \bowtie B$ is also a~Hopf superalgebra with antipode~${S_{A} \otimes S_{B}}$.
\end{Proposition}

\begin{proof}
By the definition of module supercoalgebra structures and module actions, the following identities hold:
\begin{align*}
 \Delta(b \rhd a) &= \sum_{(a)(b)} (-1)^{|b_{2}||a_{1}|} (b_{1} \rhd a_{1}) \otimes (b_{2} \rhd a_{2}),\qquad
 \epsilon(b \rhd a) = \epsilon(b)\epsilon(a),\\
 \Delta(b \lhd a) &= \sum_{(a)(b)} (-1)^{|b_{2}||a_{1}|} (b_{1} \lhd a_{1}) \otimes (b_{2} \lhd a_{2}),\qquad
 \epsilon(b \lhd a) = \epsilon(b)\epsilon(a),
\end{align*}
as well as
$
 b \rhd (b' \rhd a) = bb' \rhd a$, $ 1 \rhd a = a$, $
 (b \lhd a) \lhd a' = b \lhd aa'$, $ b \lhd 1 = b$.
We now verify the superbialgebra axioms. Since many of the required conditions hold trivially by construction, we exhibit
only the nontrivial cases.

\noindent
$\bullet$ \textbf{Associativity 1:}
\begin{align*}
 (1 \otimes b)\{(1 \otimes b')(a \otimes 1)\}
 ={} &(1 \otimes b) \sum_{(b')(a)} (-1)^{|b'_{2}||a_{1}|} (b'_{1} \rhd a_{1}) \otimes (b'_{2} \lhd a_{2}) \\
 ={} & \sum_{(b)(b')(a)} (-1)^{|b'_{2}||a_{1}|} (-1)^{|b'_{1,2}||a_{1,1}|} (-1)^{|b_{2}|(|b'_{1,1}|+|a_{1,1}|)}\\
 & \times \{b_{1} \rhd (b'_{1,1} \rhd a_{1,1})\} \otimes \{b_{2} \lhd (b'_{1,2} \rhd a_{1,2})\}(b'_{2} \lhd a_{2})\\
 ={} & \sum_{(b)(b')(a)} (-1)^{|b'_{3}|(|a_{1}|+|a_{2}|)} (-1)^{|b'_{2}||a_{1}|} (-1)^{|b_{2}|(|b'_{1}|+|a_{1}|)}\\
 & \times \{b_{1} \rhd (b'_{1} \rhd a_{1})\} \otimes \{b_{2} \lhd (b'_{2} \rhd a_{2})\}(b'_{3} \lhd a_{3})\\
 ={} &\sum_{(b)(b')(a)} (-1)^{|b'_{2}||a_{1}|} (-1)^{|b_{2}|(|b'_{1}|+|a_{1}|)}
 \{b_{1}b'_{1} \rhd a_{1}\} \otimes \{b_{2}b'_{2} \lhd a_{2}\},
\end{align*}
which coincides with
\begin{align*}
 \{(1 \otimes b)(1 \otimes b')\} (a \otimes 1)
 &= (1 \otimes bb')(a \otimes 1) \\
 &= \sum_{(b)(b')(a)} (-1)^{|b_{2}||b'_{1}|} (-1)^{(|b_{2}|+|b'_{2}|)|a_{1}|}
 (b_{1}b'_{1} \rhd a_{1}) \otimes (b_{2}b'_{2} \lhd a_{2}).
\end{align*}

\noindent
$\bullet$ \textbf{Associativity 2:}
\begin{align*}
 \{(1 \otimes b)(a \otimes 1)\}(a' \otimes 1) ={}&
 \sum_{(b)(a)} (-1)^{|b_{2}||a_{1}|} \{(b_{1} \rhd a_{1}) \otimes (b_{2} \lhd a_{2})\}(a' \otimes 1)\\
 ={} &\sum_{(b)(a)(a')} (-1)^{|b_{2}||a_{1}|} (-1)^{|b_{2,2}||a_{2,1}|} (-1)^{(|b_{2,2}|+|a_{2,2}|)|a'_{1}|}\\
 & \cdot \{(b_{1} \rhd a_{1})((b_{2,1} \lhd a_{2,1}) \rhd a'_{1})\} \otimes \{(b_{2,2} \lhd a_{2,2}) \lhd a'_{2}\}\\
 ={} & \sum_{(b)(a)(a')} (-1)^{(|b_{2}|+|b_{3}|)|a_{1}|} (-1)^{|b_{3}||a_{2}|} (-1)^{(|b_{3}|+|a_{3}|)|a'_{1}|}\\
 & \cdot \{(b_{1} \rhd a_{1})((b_{2} \lhd a_{2}) \rhd a'_{1})\} \otimes \{(b_{3} \lhd a_{3}) \lhd a'_{2}\}\\
 ={} & \sum_{(b)(a)(a')} (-1)^{|b_{2}||a_{1}|} (-1)^{(|b_{2}|+|a_{2}|)|a'_{1}|}
 \{b_{1} \rhd a_{1}a'_{1}\} \otimes \{b_{2} \lhd a_{2}a'_{2}\},
\end{align*}
which coincides with
\begin{align*}
 (1 \otimes b)\{(a \otimes 1)(a' \otimes 1)\}
 &= (1 \otimes b)(aa' \otimes 1) \\
 &= \sum_{(b)(a)(a')} (-1)^{|a_{2}||a'_{1}|} (-1)^{|b_{2}|(|a_{1}|+|a'_{1}|)}
 (b_{1} \rhd a_{1}a'_{1}) \otimes (b_{2} \lhd a_{2}a'_{2}).
\end{align*}

\noindent
$\bullet$ \textbf{Unit 1:}
\begin{align*}
 (1 \otimes b)(1 \otimes 1) = \sum_{(b)} (b_{1} \rhd 1) \otimes (b_{2} \lhd 1) =
 \sum_{(b)} \epsilon(b_{1}) \otimes b_{2} = 1 \otimes b.
\end{align*}

\noindent
$\bullet$ \textbf{Unit 2:}
\begin{align*}
 (1 \otimes 1)(a \otimes 1) = \sum_{(a)} (1 \rhd a_{1}) \otimes (1 \lhd a_{2}) =
 \sum_{(a)} a_{1} \otimes \epsilon(a_{2}) = a \otimes 1.
\end{align*}

\noindent
$\bullet$ \textbf{Comultiplication and multiplication:}
{\samepage\begin{gather*}
 \Delta\{(1 \otimes b)(a \otimes 1)\}\\
 \qquad
 = \sum_{(a)(b)} (-1)^{|b_{2}||a_{1}|} \Delta\{(b_{1} \rhd a_{1}) \otimes (b_{2} \lhd a_{2})\} \\
 \qquad = \sum_{(a)(b)} (-1)^{(|b_{3}|+|b_{4}|)(|a_{1}|+|a_{2}|)} (-1)^{|b_{2}||a_{1}|} (-1)^{|b_{4}||a_{3}|}
 (-1)^{(|b_{2}|+|a_{2}|)(|b_{3}|+|a_{3}|)}\\
\phantom{\qquad =}{}\times
 \{(b_{1} \rhd a_{1}) \otimes (b_{3} \lhd a_{3})\} \otimes \{(b_{2} \rhd a_{2}) \otimes (b_{4} \lhd a_{4})\} \\
 \qquad = \sum_{(a)(b)} (-1)^{|b_{4}||a_{1}|}
 (-1)^{(|b_{2}|+|b_{3}|)|a_{1}|} (-1)^{|b_{4}|(|a_{2}|+|a_{3}|)} (-1)^{|b_{3}||a_{2}|}
 (-1)^{(|b_{2}|+|a_{2}|)(|b_{3}|+|a_{3}|)}\\
\phantom{\qquad =}{}\times
 \{(b_{1} \rhd a_{1}) \otimes (b_{3} \lhd a_{3})\} \otimes \{(b_{2} \rhd a_{2}) \otimes (b_{4} \lhd a_{4})\} \\
 \qquad = \sum_{(a)(b)} (-1)^{|b_{4}||a_{1}|}
 (-1)^{(|b_{2}|+|b_{3}|)|a_{1}|} (-1)^{|b_{4}|(|a_{2}|+|a_{3}|)} (-1)^{|b_{3}||a_{2}|}\\
\phantom{\qquad =}{}\times
 \{(b_{1} \rhd a_{1}) \otimes (b_{2} \lhd a_{2})\} \otimes \{(b_{3} \rhd a_{3}) \otimes (b_{4} \lhd a_{4})\},
\end{gather*}}
which coincides with
\begin{align*}
 \Delta(1 \otimes b)\Delta(a \otimes 1)
 ={}& \sum_{(a)(b)} (-1)^{|b_{2}||a_{1}|} (1 \otimes b_{1})(a_{1} \otimes 1) \otimes (1 \otimes b_{2})(a_{2} \otimes 1) \\
 ={}& \sum_{(a)(b)} (-1)^{(|b_{3}|+|b_{4}|)(|a_{1}|+|a_{2}|)} (-1)^{|b_{2}||a_{1}|} (-1)^{|b_{4}||a_{3}|} \\
 &\times \{(b_{1} \rhd a_{1}) \otimes (b_{2} \lhd a_{2})\} \otimes \{(b_{3} \rhd a_{3}) \otimes (b_{4} \lhd a_{4})\}.
\end{align*}

\noindent
$\bullet$ \textbf{Counit and multiplication:}
Since $\epsilon$ vanishes on odd-degree elements,
\begin{align*}
\epsilon\{(1 \otimes b)(a \otimes 1)\}
 &= \sum_{(a)(b)} (-1)^{|b_{2}||a_{1}|} \epsilon\{(b_{1} \rhd a_{1}) \otimes (b_{2} \lhd a_{2})\}\\
 &
 = \sum_{(a)(b)} (-1)^{|b_{2}||a_{1}|} \epsilon(b_{1} \rhd a_{1})\epsilon(b_{2} \lhd a_{2}) \\
 &= \sum_{(a)(b)} (-1)^{|b_{2}||a_{1}|} \epsilon(b_{1})\epsilon(a_{1})\epsilon(b_{2})\epsilon(a_{2})\\
 &
 = \sum_{(a)(b)} \epsilon(b_{1})\epsilon(a_{1})\epsilon(b_{2})\epsilon(a_{2})
 = \epsilon(b)\epsilon(a),
\end{align*}
which coincides with $\epsilon(1 \otimes b)\epsilon(a \otimes 1) = \epsilon(b)\epsilon(a)$.
\end{proof}

\subsection{Convolution product}

Let us briefly recall the \emph{convolution product} that will be frequently used below. Let $(A, \mu_{A}, \eta_{A})$
be a~superalgebra, $(C, \Delta_{C}, \epsilon_{C})$ be a~supercoalgebra, and consider the space of morphisms $\Hom_{\sVect}(C,A)$.
For any two morphisms $\phi, \psi \colon C \to A$, their convolution $\phi \star \psi \colon C \to A$ is defined as
the composition $\mu_{A} \circ (\phi \otimes \psi) \circ \Delta_{C}$. This equips $\Hom_{\sVect}(C,A)$ with a~unital
associative (non-super) algebra structure, with the identity element $\eta_{A} \circ \epsilon_{C}$. A~morphism is called \emph{convolution invertible} if it admits an inverse with respect to this product.

\begin{Remark}
We restrict the convolution product to the space of morphisms $\Hom_{\sVect}(C,A)$, i.e., degree-preserving linear maps.
Alternatively, one could use the internal $\Hom_{\BK}(C,A)$ (cf.\ Section~\ref{ssec:super A}) consisting of all linear maps,
in which case the convolution algebra naturally becomes a~superalgebra. Unless stated otherwise, we shall work in the space
of morphisms and thus frequently omit the subscript $\sVect$.
\end{Remark}

The following functoriality properties of the convolution product are straightforward.

\begin{Lemma}\label{lem:conv prod id}
Let $A$ and $A'$ be superalgebras, and let $C$ and $C'$ be supercoalgebras.
\begin{itemize}\itemsep=0pt
\item[$(a)$] For any superalgebra morphism $\phi \colon A \to A'$, the postcomposition $\phi_{*} \colon \Hom(C,A) \to \linebreak {\Hom(C, A')}$
is an algebra morphism $($i.e., it preserves the convolution product$)$.

\item[$(b)$] For any supercoalgebra morphism $\phi \colon C \to C'$, the precomposition $\phi^{*} \colon \Hom(C'\!,A) \to \Hom(C,A)$
is an algebra morphism $($i.e., it preserves the convolution product$)$.
\end{itemize}
\end{Lemma}

We also note the following general result.

\begin{Lemma}\label{lem:convolution inverse anti-morphism}
Let $A$ be a~superalgebra, $C$ be a~supercoalgebra, and consider the space $\Hom(C,A)$ endowed with the convolution product.
\begin{itemize}\itemsep=0pt
\item[$(a)$] If $C$ has a~superbialgebra structure and $\phi \colon C \to A$ is a~superalgebra morphism, then its convolution inverse
$\phi'$ is a~superalgebra morphism from $C$ to $A^{\opp}$ {\rm(}or equivalently, from~$C^{\opp}$ to $A${\rm)}.

\item[$(b)$] If $A$ has a~superbialgebra structure and $\phi \colon C \to A$ is a~supercoalgebra morphism, then its convolution inverse
$\phi'$ is a~supercoalgebra morphism from $C^{\copp}$ to $A$ {\rm(}or equivalently, from~$C$ to $A^{\copp}${\rm)}.
\end{itemize}
\end{Lemma}

\begin{proof}
(a) As the multiplication $\mu_{C} \colon C \otimes C \to C$ is a~supercoalgebra morphism, Lemma~\ref{lem:conv prod id}\,(b)
implies that $\phi' \circ \mu_{C} \colon C \otimes C \to A$ is a~convolution inverse of $\phi \circ \mu_{C}$. Furthermore,
one can check that $g \coloneqq \mu_{A} \circ \tau_{AA} \circ (\phi' \otimes \phi')$ is also a~convolution inverse of
$f \coloneqq \phi \circ \mu_{C} = \mu_{A} \circ (\phi \otimes \phi)$ via the following computation:
\begin{align*}
 (f \star g)(c \otimes c')
 &= \sum_{(c)(c')} (-1)^{|c_{2}||c'_{1}|} f(c_{1} \otimes c'_{1}) g(c_{2} \otimes c'_{2}) \\
 &= \sum_{(c)(c')} (-1)^{|c_{2}||c'|} \phi(c_{1})\phi(c'_{1}) \phi'(c'_{2})\phi'(c_{2}) \\
 &= \sum_{(c)} \epsilon_{C}(c') \phi(c_{1})\phi'(c_{2})
 = \epsilon_{C}(c)\epsilon_{C}(c')
 = \epsilon_{C \otimes C}(c \otimes c').
\end{align*}
By the uniqueness of the convolution inverse, we get $\phi' \circ \mu_{C} = \mu_{A} \circ \tau_{AA} \circ (\phi' \otimes \phi')$,
completing the proof.

(b) The proof is completely analogous to that of part (a).
\end{proof}

\subsection{Coadjoint actions}\label{ssec:coadjoint actions}

Let $A$, $B$ be superbialgebras over a~field $\BK$, and consider a~\emph{dual pairing} between $A$ and $B$, i.e., a~bilinear map
$
 (\cdot,\cdot) \colon A \otimes B \to \BK$,
satisfying
\begin{gather*}
 (1,b) = \epsilon(b),\qquad (aa', b) = (a \otimes a', \Delta(b)), \qquad (a,1) = \epsilon(a),
\qquad (a, bb') = (\Delta(a), b \otimes b'),
\end{gather*}
subject to the usual sign convention
$
 (a \otimes a', b \otimes b') \coloneqq (-1)^{|a'||b|} (a,b) (a',b')$.

\begin{Remark}\label{rem:dual pairing equivalent morphisms}
Recall that the dual space of a~supercoalgebra admits a~canonical superalgebra structure. Therefore, a~bilinear map
$(\cdot,\cdot) \colon A \otimes B \to \BK$ is a~dual pairing iff the following are superalgebra morphisms:
\begin{gather}\label{eq:pairing_induced_maps}
 A \to B^{*},\qquad a \mapsto \bigl( b \mapsto (a,b) \bigr),\qquad
 B \to A^{*},\qquad b \mapsto \bigl( a \mapsto (-1)^{|a||b|}(a,b) \bigr).
\end{gather}
Alternatively, recall the \emph{restricted dual} $H^{\circ}$ of a~superbialgebra $H$, consisting of all linear functionals
$f \in H^{*}$ such that their pullback $f \circ \mu_{H}$ along the multiplication map $\mu_{H} \colon H \otimes H \to H$
lies in the subspace $H^{*} \otimes H^{*} \subseteq (H \otimes H)^{*}$ (cf.~\cite[Section~1.2.8]{ks}). By construction,
the restricted dual of a~superbialgebra is again a~superbialgebra. Thus, a~bilinear map is a~dual pairing if and only if
the induced linear map $A \to B^*$ takes values in $B^{\circ}$ and defines a~superbialgebra morphism $A \to B^{\circ}$
(or, analogously, if the induced map $B \to A^*$ defines a~superbialgebra morphism~${B \to A^{\circ}}$).
Since the morphism $A \to B^{\circ}$ canonically identifies with the morphisms
\begin{gather*}
 A^{\opp} \to \bigl(B^{\circ}\bigr)^{\opp} = \bigl(B^{\copp}\bigr)^{\circ},\qquad
 A^{\copp} \to \bigl(B^{\circ}\bigr)^{\copp} = \bigl(B^{\opp}\bigr)^{\circ},\\
 A^{\opp,\copp} \to \bigl(B^{\circ}\bigr)^{\opp,\copp} = \bigl(B^{\opp,\copp}\bigr)^{\circ},
\end{gather*}
the initial pairing naturally induces dual pairings between the following superbialgebras:
\begin{gather*}
 A^{\opp} \otimes B^{\copp} \to \BK,\qquad
 A^{\copp} \otimes B^{\opp} \to \BK,\qquad
 A^{\opp,\copp} \otimes B^{\opp,\copp} \to \BK.
\end{gather*}
The underlying morphism of superspaces (hence the explicit formula for the pairing) is identical in all cases.
\end{Remark}

\begin{Remark}\label{rem:pairing-degree-zero}
Since $(\cdot,\cdot)$ is a~morphism between superspaces, it is degree-preserving by definition and thus $(a,b) = 0$ if
$|a| + |b| \neq \bbar{0}$.
\end{Remark}

Since $A \otimes B$ admits a~canonical supercoalgebra structure, the space of superspace morphisms from $A \otimes B$
to $\BK$ can be endowed with the convolution product: the convolution $\phi \star \psi$ of morphisms $\phi$ and $\psi$
is given by
\begin{gather*}
 (\phi \star \psi)(a \otimes b)
 \coloneqq \bigl( \mu_{\BK} \circ (\phi \otimes \psi) \circ \Delta \bigr) (a \otimes b)
 = \sum_{(a)(b)} (-1)^{|a_{2}||b_{1}|} \phi(a_{1} \otimes b_{1}) \psi(a_{2} \otimes b_{2}).
\end{gather*}
Moreover, the morphism $a \otimes b \mapsto \epsilon(a)\epsilon(b)$ is the identity element for the convolution product. A~dual pairing is called \emph{convolution invertible} if it admits a~convolution inverse. Notably, any dual
pairing $(-,-) \colon A \otimes B \to \BK$ between Hopf superalgebras is convolution invertible, with its inverse given by
\begin{gather}\label{eq:inverse-pairing}
 a \otimes b \mapsto (S(a),b) = (a,S(b)).
\end{gather}

Let $(-,-) \colon A \otimes B \to \BK$ be a~convolution invertible dual pairing between superbialgebras,
and let $(-,-)'$ denote its convolution inverse. By definition, the following identity holds:
\begin{gather}\label{eq:inverse pairing identity}
 \sum_{(a)(b)} (-1)^{|a_{2}||b_{1}|} (a_{1},b_{1}) (a_{2},b_{2})'
 = \epsilon(a)\epsilon(b)
 = \sum_{(a)(b)} (-1)^{|a_{2}||b_{1}|} (a_{1},b_{1})' (a_{2},b_{2}).
\end{gather}
We also note the following.

\begin{Lemma}\label{lem:inverse_dual_pairing}
Let $\sigma \colon A \otimes B \to \BK$ be a~convolution invertible dual pairing. Its convolution inverse~$\sigma'$
is a~dual pairing between $A^{\opp}$ and $B^{\opp}$ {\rm(}equivalently, a~dual pairing between $A^{\copp}$ and~$B^{\copp}$
by Remark~{\rm \ref{rem:dual pairing equivalent morphisms})}.
\end{Lemma}

\begin{proof}

Consider the superalgebra structure on $B^{*}$, and let $\lambda \colon \Hom(A \otimes B, \BK) \iso \Hom\bigl(A, B^{*}\bigr)$
be the canonical isomorphism of vector spaces mapping $\phi$ to $\lambda(\phi) \colon a \mapsto (b \mapsto \phi(a \otimes b))$.
We claim that $\lambda$ intertwines the convolution products. Indeed, for any $\phi, \psi \in \Hom(A \otimes B, \BK)$, we have
\begin{gather*}
 (\lambda(\phi) \star \lambda(\psi))(a)(b)
\\
\qquad= \sum_{(a)} \bigl( \lambda(\phi)(a_{1}) \cdot \lambda(\psi)(a_{2}) \bigr)(b)
 = \sum_{(a)(b)} (-1)^{|a_{2}||b_{1}|} \lambda(\phi)(a_{1})(b_{1}) \cdot \lambda(\psi)(a_{2})(b_{2}) \\
 \qquad= \sum_{(a)(b)} (-1)^{|a_{2}||b_{1}|} \phi(a_{1} \otimes b_{1}) \cdot \psi(a_{2} \otimes b_{2})
 = (\phi \star \psi)(a \otimes b) = \lambda(\phi \star \psi)(a)(b).
\end{gather*}
Since $\sigma$ is a~dual pairing, $\lambda(\sigma) \colon A \to B^{*}$ is a~superalgebra morphism
(see Remark~\ref{rem:dual pairing equivalent morphisms}). By Lemma~\ref{lem:convolution inverse anti-morphism}\,(a), its
convolution inverse $\lambda(\sigma') = \lambda(\sigma)'$ is a~superalgebra morphism from~$A^{\opp}$ to~$B^{*}$, which
establishes the first condition of~\eqref{eq:pairing_induced_maps} for $A^{\opp}$ and $B^{\opp}$.

A~completely similar argument utilizing $A^{*}$ and the canonical isomorphism $\Hom(A \otimes B, \BK) \iso \allowbreak \Hom\bigl(B, A^{*}\bigr)$
completes the proof.
\end{proof}

 A~convolution invertible dual pairing canonically induces a~double cross product algebra structure via the following
coadjoint actions.

\begin{Definition}
The \emph{right coadjoint action} $\lhd_{\Coad} \colon B \otimes A \to B$ of $A$ on $B$ is defined by
\begin{gather*}
 b \lhd_{\Coad} a = \sum_{(a)(b)} (-1)^{|a|}
 (-1)^{|b_{1}||b_{2}| + |b_{1}||b_{3}|} (a_{1}, b_{1}) (a_{2}, b_{3})' b_{2}.
\end{gather*}
Analogously, the \emph{right coadjoint action} $\lhd_{\Coad} \colon A \otimes B \to A$ of $B$ on $A$ is defined by
\begin{gather*}
 a \lhd_{\Coad} b = \sum_{(a)(b)}
 (-1)^{|a_{1}||a_{2}| + |a_{1}||a_{3}|} (a_{1}, b_{1}) (a_{3}, b_{2})' a_{2}.
\end{gather*}
\end{Definition}

These actions are structurally motivated by duality with the \emph{left adjoint actions} of Hopf algebras $A$ and $B$ on
themselves, given by
\begin{align*}
 & \rhd_{\Ad} \colon\ A \otimes A \to A,\qquad
 a \rhd_{\Ad} a' = \sum_{(a)} (-1)^{|a_{2}||a'|} a_{1} a' S(a_{2}),\\
 & \rhd_{\Ad} \colon\ B \otimes B \to B,\qquad
 b \rhd_{\Ad} b' = \sum_{(b)} (-1)^{|b_{2}||b'|} b_{1} b' S(b_{2}).
\end{align*}

\begin{Proposition}
Let $A$ and $B$ be Hopf superalgebras. For any $a,a' \in A$ and $b,b' \in B$, the following holds:
\begin{align*}
 (a',b \lhd_{\Coad} a) = (-1)^{|a|(|a'|+|b|)} (a \rhd_{\Ad} a', b) \qquad \text{and} \qquad
 (a \lhd_{\Coad} b,b') = (a, b \rhd_{\Ad} b').
\end{align*}
\end{Proposition}

\begin{proof}These identities follow from direct computation
\begin{align*}
 (-1)^{|a|(|a'|+|b|)} (a \rhd_{\Ad} a', b)
={}& \sum_{(a)} (-1)^{|a|(|a'|+|b|)} (-1)^{|a_{2}||a'|} (a_{1} a' S(a_{2}), b) \\
={}& \sum_{(a)(b)} (-1)^{|a|(|a'|+|b|)} (-1)^{|a_{2}||a'|}
 (a_{1} \otimes a' \otimes S(a_{2}), b_{1} \otimes b_{2} \otimes b_{3})\\
={}& \sum_{(a)(b)} (-1)^{|a|(|a'|+|b|)} (-1)^{|a_{2}||a'|}
 (-1)^{|b_{1}||b_{2}| + |b_{1}||b_{3}| + |b_{2}||b_{3}|} \\
 & \times (a_{1}, b_{1}) (a', b_{2}) (S(a_{2}), b_{3}) \\
={}& \sum_{(a)(b)} (-1)^{|a|} (-1)^{|b_{1}||b_{2}| + |b_{1}||b_{3}|} (a_{1}, b_{1}) (a', b_{2}) (S(a_{2}), b_{3}),
\end{align*}
and
\begin{align*}
 (a, b \rhd_{\Ad} b')
 &= \sum_{(b)} (-1)^{|b_{2}||b'|} (a, b_{1} b' S(b_{2}))
 = \sum_{(a)(b)} (-1)^{|b_{2}||b'|} (a_{1} \otimes a_{2} \otimes a_{3}, b_{1} \otimes b' \otimes S(b_{2})) \\
 &= \sum_{(a)(b)} (-1)^{|b_{2}||b'|}
 (-1)^{|a_{1}||a_{2}| + |a_{1}||a_{3}| + |a_{2}||a_{3}|} (a_{1}, b_{1}) (a_{2}, b') (a_{3}, S(b_{2})) \\
 &= \sum_{(a)(b)} (-1)^{|a_{1}||a_{2}| + |a_{1}||a_{3}|} (a_{1}, b_{1}) (a_{2}, b') (a_{3}, S(b_{2})). \tag*{\qed}
\end{align*}
\renewcommand{\qed}{}
\end{proof}

\begin{Remark}
The right coadjoint action $\lhd_{\Coad} \colon A \otimes B \to A$ of $B$ on $A$ induces a~\emph{left coadjoint action}
$\rhd_{\Coad} \colon B^{\opp} \otimes A \to A$ of $B^{\opp}$ on $A$ via
\begin{align*}
 b \rhd_{\Coad} a = (-1)^{|a||b|} a \lhd_{\Coad} b =
 \sum_{(a)(b)} (-1)^{|a||b|} (-1)^{|a_{1}||a_{2}| + |a_{1}||a_{3}|} (a_{1}, b_{1}) (a_{3}, b_{2})' a_{2}.
\end{align*}
\end{Remark}

Since the supercoalgebra structures of $B$ and $B^{\opp}$ are the same, we note that $\lhd_{\Coad} \colon B \otimes A \to B$ can be
thought of as a~right coadjoint action of $A$ on $B^{\opp}$ as well. The following proposition ensures that the coadjoint actions
are compatible in the sense of a~double cross product.

\begin{Proposition}
The morphisms $\lhd_{\Coad} \colon B^{\opp} \otimes A \to B^{\opp}$ and $\rhd_{\Coad} \colon B^{\opp} \otimes A \to A$
satisfy the conditions in Proposition~{\rm\ref{prop:matched-pair-super}}.
\end{Proposition}

\begin{proof}
We first verify that $\rhd_{\Coad}$ defines a~left $B^{\opp}$-module structure on $A$
\begin{gather*}
 b \rhd_{\Coad} (b' \rhd_{\Coad} a)\\
 \qquad
= b \rhd_{\Coad} \sum_{(a)(b')} (-1)^{|a||b'|} (-1)^{|a_{1}|(|a_{2}|+|a_{3}|)} (a_{1},b'_{1}) (a_{3},b'_{2})' a_{2}\\
 \qquad
= \sum_{(a)(b)(b')} (-1)^{|a||b'|} (-1)^{|a_{1}|(|a_{2}|+|a_{3}|+|a_{4}|+|a_{5}|)} (-1)^{(|a_{2}|+|a_{3}|+|a_{4}|)|b|} (-1)^{|a_{2}|(|a_{3}|+|a_{4}|)}\\
\phantom{ \qquad
= }{}\times (a_{1},b'_{1}) (a_{5},b'_{2})' (a_{2},b_{1}) (a_{4},b_{2})' a_{3} \\
 \qquad
= \sum_{(a)(b)(b')} (-1)^{|a||b'_{2}|} (-1)^{|b'_{1}|+|b|} (-1)^{|b_{2}|(|a_{3}|+|b_{1}|)} \cdot (-1)^{|b_{1}||b'_{1}|+|b_{2}||b'_{2}|} \\
\phantom{ \qquad
= }{}\times (a_{1} \otimes a_{2},b'_{1} \otimes b_{1}) (a_{4} \otimes a_{5},b_{2} \otimes b'_{2})' a_{3} \\
 \qquad
= \sum_{(a)(b)(b')} (-1)^{|a||b'_{2}|} (-1)^{|b'_{1}|+|b|} (-1)^{|b_{2}|(|a_{2}|+|b_{1}|)} \cdot (-1)^{|b_{1}||b'|}\cdot (a_{1},(b'b)_{1}) (a_{3},(b'b)_{2})' a_{2} \\
 \qquad
= (-1)^{|b||b'|} \cdot \sum_{(a)(b)(b')} (-1)^{|a|(|b|+|b'|)} (-1)^{|a_{1}|(|a_{2}|+|a_{3}|)} (a_{1},(b'b)_{1}) (a_{3},(b'b)_{2})' a_{2} \\
 \qquad
= (-1)^{|b||b'|} (b'b) \rhd_{\Coad} a,
\end{gather*}
as well as
\begin{align*}
1 \rhd_{\Coad} a = \sum_{(a)} (-1)^{|a_{1}|(|a_{2}|+|a_{3}|)} (a_{1},1)(a_{3},1)' a_{2}
= \sum_{(a)} \epsilon(a_{1})\epsilon(a_{3}) a_{2} = a.
\end{align*}

Furthermore, this action is compatible with the supercoalgebra structure of $A$
\begin{gather*}
 \sum_{(a)(b)} (-1)^{|b_{2}||a_{1}|} (b_{1} \rhd_{\Coad} a_{1}) \otimes (b_{2} \rhd_{\Coad} a_{2})\\
\qquad= \sum_{(a)(b)} (-1)^{|b_{2}||a_{1}|}
(-1)^{|a_{1}||b_{1}|} (-1)^{|a_{1,1}|(|a_{1,2}|+|a_{1,3}|)} (-1)^{|a_{2}||b_{2}|} (-1)^{|a_{2,1}|(|a_{2,2}|+|a_{2,3}|)}\\
\phantom{ \qquad
= }{}\times (a_{1,1},b_{1,1})(a_{1,3},b_{1,2})'(a_{2,1},b_{2,1})(a_{2,3},b_{2,2})' a_{1,2} \otimes a_{2,2}\\
\qquad= \sum_{(a)(b)} (-1)^{|a||b|}
(-1)^{|a_{2,1}||a_{1,3}|} (-1)^{|a_{1,1}|(|a_{1,2}|+|a_{1,3}|+|a_{2,1}|)}\\
\phantom{ \qquad
= }{}\times (-1)^{(|a_{1,1}|+|a_{1,3}|+|a_{2,1}|)(|a_{2,2}|+|a_{2,3}|)}\\
\phantom{ \qquad
= }{}\times (a_{1,1},b_{1,1})(a_{1,3},b_{1,2})'(a_{2,1},b_{2,1})(a_{2,3},b_{2,2})' a_{1,2} \otimes a_{2,2}\\
\qquad\!\!\!\overset{\eqref{eq:inverse pairing identity}}{=} \sum_{(a)(b)} (-1)^{|a||b|}
(-1)^{|a_{1}||a_{2}|} (-1)^{|a_{1}|(|a_{4}|+|a_{5}|)}
\cdot \epsilon(a_{3})\epsilon(b_{2}) (a_{1},b_{1})(a_{5},b_{3})' a_{2} \otimes a_{4}\\
\qquad= \sum_{(a)(b)} (-1)^{|a||b|}
(-1)^{|a_{1}|(|a_{2}|+|a_{3}|+|a_{4}|)}
\cdot (a_{1},b_{1})(a_{4},b_{2})' a_{2} \otimes a_{3}
= \Delta(b \rhd_{\Coad} a),
\end{gather*}
as well as
\begin{align*}
\epsilon(b \rhd_{\Coad} a)
&= \sum_{(a)(b)} (-1)^{|a||b|} (-1)^{|a_{1}|(|a_{2}|+|a_{3}|)} (a_{1},b_{1}) (a_{3},b_{2})' \epsilon(a_{2})\\
&= \sum_{(a)(b)} (-1)^{|a||b|} (-1)^{|a_{1}||a_{2}|} (a_{1},b_{1}) (a_{2},b_{2})' \overset{\eqref{eq:inverse pairing identity}}{=} \epsilon(a)\epsilon(b).
\end{align*}

Analogous computations confirm that $\lhd_{\Coad}$ defines a~right $A$-module supercoalgebra structure on $B^{\opp}$.
Since the verification of the remaining compatibility conditions from Proposition~\ref{prop:matched-pair-super} is
purely computational, we explicitly exhibit only the relation governing the left action on a~product
\begin{gather*}
\sum_{(a)(b)} (-1)^{|a_{1}||b_{2}|} (b_{1} \rhd_{\Coad} a_{1}) \cdot ((b_{2} \lhd_{\Coad} a_{2}) \rhd_{\Coad} a')\\
\qquad= \sum_{(a)(b)} (-1)^{|a_{1}||b_{2}|} (-1)^{|a_{1}||b_{1}|} (-1)^{|a_{1,1}|(|a_{1,2}|+|a_{1,3}|)} (-1)^{|a_{2}|} (-1)^{|b_{2,1}|(|b_{2,2}|+|b_{2,3}|)}\\
\phantom{\qquad=}{}\times (a_{1,1},b_{1,1})(a_{1,3},b_{1,2})'(a_{2,1},b_{2,1})(a_{2,2},b_{2,3})' a_{1,2} \cdot (b_{2,2} \rhd_{\Coad} a')\\
\qquad= \sum_{(a)(a')(b)} (-1)^{(|a|-|a_{2,2}|)|b|} (-1)^{|a_{1,1}|(|a_{1,2}|+|a_{1,3}|+|a_{2,1}|)} (-1)^{|a_{2,2}|} (-1)^{|a_{1,3}||a_{2,1}|} \\
\phantom{\qquad=}{}\times (-1)^{|a'|(|b_{2,2}|+|b_{2,3}|)+|a'_{1}|(|a'_{2}|+|a'_{3}|)} \\
\phantom{\qquad=}{}\times (a_{1,1},b_{1,1})(a_{1,3},b_{1,2})'(a_{2,1},b_{2,1})(a_{2,2},b_{2,4})'(a'_{1},b_{2,2})(a'_{3},b_{2,3})' a_{1,2}a'_{2} \\
\qquad\!\!\!\overset{\eqref{eq:inverse pairing identity}}{=} \sum_{(a)(a')(b)} (-1)^{(|a|-|a_{4}|)|b|} (-1)^{|a_{1}|(|a_{2}|+|a_{3}|)} (-1)^{|a_{4}|}
(-1)^{|a'|(|b_{3}|+|b_{4}|)+|a'_{1}|(|a'_{2}|+|a'_{3}|)}\\
\phantom{\qquad=}{}\times \epsilon(a_{3})\epsilon(b_{2}) (a_{1},b_{1})(a_{4},b_{5})'(a'_{1},b_{3})(a'_{3},b_{4})' a_{2}a'_{2}\\
\qquad= \sum_{(a)(a')(b)} (-1)^{(|a|+|a'|)|b|} (-1)^{|a_{3}||a'_{3}|} (-1)^{(|a_{1}|+|a'_{1}|)(|a_{2}|+|a'_{2}|+|a_{3}|+|a'_{3}|)} \\
\phantom{\qquad=}{}\times (-1)^{|a_{2}||a'_{1}|+|a_{3}||a'_{1}|+|a_{3}||a'_{2}|}(a_{1} \otimes a'_{1},b_{1} \otimes b_{2})(a'_{3} \otimes a_{3},b_{3} \otimes b_{4})' a_{2}a'_{2} \\
\qquad= \sum_{(a)(a')(b)} (-1)^{(|a|+|a'|)|b|} (-1)^{(|a_{1}|+|a'_{1}|)(|a_{2}|+|a'_{2}|+|a_{3}|+|a'_{3}|)}
(-1)^{|a_{2}||a'_{1}|+|a_{3}||a'_{1}|+|a_{3}||a'_{2}|} \\
\phantom{\qquad=}{}\times (a_{1}a'_{1},b_{1})(a_{3}a'_{3},b_{2})' a_{2}a'_{2}\\
\qquad= b \rhd_{\Coad} (aa'). \tag*{\qed}
\end{gather*}
\renewcommand{\qed}{}
\end{proof}

Combining the above result with Proposition~\ref{prop:matched-pair-super}, one obtains the double cross product $A \bowtie B^{\opp}$.

\begin{Proposition}
The double cross product $A \bowtie B^{\opp}$ corresponding to the convolution invertible dual pairing has
the cross-multiplication formula
\begin{gather}\label{eq:double cross formula 1}
 (1 \otimes b)(a \otimes 1) = \sum_{(a)(b)}
 (-1)^{(|a|-|a_{3}|)(|b|-|b_{1}|)} (-1)^{|b_{1}|} (-1)^{|a_{3}|} (a_{1}, b_{1}) (a_{3}, b_{3})' a_{2} \otimes b_{2}.
\end{gather}
\end{Proposition}

\begin{proof}
Direct computation yields the following:
\begin{gather*}
 (1 \otimes b)(a \otimes 1)\\
\qquad = \sum_{(a)(b)} (-1)^{|b_{2}||a_{1}|} (b_{1} \rhd a_{1}) \otimes (b_{2} \lhd a_{2}) \\
\qquad= \sum_{(a)(b)} (-1)^{|b_{2}||a_{1}|}
 \bigl\{ (-1)^{|a_{1}||b_{1}|} (-1)^{|a_{1,1}||a_{1,2}| + |a_{1,1}||a_{1,3}|}
 (a_{1,1}, b_{1,1}) (a_{1,3}, b_{1,2})' a_{1,2} \bigr\} \\
\phantom{\qquad=}{}\otimes
 \bigl\{ (-1)^{|a_{2}|} (-1)^{|b_{2,1}||b_{2,2}| + |b_{2,1}||b_{2,3}|}
 (a_{2,1}, b_{2,1}) (a_{2,2}, b_{2,3})' b_{2,2} \bigr\}.
\end{gather*}
Expanding and re-indexing the terms, we have
\begin{gather*}
 (1 \otimes b)(a \otimes 1) \\
 \quad= \sum_{(a)(b)} (-1)^{(|b_{3}|+|b_{4}|+|b_{5}|)(|a_{1}|+|a_{2}|+|a_{3}|)}
 (-1)^{(|a_{1}|+|a_{2}|+|a_{3}|)(|b_{1}|+|b_{2}|)} (-1)^{|a_{1}||a_{2}| + |a_{1}||a_{3}|} \\
\phantom{\qquad=}{} \cdot (-1)^{|a_{4}|+|a_{5}|} (-1)^{|b_{3}||b_{4}| + |b_{3}||b_{5}|}
 (a_{1}, b_{1}) (a_{3}, b_{2})' (a_{4}, b_{3}) (a_{5}, b_{5})' a_{2} \otimes b_{4} \\
\quad\overset{(\sharp)}{=} \!\sum_{(a)(b)}\!
 (-1)^{(|a|-|a_{5}|)(|b|-|b_{1}|)} (-1)^{|b_{1}|} (-1)^{|a_{5}|}
 (-1)^{|b_{2}||a_{4}|} (a_{1}, b_{1}) (a_{3}, b_{2})' (a_{4}, b_{3}) (a_{5}, b_{5})' a_{2}
\otimes b_{4} \\
 \quad\!\!\!\overset{\eqref{eq:inverse pairing identity}}{=} \sum_{(a)(b)}
 (-1)^{(|a|-|a_{4}|)(|b|-|b_{1}|)} (-1)^{|b_{1}|} (-1)^{|a_{4}|}
 (a_{1}, b_{1}) \epsilon(a_{3}) \epsilon(b_{2}) (a_{4}, b_{4})' a_{2} \otimes b_{3} \\
 \quad= \sum_{(a)(b)}
 (-1)^{(|a|-|a_{3}|)(|b|-|b_{1}|)} (-1)^{|b_{1}|} (-1)^{|a_{3}|}
 (a_{1}, b_{1}) (a_{3}, b_{3})' a_{2} \otimes b_{2},
\end{gather*}
where the sign identity $(\sharp)$ follows from
\begin{gather*}
(-1)^{(|a_{1}|+|a_{2}|+|a_{3}|)(|b_{1}|+|b_{2}|+|b_{3}|+|b_{4}|+|b_{5}|)} (-1)^{|a_{1}|(|a_{2}| + |a_{3}|)}
(-1)^{|a_{4}|+|a_{5}|}\\
\qquad\times (-1)^{|b_{3}|(|b_{4}| + |b_{5}|)}
(-1)^{|b_{2}||a_{4}|}\\
\qquad\phantom{\cdot}{}= (-1)^{(|a_{1}|+|a_{2}|+|a_{3}|)|b|} (-1)^{|a_{1}|(|a_{2}| + |a_{3}|)}
(-1)^{|a_{4}|+|a_{5}|} (-1)^{|b_{3}|(|b_{4}| + |b_{5}|)}
(-1)^{|b_{2}||a_{4}|}\\
\qquad\phantom{\cdot}{}= (-1)^{(|a_{1}|+|a_{2}|+|a_{3}|)(|b|-|b_{1}|)}
(-1)^{|a_{1}||b_{1}|}
(-1)^{|a_{4}|+|a_{5}|} (-1)^{|b_{3}|(|b_{4}| + |b_{5}|)}
(-1)^{|b_{2}||a_{4}|}\\
\qquad\phantom{\cdot}{}= (-1)^{(|a_{1}|+|a_{2}|+|a_{3}|)(|b|-|b_{1}|)}
(-1)^{|b_{1}|}
(-1)^{|a_{5}|} (-1)^{|b_{3}|(|b|-|b_{1}|)} \\
\qquad\phantom{\cdot}{}= (-1)^{(|a|-|a_{5}|)(|b|-|b_{1}|)}
(-1)^{|b_{1}|}
(-1)^{|a_{5}|}.\tag*{\qed}
\underline{}\end{gather*}
\renewcommand{\qed}{}
\end{proof}

The next result provides an equivalent definition of $A \bowtie B^{\opp}$ without using the convolution inverse.

\begin{Proposition}
The cross-multiplication admits the alternative form
\begin{gather}\label{eq:double cross formula 2}
 \sum_{(a)(b)} (-1)^{|b_{2}||a|} (1 \otimes b_{1})(a_{1} \otimes 1) (a_{2},b_{2}) =
 \sum_{(a)(b)} (-1)^{|b_{2}||a|} (-1)^{|a_{1}||b_{1}|} (a_{1},b_{1}) a_{2} \otimes b_{2}.
\end{gather}
\end{Proposition}

\begin{proof}
Substituting~\eqref{eq:double cross formula 1} into the left-hand side, we obtain
\begin{gather*}
 \sum_{(a)(b)} (-1)^{|b_{2}||a|} (1 \otimes b_{1})(a_{1} \otimes 1) (a_{2},b_{2}) \\
 \qquad= \sum_{(a)(b)} (-1)^{|b_{4}||a|}
 (-1)^{(|b_{2}|+|b_{3}|)(|a_{1}|+|a_{2}|)} (-1)^{|b_{1}|+|a_{3}|}
 (a_{1},b_{1})(a_{3},b_{3})'a_{2} \otimes b_{2} (a_{4},b_{4}) \\
 \qquad= \sum_{(a)(b)} (-1)^{(|a_{3}|+|a_{4}|)|a|}
 (-1)^{|b_{2}|(|a_{1}|+|a_{2}|)} (-1)^{|b_{1}|} (-1)^{|a_{3}||a_{4}|}
 (a_{1},b_{1})(a_{3},b_{3})'(a_{4},b_{4}) a_{2} \otimes b_{2} \\
 \qquad= \sum_{(a)(b)} (-1)^{|a_{3}||a|}
 (-1)^{|b_{2}|(|a_{1}|+|a_{2}|)} (-1)^{|b_{1}|}
 (a_{1},b_{1}) \epsilon(a_{3}) \epsilon(b_{3}) a_{2} \otimes b_{2} \\
 \qquad= \sum_{(a)(b)} (-1)^{|b_{2}||a|} (-1)^{|b_{1}|} (a_{1},b_{1}) a_{2} \otimes b_{2}
 = \sum_{(a)(b)} (-1)^{|b_{2}||a|} (-1)^{|a_{1}||b_{1}|} (a_{1},b_{1}) a_{2} \otimes b_{2}.
\end{gather*}
Conversely, \eqref{eq:double cross formula 1} is recovered from~\eqref{eq:double cross formula 2} as follows:
\begin{gather*}
 \sum_{(a)(b)} (-1)^{(|a|-|a_{3}|)(|b|-|b_{1}|)} (-1)^{|b_{1}|} (-1)^{|a_{3}|}
 (a_{1}, b_{1}) (a_{3}, b_{3})' a_{2} \otimes b_{2} \\
 \qquad= \sum_{(a)(b)} (-1)^{(|a_{1}|+|a_{2}|)|b_{2}|} (-1)^{|a_{1}||b_{1}|} (-1)^{|a||b_{3}|}
 (a_{1}, b_{1}) (a_{3}, b_{3})' a_{2} \otimes b_{2} \\
 \qquad\!\!\!\overset{\eqref{eq:double cross formula 2}}{=} \sum_{(a)(b)} (-1)^{(|a_{1}|+|a_{2}|)|b_{2}|}
 (-1)^{|a||b_{3}|} (a_{2}, b_{2}) (a_{3}, b_{3})' (1 \otimes b_{1})(a_{1} \otimes 1) \\
 \qquad= \sum_{(a)(b)} (-1)^{|a|(|b_{2}|+|b_{3}|)} (-1)^{|b_{2}||a_{3}|}
 (a_{2}, b_{2}) (a_{3}, b_{3})' (1 \otimes b_{1})(a_{1} \otimes 1) \\
\qquad= \sum_{(a)(b)} (-1)^{|a||b_{2}|} \epsilon(a_{2}) \epsilon(b_{2}) (1 \otimes b_{1})(a_{1} \otimes 1)
 = (1 \otimes b)(a \otimes 1). \tag*{\qed}
\end{gather*}
\renewcommand{\qed}{}
\end{proof}

\subsection{Generalized doubles}\label{ssec:generalized-doubles}

In contrast to the dual pairing used in the previous subsection, we now consider a~\emph{skew-pairing} between $A$ and $B$,
defined as a~bilinear map $(\cdot,\cdot) \colon A \otimes B \to \BK$ satisfying
\begin{gather}
 (1,b) = \epsilon(b), \qquad (aa', b) = (a \otimes a', \Delta(b)), \nonumber\\
 (a,1) = \epsilon(a),
\qquad (a, bb') = \bigl(\Delta^{\opp}(a), b \otimes b'\bigr),\label{eq:skew pairing structural property}
\end{gather}
subject to the standard super-sign convention $(a \otimes a', b \otimes b') \coloneqq (-1)^{|a'||b|} (a,b) (a',b')$.

\begin{Remark}
The skew-pairing property implies the following:
\begin{align*}
 (a, bb')
 &= \bigl(\Delta^{\opp}(a), b \otimes b'\bigr) = \sum_{(a)} (-1)^{|a_{1}||a_{2}|} (a_{2} \otimes a_{1}, b \otimes b') \\
 &= \sum_{(a)} (a_{2},b)(a_{1},b') = (-1)^{|b'||b|} (\Delta(a), b' \otimes b),
\end{align*}
which yields $(a, b' *_{\opp} b) = (\Delta(a), b' \otimes b)$, where $*_{\opp}$ denotes the opposite multiplication of $B$.
\end{Remark}

The remark above shows that a~skew-pairing is equivalent to a~dual pairing between $A$ and~$B^{\opp}$. Accordingly, we call a~skew-pairing convolution invertible if the corresponding dual pairing~$A \otimes B^{\opp} \to \BK$ is convolution invertible.
Consequently, for any convolution invertible skew-pairing, the construction in Appendix~\ref{ssec:coadjoint actions} yields a~double cross product $A \bowtie B$, which we denote by $\CD(A,B)$ and call the \textit{generalized double} of $A$ and $B$.
Moreover, if the skew-pairing is non-degenerate, $\CD(A,B)$ is referred to as the \textit{Drinfeld double}. Since the underlying
coalgebra structures of $B$ and $B^{\opp}$ are identical, the cross-multiplication formulas~\eqref{eq:double cross formula 1} and~\eqref{eq:double cross formula 2} remain unchanged.

\begin{Remark}
Since a~skew-pairing between $A$ and $B$ and the corresponding dual pairing between $A$ and $B^{\opp}$ coincide as morphisms
of the underlying superspaces, we use the exact same symbol, e.g., $(\cdot,\cdot)$ or $\sigma(\cdot,\cdot)$, to denote both a~skew-pairing and its corresponding dual pairing.
\end{Remark}

The following is clear from the construction.

\begin{Proposition}\label{prop:gen double}
The generalized double $\CD(A,B)$ is isomorphic to the quotient of the free product $A * B$ of the algebras $A$ and $B$ by
the relations~\eqref{eq:double cross formula 1} {\rm(}or equivalently~\eqref{eq:double cross formula 2}{\rm)}. Specifically, for each
pair of homogeneous elements $a \in A$ and $b \in B$, define
\begin{gather*}
 u_{a,b} \coloneqq \sum_{(a)(b)} (-1)^{|b_{2}||a|} \cdot
 \bigl\{ (-1)^{|a_{1}||b_{1}|} (a_{1},b_{1}) \cdot a_{2}b_{2} - (a_{2},b_{2}) \cdot b_{1}a_{1} \bigr\},
\end{gather*}
and let $\CI$ be the two-sided ideal of $A * B$ generated by these elements. Then $\CD(A,B) \simeq (A * B)/\CI$.
\end{Proposition}

The multiplicative properties of these elements are described in the next result.

\begin{Proposition}\label{prop:double gen rel}
For any $a,a' \in A$ and $b,b' \in B$, the following equalities hold
\begin{align*}
 u_{aa',b}
 &= \sum_{(a)(b)} (-1)^{|b_{2}||a|}
 (-1)^{|a_{1}|} (a_{1}, b_{1}) \cdot a_{2} u_{a',b_{2}}
 + \sum_{(a')(b)} (-1)^{|b_{2}|(|a|+|a'|)} (a'_{2}, b_{2}) \cdot u_{a,b_{1}}a'_{1},\\
 u_{a,bb'}
 &= \sum_{(a)(b')} (-1)^{|b'_{2}||a|}
 (-1)^{|b'_{1}|} (a_{1},b'_{1}) \cdot u_{a_{2},b}b'_{2}
 + \sum_{(a)(b)} (-1)^{|b_{2}|(|a|+|b'|)} (a_{2},b_{2}) \cdot b_{1}u_{a_{1},b'}.
\end{align*}
\end{Proposition}

\begin{proof}
For any $a,a' \in A$ and $b \in B$, direct expansion of the first term of $u_{aa',b}$ yields
\begin{gather*}
 \sum_{(a)(a')(b)} (-1)^{|b_{2}|(|a|+|a'|)}
 (-1)^{(|a_{1}|+|a'_{1}|)|b_{1}|} (-1)^{|a_{2}||a'_{1}|} (a_{1}a'_{1},b_{1}) \cdot a_{2}a'_{2}b_{2} \\
 \qquad=\sum_{(a)(a')(b)} (-1)^{|b_{3}|(|a|+|a'|)}
 (-1)^{|a_{1}|+|a'_{1}|} (-1)^{|a||a'_{1}|} (a_{1}, b_{1}) (a'_{1}, b_{2}) \cdot a_{2}a'_{2}b_{3} \\
 \qquad=\sum_{(a)(a')(b)} (-1)^{(|b_{2}|+|b_{3}|)|a|}
 (-1)^{|a_{1}|} (a_{1}, b_{1})a_{2} \cdot \bigl\{(-1)^{|b_{3}||a'|} (-1)^{|a'_{1}||b_{2}|} (a'_{1}, b_{2}) a'_{2}b_{3}\bigr\} \\
 \qquad=\sum_{(a)(b)} (-1)^{|b_{2}||a|}
 (-1)^{|a_{1}|} (a_{1}, b_{1}) \cdot a_{2} u_{a',b_{2}}\\
 \phantom{ \qquad=}{} + \sum_{(a)(a')(b)} (-1)^{(|b_{2}|+|b_{3}|)|a|}
 (-1)^{|a_{1}|} (a_{1}, b_{1})a_{2}
 \bigl\{(-1)^{|b_{3}||a'|} (a'_{2}, b_{3}) b_{2}a'_{1}\bigr\}.
\end{gather*}
Similarly, for the second term, we have
\begin{gather*}
 \sum_{(a)(a')(b)} (-1)^{|b_{2}|(|a|+|a'|)}
 (-1)^{|a_{2}||a'_{1}|} (a_{2}a'_{2},b_{2}) \cdot b_{1}a_{1}a'_{1} \\
 \qquad=\sum_{(a)(a')(b)} (-1)^{(|b_{2}|+|b_{3}|)(|a|+|a'|)}
 (-1)^{|a_{2}||a'|} (a_{2}, b_{2}) (a'_{2}, b_{3}) \cdot b_{1}a_{1}a'_{1} \\
 \qquad=\sum_{(a)(a')(b)} (-1)^{|b_{3}|(|a|+|a'|)} \bigl\{(-1)^{|b_{2}||a|} (a_{2}, b_{2}) b_{1}a_{1}\bigr\} \cdot
 (a'_{2}, b_{3})a'_{1} \\
 \qquad=-\sum_{(a')(b)} (-1)^{|b_{2}|(|a|+|a'|)} (a'_{2}, b_{2}) \cdot u_{a,b_{1}}a'_{1} \\
 \phantom{ \qquad=}{}+ \sum_{(a)(a')(b)} (-1)^{|b_{3}|(|a|+|a'|)} \bigl\{(-1)^{|b_{2}||a|}
 (-1)^{|a_{1}||b_{1}|} (a_{1}, b_{1}) a_{2}b_{2}\bigr\} \cdot (a'_{2}, b_{3})a'_{1}.
\end{gather*}
Subtracting these expressions recovers the identity for $u_{aa',b}$. The other identity follows from an analogous computation
\begin{gather*}
 \sum_{(a)(b)(b')} (-1)^{(|b_{2}|+|b'_{2}|)|a|}
 (-1)^{|a_{1}|(|b_{1}|+|b'_{1}|)} (-1)^{|b_{2}||b'_{1}|} (a_{1},b_{1}b'_{1}) \cdot a_{2}b_{2}b'_{2} \\
 \qquad=\sum_{(a)(b)(b')} (-1)^{(|b_{2}|+|b'_{2}|)|a|}
 (-1)^{|b_{1}|+|b'_{1}|} (-1)^{|b_{2}||b'_{1}|} (a_{1},b'_{1}) (a_{2},b_{1}) \cdot a_{3}b_{2}b'_{2} \\
 \qquad=\sum_{(a)(b)(b')} (-1)^{|b'_{2}||a|}
 (-1)^{|b'_{1}|} \bigl\{(-1)^{|b_{2}|(|a_{2}|+|a_{3}|)} (-1)^{|a_{2}||b_{1}|} (a_{2}, b_{1}) a_{3}b_{2}\bigr\} \cdot
 (a_{1},b'_{1})b'_{2} \\
 \qquad=\sum_{(a)(b')} (-1)^{|b'_{2}||a|}
 (-1)^{|b'_{1}|} (a_{1},b'_{1}) \cdot u_{a_{2},b}b'_{2} \\
\phantom{ \qquad=}{}+ \sum_{(a)(b)(b')} (-1)^{|b'_{2}||a|}
 (-1)^{|b'_{1}|} \bigl\{(-1)^{|b_{2}|(|a_{2}|+|a_{3}|)} (a_{3}, b_{2}) b_{1}a_{2}\bigr\} \cdot (a_{1},b'_{1})b'_{2},
\end{gather*}
and
\begin{gather*}
 \sum_{(a)(b)(b')} (-1)^{(|b_{2}|+|b'_{2}|)|a|}
 (-1)^{|b_{2}||b'_{1}|} (a_{2},b_{2}b'_{2}) \cdot b_{1}b'_{1}a_{1} \\
 \qquad=\sum_{(a)(b)(b')} (-1)^{(|b_{2}|+|b'_{2}|)|a|}
 (-1)^{|b_{2}||b'_{1}|} (a_{2},b'_{2}) (a_{3},b_{2}) \cdot b_{1}b'_{1}a_{1} \\
 \qquad=\sum_{(a)(b)(b')} (-1)^{|b_{2}|(|a|+|b'|)}
 (a_{3},b_{2})b_{1} \cdot \bigl\{(-1)^{|b'_{2}|(|a_{1}|+|a_{2}|)} (a_{2}, b'_{2}) b'_{1}a_{1}\bigr\} \\
 \qquad= -\sum_{(a)(b)} (-1)^{|b_{2}|(|a|+|b'|)} (a_{2},b_{2}) \cdot b_{1}u_{a_{1},b'}\\
\phantom{ \qquad=}{}
 + \sum_{(a)(b)(b')} (-1)^{|b_{2}|(|a|+|b'|)} (a_{3},b_{2})b_{1}
 \bigl\{(-1)^{|b'_{2}|(|a_{1}|+|a_{2}|)} (-1)^{|a_{1}||b'_{1}|} (a_{1}, b'_{1}) a_{2}b'_{2}\bigr\}.
 \tag*{\qed}
\end{gather*}
\renewcommand{\qed}{}
\end{proof}

\begin{Remark}\label{rem:reduced_generators}
The result above implies that if $A$ and $B$ are generated by subsets $X$ and $Y$ having well-behaved comultiplications,
then the ideal $\CI$ is generated by $u_{x,y}$ restricted to the generators $x \in X$ and $y \in Y$, rather than all $u_{a,b}$.
In all cases of interest to us, specifically~$U_{q}(\fg)$ and $U(R)$, this reduction holds, thus allowing us to define their
doubles by imposing the cross-multiplication~\eqref{eq:double cross formula 1} solely on the algebra generators.
\end{Remark}

\subsection{The opposite double construction}

Consider the generalized double $\CD(A,B)$ of superbialgebras $A$ and $B$ associated with a~convolution invertible skew-pairing.
It is often useful to formulate the double in the opposite order, namely as a~generalized double $\CD(B,A)$ with respect to an
appropriately induced skew-pairing. We shall now detail this canonical construction.

Let $\sigma \colon A \otimes B \to \BK$ be a~convolution invertible skew-pairing, viewed as a~dual pairing between~$A$ and $B^{\opp}$.
Lemma~\ref{lem:inverse_dual_pairing} implies that its convolution inverse $\sigma'$ is a~dual pairing between~$A^{\opp}$ and $B$,
hence, the composition
\begin{gather}\label{eq:sigma-tilde}
 \widetilde{\sigma} \coloneqq \sigma' \circ \tau \colon \ B \otimes A^{\opp} \to \BK
\end{gather}
defines a~dual pairing between $B$ and $A^{\opp}$ and thus a~skew-pairing between $B$ and $A$. We shall refer to $\widetilde{\sigma}$
as the \emph{transposed inverse} of $\sigma$. This induced pairing allows us to construct the double~$\CD(B,A)$. We also note that
$\widetilde{\sigma}' = \sigma \circ \tau$ by Lemma~\ref{lem:conv prod id}\,(b). To prevent any ambiguity, we shall explicitly specify the
chosen pairing in our notation for the double, denoting the resulting algebras by $\CD(A,B,\sigma)$ and $\CD(B,A,\widetilde{\sigma})$.

\begin{Proposition}\label{prop:opposite double}
There exists an isomorphism of superbialgebras $\CD(B,A,\widetilde{\sigma}) \iso \CD(A,B,\sigma)$ which is the identity on both subbialgebras $A$ and $B$.
\end{Proposition}

\begin{proof}
Since both doubles are quotients of the isomorphic free products $A * B \simeq B * A$, it suffices to show that their
respective cross-multiplication relations~\eqref{eq:double cross formula 1} generate the exact same two-sided ideal.
For simplicity, we identify the superbialgebras $A$ and $B$ with their respective images inside the doubles, thereby
omitting the tensor product symbols.

For any $a,b \in \CD(A,B,\sigma)$, a~direct computation yields (cf.\ Remark~\ref{rem:pairing-degree-zero})
\begin{gather*}
 \sum_{(a)(b)} (-1)^{(|b|-|b_{3}|)(|a|-|a_{1}|)} (-1)^{|a_{1}|}(-1)^{|b_{3}|} \cdot
 \widetilde{\sigma}(b_{1},a_{1})\widetilde{\sigma}'(b_{3},a_{3})b_{2}a_{2} \\
 \qquad= \sum_{(a)(b)} (-1)^{(|b|-|b_{5}|)(|a|-|a_{1}|)} (-1)^{|a_{1}|}(-1)^{|b_{5}|}
 (-1)^{(|a_{2}|+|a_{3}|)(|b_{3}|+|b_{4}|)} (-1)^{|b_{2}|}(-1)^{|a_{4}|} \\
\phantom{\qquad= }{}
 \times \widetilde{\sigma}(b_{1},a_{1}) \widetilde{\sigma}'(b_{5},a_{5}) \sigma(a_{2},b_{2}) \sigma'(a_{4},b_{4})a_{3}b_{3} \\
 \qquad= \sum_{(a)(b)} (-1)^{(|b|-|b_{5}|)(|a|-|a_{1}|)}
 (-1)^{(|a_{2}|+|a_{3}|)(|b_{3}|+|b_{4}|)} (-1)^{|b_{2}|+|a_{4}|} \\
\phantom{\qquad= }{}
 \times \sigma'(a_{1},b_{1}) \sigma(a_{5},b_{5}) \sigma(a_{2},b_{2}) \sigma'(a_{4},b_{4})a_{3}b_{3} \\
 \qquad\overset{(\dagger)}{=}
 \sum_{(a)(b)} (-1)^{(|b|-|b_{4}|-|b_{5}|)(|a|-|a_{1}|-|a_{2}|)}
 (-1)^{|a_{2}||b_{1}|+|a_{3}||b_{3}|+|b_{4}||a_{5}|} \\
\phantom{\qquad= }{}
 \times \sigma'(a_{1},b_{1}) \sigma(a_{2},b_{2}) \sigma'(a_{4},b_{4}) \sigma(a_{5},b_{5}) a_{3}b_{3} \\
 \qquad= \sum_{(a)(b)} (-1)^{(|b|-|b_{3}|)(|a|-|a_{1}|)} (-1)^{|a_{2}||b_{2}|}
 \cdot \epsilon(a_{1})\epsilon(b_{1}) \epsilon(a_{3})\epsilon(b_{3}) a_{2}b_{2} = ab,
\end{gather*}
where the sign identity $(\dagger)$ follows from Remark~\ref{rem:pairing-degree-zero}
\begin{gather*}
 (-1)^{(|b|-|b_{5}|)(|a|-|a_{1}|)}
 (-1)^{(|a_{2}|+|a_{3}|)(|b_{3}|+|b_{4}|)} (-1)^{|b_{2}|}(-1)^{|a_{4}|} \cdot \sigma(a_{2},b_{2}) \sigma'(a_{4},b_{4})\\
 \qquad= (-1)^{(|b|-|b_{5}|)(|a|-|a_{1}|-|a_{2}|)}
 (-1)^{|a_{2}||b_{1}|} (-1)^{|a_{3}|(|b_{3}|+|b_{4}|)} (-1)^{|a_{4}|} \cdot \sigma(a_{2},b_{2}) \sigma'(a_{4},b_{4})\\
 \qquad= (-1)^{(|b|-|b_{4}|-|b_{5}|)(|a|-|a_{1}|-|a_{2}|)}
 (-1)^{|a_{2}||b_{1}|} (-1)^{|a_{3}||b_{3}|} (-1)^{|b_{4}||a_{5}|} \cdot \sigma(a_{2},b_{2}) \sigma'(a_{4},b_{4}).
\end{gather*}

This computation demonstrates that the cross-multiplication relation defining $\CD(B,A,\widetilde{\sigma})$ is satisfied within
$\CD(A,B,\sigma)$. By a~completely analogous calculation, the cross-multiplication relation defining $\CD(A,B,\sigma)$
is satisfied within $\CD(B,A,\widetilde{\sigma})$, establishing the algebra isomorphism. Finally, because the coproducts on
both doubles are determined by the exact same formulas on the elements of $A$ and $B$, their supercoalgebra structures
automatically coincide.
\end{proof}

\subsection{Canonical element}

Suppose the skew-pairing $(\cdot,\cdot) \colon A \otimes B \to \BK$ is non-degenerate and that both $A$ and $B$ are
finite-dimensional. The non-degeneracy implies that $A$ and $B$ must have the same dimension, and we may choose bases
$\{a_{1}, \ldots, a_{N}\}$ and $\{b_{1}, \ldots, b_{N}\}$ of $A$ and $B$, respectively, dual to each other with respect
to the skew-pairing: $(a_i,b_j)=\delta_{ij}$.

\begin{Definition}
The \emph{canonical element} of $\CD(A,B)$ is defined as
\begin{gather}\label{eq:canonical element}
 \fR = \sum_{1\leq i\leq N} b_{i} \otimes a_{i} \in \CD(A,B) \otimes \CD(A,B).
\end{gather}
This construction is independent of the choice of dual bases and clearly satisfies $|\fR| = \bbar{0}$.
\end{Definition}

The basic properties of this $\fR$ are summarized in the following result.

\begin{Proposition}\label{prop:R properties}
The canonical element $\fR$ satisfies the following properties:
\begin{gather*}
 (\Delta \otimes \Id)\fR = \fR_{(13)}\fR_{(23)},\qquad
 (\Id \otimes \Delta)\fR = \fR_{(13)}\fR_{(12)},\\
 (\epsilon \otimes \Id)\fR = 1 = (\Id \otimes \epsilon)\fR,\qquad
 \fR_{(12)}\fR_{(13)}\fR_{(23)} = \fR_{(23)}\fR_{(13)}\fR_{(12)},\\
 \Delta^{\opp}(u) \fR = \fR \Delta(u) \qquad \textrm{for all} \ u \in \CD(A,B),
\end{gather*}
where the notation of Section~{\rm\ref{ssec:leg-numbering}} is used.
\end{Proposition}

\begin{proof}
We first note that for any $a \in A$ and $b \in B$, the following identities hold:
\begin{gather}\label{eq:dual-basis-expansion}
 a = \sum_{1\leq i\leq N} (a,b_{i}) a_{i},\qquad
 b = \sum_{1\leq i\leq N} (a_{i},b) b_{i}.
\end{gather}
We begin by proving $\Delta_{(2)} \fR = (\Id \otimes \Delta)\fR = \fR_{(13)}\fR_{(12)}$. Since the skew-pairing
is non-degenerate, it suffices to show that the equality holds after applying $(a_{k},-) \otimes \Id \otimes \Id$
to both sides for all~${1 \leq k \leq N}$
\begin{gather}\label{eq:R-property-1}
 \sum_{1\leq i\leq N} (a_{k},b_{i}) \cdot \Delta(a_{i})
 = \sum_{1\leq i,j\leq N} (a_{k},b_{i}b_{j}) \cdot a_{j} \otimes a_{i}.
\end{gather}
The left-hand side of~\eqref{eq:R-property-1} is simply $\Delta(a_{k})$, while the right-hand side becomes{\samepage
\begin{align*}
 \sum_{1\leq i,j\leq N} (a_{k},b_{i}b_{j}) \cdot a_{j} \otimes a_{i}
 &= \sum_{1\leq i,j\leq N} \sum_{(a_{k})} \bigl((a_{k})_{1},b_{j}\bigr) \bigl((a_{k})_{2},b_{i}\bigr) \cdot a_{j} \otimes a_{i} \\
 \overset{\eqref{eq:dual-basis-expansion}}&{=} \sum_{(a_{k})} (a_{k})_{1} \otimes (a_{k})_{2}
 = \Delta(a_{k}),
\end{align*}
thus establishing~\eqref{eq:R-property-1}.
Analogous computations prove $\Delta_{(1)}\fR = (\Delta \otimes \Id)\fR = \fR_{(13)}\fR_{(23)}$.}

The identity $(\epsilon \otimes \Id)\fR = 1$ can be verified by the following direct computation
\begin{gather*}
 (\epsilon \otimes \Id)\fR = \sum_{1 \leq i \leq N} \epsilon(b_{i}) a_{i}
 = \sum_{1 \leq i \leq N} (1,b_{i}) a_{i} \overset{\eqref{eq:dual-basis-expansion}}{=} 1,
\end{gather*}
and analogously for $(\Id \otimes \epsilon)\fR = 1$.

We now prove $\Delta^{\opp}(u) \fR = \fR \Delta(u)$ for all $u \in \CD(A,B)$. As $\{a_{i}, b_{i}\}_{i=1}^{N}$ generate $\CD(A,B)$,
it is sufficient to check the identity for $u = a_k$ and $u = b_k$. Setting $u = b_k$ and applying $(a_{l},-) \otimes \Id$ to both
sides yields
\begin{align*}
 \bigl( (a_{l},-) \otimes \Id \bigr) \bigl(\Delta^{\opp}(b_{k}) \fR\bigr)
 &= \sum_{1\leq i\leq N} \sum_{(b_{k})} (-1)^{|(b_{k})_{1}||(b_{k})_{2}|} (-1)^{|(b_{k})_{1}||b_{i}|}
 \bigl(a_{l}, (b_{k})_{2} b_{i}\bigr) \cdot (b_{k})_{1} a_{i} \\
 &= \sum_{1\leq i\leq N} \sum_{(a_{l})(b_{k})} (-1)^{|(b_{k})_{1}||a_{l}|}
 \bigl((a_{l})_{1}, b_{i}\bigr) \bigl((a_{l})_{2}, (b_{k})_{2}\bigr) \cdot (b_{k})_{1} a_{i} \\
 \overset{\eqref{eq:dual-basis-expansion}}&{=}
 (-1)^{|b_{k}||a_{l}|} \sum_{(a_{l})(b_{k})} (-1)^{|(b_{k})_{2}||a_{l}|}
 \bigl((a_{l})_{2}, (b_{k})_{2}\bigr) \cdot (b_{k})_{1} (a_{l})_{1}
\end{align*}
and
\begin{align*}
 \bigl( (a_{l},-) \otimes \Id \bigr) (\fR \Delta(b_{k}))
 &= \sum_{1\leq i\leq N} \sum_{(b_{k})} (-1)^{|(b_{k})_{1}||b_{i}|} \bigl(a_{l}, b_{i} (b_{k})_{1}\bigr) \cdot a_{i} (b_{k})_{2} \\
& = \sum_{1\leq i\leq N} \sum_{(a_{l})(b_{k})} (-1)^{|(b_{k})_{1}||(a_{l})_{2}|}
 \bigl((a_{l})_{1}, (b_{k})_{1}\bigr) \bigl((a_{l})_{2}, b_{i}\bigr) \cdot a_{i} (b_{k})_{2}\\
 \overset{\eqref{eq:dual-basis-expansion}} &{=}
 (-1)^{|b_{k}||a_{l}|} \sum_{(a_{l})(b_{k})} (-1)^{|(b_{k})_{2}||a_{l}|} (-1)^{|(b_{k})_{1}||(a_{l})_{1}|} \\
 &\quad \ \times \bigl((a_{l})_{1}, (b_{k})_{1}\bigr) \cdot (a_{l})_{2} (b_{k})_{2},
\end{align*}
which coincide by~\eqref{eq:double cross formula 2}. The verification for $u = a_k$ is analogous, concluding the proof of
$\Delta^{\opp}(u) \fR = \fR \Delta(u)$. Lastly, the Yang--Baxter equation
$\fR_{(12)}\fR_{(13)}\fR_{(23)} = \fR_{(23)}\fR_{(13)}\fR_{(12)}$ follows directly from above identities
\begin{align*}
 \fR_{(12)}\fR_{(13)}\fR_{(23)} &= \fR_{(12)} \Delta_{(1)}(\fR) = \Delta^{\opp}_{(1)}(\fR) \fR_{(12)} \\
 &= (\tau_{(12)} \Delta_{(1)})(\fR) \fR_{(12)} = \tau_{(12)} (\fR_{(13)}\fR_{(23)}) \fR_{(12)}
 = \fR_{(23)}\fR_{(13)}\fR_{(12)}. \tag*{\qed}
\end{align*}
\renewcommand{\qed}{}
\end{proof}

Furthermore, the existence of antipodes yields explicit formulas for the inverse of the canonical element.

\begin{Proposition}\label{prop:R-invertible}
If $A$ is a~Hopf superalgebra with an invertible antipode $S$ {\rm(}thus $A^{\opp}$ and $A^{\copp}$ are also Hopf superalgebras
with antipode~$S^{-1}${\rm)}, then the canonical element $\fR$ is invertible, with its inverse given by
\[
\fR^{-1} = \sum_{1\leq i\leq N} b_{i} \otimes S^{-1}(a_{i}).
\]
Likewise, if $B$ is a~Hopf superalgebra with antipode $S$, then $\fR$ is invertible, with its inverse given~by
\[
\fR^{-1} = \sum_{1\leq i\leq N} S(b_{i}) \otimes a_{i}.
\]
\end{Proposition}

\begin{proof}
Assume that $A$ is a~Hopf superalgebra with an invertible antipode. For any $1 \leq k \leq N$, we have
\begin{gather*}
 ((a_{k},-) \otimes \Id) \Big( \fR \sum_{1\leq j\leq N} b_{j} \otimes S^{-1}(a_{j}) \Big)\\
 \qquad= \sum_{1\leq i,j\leq N} (-1)^{|a_{i}||a_{j}|} (a_{k}, b_{i}b_{j}) \cdot a_{i} S^{-1}(a_{j})\\
 \qquad= \sum_{1\leq i,j\leq N} \sum_{(a_{k})} (-1)^{|(a_{k})_{1}||(a_{k})_{2}|}
 ((a_{k})_{1}, b_{j}) ((a_{k})_{2}, b_{i}) \cdot a_{i} S^{-1}(a_{j})\\
 \quad\ \!\overset{\eqref{eq:dual-basis-expansion}}{=}
 \sum_{(a_{k})} (-1)^{|(a_{k})_{1}||(a_{k})_{2}|} \cdot (a_{k})_{2} S^{-1}((a_{k})_{1})
 = \epsilon(a_{k}),
\end{gather*}
which implies that $\fR \bigl( \sum_{j=1}^{N} b_{j} \otimes S^{-1}(a_{j}) \bigr) = 1$. The verification of
$\bigl( \sum_{j=1}^{N} b_{j} \otimes S^{-1}(a_{j}) \bigr) \fR = 1$ and the analogous statements for $B$ proceed similarly,
completing the proof.
\end{proof}

Assume the hypotheses of Proposition~\ref{prop:R-invertible} hold. For notational convenience, let us denote the skew-pairing
by $\sigma \colon A \otimes B \to \BK$ while using $\widetilde{\sigma} \colon B \otimes A \to \BK$ for its transposed inverse. Invoking~${\CD(B,A,\widetilde{\sigma}) \simeq \CD(A,B,\sigma)}$ from Proposition~\ref{prop:opposite double}, let us now relate their canonical elements.

\begin{Corollary}
Under the assumptions of Proposition~{\rm \ref{prop:R-invertible}}, the canonical element of superbialgebra $\CD(B,A,\widetilde{\sigma})$ with respect
to $\widetilde{\sigma}$ is $\fR^{-1}_{(21)}$.
\end{Corollary}

\begin{proof}
First, consider the case when $A$ is a~Hopf superalgebra with an invertible antipode $S$. The evaluation
\begin{gather*}
 \widetilde{\sigma}\bigl(b_{j}, S^{-1}(a_{i})\bigr) \overset{\eqref{eq:inverse-pairing},\, \eqref{eq:sigma-tilde}}{=}
 (-1)^{|a_{i}||b_{j}|} \sigma(a_{i}, b_{j}) = (-1)^{|a_{i}||b_{j}|} \delta_{ij}
\end{gather*}
implies that $\{b_{j}\}_{j=1}^{N}$ and \smash{$\bigl\{(-1)^{|a_{i}||b_{i}|}S^{-1}(a_{i}) \bigr\}_{i=1}^{N}$} form dual bases with respect
to $\widetilde{\sigma}$. The result then follows from Proposition~\ref{prop:R-invertible}.

The case where $B$ is a~Hopf superalgebra is analogous.
\end{proof}

\begin{Remark}
The results of this subsection naturally generalize to any infinite-dimensional setting where the
summation~\eqref{eq:canonical element} is well-defined, such as the $\hbar$-adic quantum supergroups
(cf.\ Remark~\ref{rem:h-adic-completion-avoidance}).
\end{Remark}

\subsection*{Acknowledgements}

K.H.\ is grateful to A.~Uruno for discussions on the orthosymplectic Lie superalgebras. A.T.\ is indebted to A.~Negu\c{t}
for enlightening discussions over the years, specifically, for emphasizing the importance of skew pairings and doubles.
We are also grateful to H.~Zhang for the correspondence on~\cite{hz} and to the referees for their useful
suggestions that improved the exposition. A.T.\ is grateful to MPIM Bonn, where the final version of this paper was prepared in July 2026, for the hospitality and wonderful working conditions.
Both authors were partially supported by the NSF Grant DMS-$2302661$. A.T.\ also acknowledges support by Simons Fellowship MPS-SFM-00020373.

\pdfbookmark[1]{References}{ref}
\LastPageEnding

\end{document}